\documentclass{amsart}
\usepackage{amsmath, amssymb, amscd, amsthm, amsfonts, amstext,
  amsbsy, enumerate, mathrsfs,  upgreek, hyperref, mathtools, stmaryrd}

\newcommand\V{\mathsf{v}}
\newcommand\bB{\boldsymbol{{\rm B}}}
\newcommand\bN{\boldsymbol{{\rm N}}}
\renewcommand{\L}{\mathcal{L}}
\newcommand{\sL}{\mathscr{L}}
\newcommand{\sS}{\mathscr{S}}

\newcommand{\T}{\mathcal{T}}

\newcommand{\ZZ}{\mathbb{Z}}

\newcommand{\Hes}{\mathcal{H}}
\newcommand{\pre}[2]{{}^{#1} #2}
\newcommand{\seq}[2]{\langle #1 \mid #2 \rangle}
\newcommand{\set}[2]{\{ #1 \mid #2 \}}

\newcommand{\cof}{\operatorname{cof}}
\newcommand{\proj}{\operatorname{p}}
\newcommand{\pow}{\mathscr{P}}

\newcommand{\Mod}{\operatorname{Mod}}
\newcommand{\range}{\operatorname{Range}}

\newcommand{\pred}{\operatorname{Pred}}
\newcommand{\cone}{\operatorname{Cone}}
\newcommand{\SUCC}{\operatorname{Succ}}
\newcommand{\leng}[1]{\operatorname{length}(#1)}

\newcommand{\On}{{\sf On}}

\newtheorem{theorem}{Theorem}[section]
\newtheorem{lemma}[theorem]{Lemma}
\newtheorem{corollary}[theorem]{Corollary}
\newtheorem{proposition}[theorem]{Proposition}

\newtheorem{question}[theorem]{Question}

\theoremstyle{definition}
\newtheorem{claim}{Claim}[theorem]
\newtheorem{defin}[theorem]{Definition}
\newtheorem{definition}[theorem]{Definition}
\newtheorem{fact}[theorem]{Fact}
\newtheorem{example}[theorem]{Example}
\theoremstyle{remark}
\newtheorem{remark}[theorem]{Remark}

\begin{document}

\title[Embeddability relation on models of large size]{The descriptive set-theoretical complexity \\ of the embeddability relation \\ on models of large size}
\date{\today}
\author{Luca Motto Ros}
\address{Albert-Ludwigs-Universit\"at Freiburg \\
Mathematisches Institut -- Abteilung f\"ur Mathematische Logik\\
Eckerstra{\ss}e, 1 \\
 D-79104 Freiburg im Breisgau\\
Germany}
\email{luca.motto.ros@math.uni-freiburg.de}
\thanks{The author would like to thank very much M.~D\v{z}amonja, P.~L\"ucke and P.~Schlicht for many useful discussions and suggestions concerning the results of this paper.}

\begin{abstract}
We show that if \( \kappa \) is a weakly compact cardinal then the 
embeddability relation on (generalized) trees of size \( \kappa \) is invariantly 
universal. This means that for every analytic quasi-order \( R \) on the generalized 
Cantor space \( \pre{\kappa}{2} \) there is an \( \L_{\kappa^+ \kappa} \)-sentence \( \upvarphi \) such that the embeddability relation on its models of 
size \( \kappa \), which are all trees, is Borel bi-reducible (and, in fact, classwise Borel 
isomorphic) to \( R \). In particular, this implies that the relation of embeddability on trees of size \( \kappa \) is complete for analytic quasi-orders on \( \pre{\kappa}{2} \). These facts generalize analogous results for \( \kappa = 
\omega \) obtained in~\cite{louros,frimot}, and it also partially extends a result 
from~\cite{bau} concerning the structure of the embeddability relation on 
linear orders of size \( \kappa \).  
\end{abstract}

\maketitle

\section{Introduction}

The aim of this paper is to establish a connection between descriptive set theory and (basic) model theory of uncountable models. In particular, we want to analyze the complexity of the embeddability relation on various classes of structures using typical methods of descriptive set theory, namely definable reducibility between  quasi-orders and equivalence relations.

The embeddability relation, denoted in this paper by \( \sqsubseteq \), is an 
important notion in model theory, but has also been widely considered in set 
theory. For example, in a long series of papers (see e.g.\ \cite{sheintro,mek,kojshe,dzashe,tho} and the references contained therein), it was determined for various 
cardinals \( \kappa \) whether there is a \emph{universal} graph of size \( 
\kappa \) (i.e.\ a graph such that all other graphs of size \( \kappa \) embed
into it) and, in the negative case, the possible size of a minimal 
\emph{universal family}, i.e.\ of a family \( \mathcal{D} \) of graphs of size 
\( \kappa \) with the property that for every other graph \( G \) of size \( 
\kappa \) there is \( H \in \mathcal{D} \) such that \( G \sqsubseteq H \). 
Another interesting example is contained in the paper~\cite{bau}, 
where Baumgartner shows that the embeddability relation on linear orders of size 
a regular cardinal \( \kappa \) is extremely rich and complicated (see 
Remark~\ref{rembaumgartner}), a fact that should be contrasted with the celebrated 
Laver's proof~\cite{lav} of the Fra\"iss\'e conjecture, which states that the embeddability relation 
on countable linear orders is a wqo.

Fix an infinite cardinal \( \kappa \). Starting from the mentioned result from~\cite{bau}, in this work we 
will compare the complexity of the embeddability relation on various 
\emph{elementary classes} of models, i.e.\ on the classes \( 
\Mod^\kappa_\upvarphi \) of models of size \( \kappa \) of various \( 
\L_{\kappa^+ \kappa} \)-sentences \( \upvarphi \). A standard way to achieve 
this goal is to say that the embeddability relation \( \sqsubseteq \restriction 
\Mod^\kappa_\upvarphi \)  is no more complicated than the relation \( 
\sqsubseteq \restriction \Mod^\kappa_\uppsi \) (where \( \upvarphi \) and \( 
\uppsi \) are two \( \L_{\kappa^+ \kappa} \)-sentences) exactly when there is 
a ``simply definable'' reduction between \( \sqsubseteq \restriction 
\Mod^\kappa_\upvarphi \) and \( \sqsubseteq \restriction 
\Mod^\kappa_\uppsi \). This idea is precisely formalized in Definition~\ref{defborelreducibility} with the notion of \emph{Borel reducibility} \( \leq_B 
\) (and of the induced equivalence relation of \emph{Borel bi-reducibility}  \( 
\sim_B \)) between analytic quasi-orders.\footnote{See Definition~\ref{defanalyticqo}.}
This notion of reducibility was first introduced in~\cite{frista} and~\cite{harkeclou} for the case \( \kappa  = \omega \). Our generalization to uncountable cardinals \( \kappa \) was independently  introduced also in~\cite{frihytkul}, 
where (among many other results)  the complexity in terms of Shelah's 
stability theory of two first order theories \( T,T' \) is related to the relative 
complexity under \( \leq_B \) of the corresponding isomorphism relations \( 
\cong \restriction \Mod^\kappa_T \) and \( \cong \restriction 
\Mod^\kappa_{T'} \) (for suitable uncountable cardinals \( \kappa \)).

The main result of this paper is the following.

\begin{theorem}\label{theorintro}
Let \( \kappa \) be a weakly compact cardinal.\footnote{See Definition~\ref{defweaklycompact}.} 
The embeddability relation on 
(generalized) trees of size \( \kappa \) is \emph{(strongly) invariantly 
universal},\footnote{See Definitions~\ref{definvariantlyuniversal} and~\ref{defstronginvariantlyuniversal}.} 
i.e.\ for every analytic quasi-order \( R \) on a standard Borel \( \kappa \)-space\footnote{See Definition~\ref{defstandard}.}  
there is an \( \L_{\kappa^+ \kappa} \)-sentence \( \upvarphi \) all of whose 
models are trees such that \( R \) is Borel bi-reducible with (and, in fact, even 
classwise Borel isomorphic\footnote{See Definition~\ref{defBorelisomorphic}.} to) the embeddability relation \( \sqsubseteq \restriction 
\Mod^\kappa_\upvarphi \). 
\end{theorem}

Notice that since every relation of the form \( 
\sqsubseteq \restriction \Mod^\kappa_\upvarphi \) is an analytic quasi-order, 
Theorem~\ref{theorintro} actually yields a characterization of the class of analytic quasi-orders.
\begin{corollary}
Let \( \kappa \) be a weakly compact cardinal. A binary relation \( R \) on a 
standard Borel \( \kappa \)-space
is an analytic quasi-order if and only if there is an \( \L_{\kappa^+ \kappa} \)-sentence \( 
\upvarphi \) such that \( R \sim_B {\sqsubseteq \restriction 
\Mod^\kappa_\upvarphi} \). 
\end{corollary}

Moreover, Theorem~\ref{theorintro} obviously 
yields an analogous result for analytic equivalence relations, namely that the 
bi-embeddability 
relation on trees of size \( \kappa \) is (strongly) invariantly universal for the 
class of analytic equivalence relations on standard Borel \( \kappa \)-spaces.

Theorem~\ref{theorintro} can be na\"ively interpreted as saying that the 
embeddability relations (on elementary classes) are ubiquitous in the realm of 
analytic quasi-orders, and that given any ``complexity'' for an analytic 
quasi-order there is always an elementary class \( \Mod^\kappa_\upvarphi \) such 
that \( \sqsubseteq \restriction \Mod^\kappa_\upvarphi \) has exactly that 
complexity. So, in particular, there are elementary classes such that the 
corresponding embeddability relation is very simple (e.g.\ a linear order, a 
nonlinear bqo, an equivalence relation with any permitted\footnote{See e.g.~\cite{shevai2}.} number of classes, 
and so on), and other elementary classes giving rise to a very complicated 
embeddability relation.

Moreover, since Theorem~\ref{theorintro} establishes an exact correspondence between the structure 
of the embeddability relations (on elementary classes) under \( \leq_B \) and 
the structure of analytic quasi-orders under  \( \leq_B \), any result concerning 
one of these two structures can be automatically transferred to the other one. 
For example, in \cite[Theorem 53]{frihytkul} it is shown that there are models 
of \( \mathsf{GCH} \) where for any regular uncountable cardinal \( \kappa \) 
the partial order \( (\pow(\kappa), \subseteq) \) embeds into the structure 
consisting of analytic equivalence relations on \( \pre{\kappa}{2} \) under \( 
\leq_B \). This result can be 
automatically translated in our context by saying that there is a model of \( 
\mathsf{GCH} \) in which if \( \kappa \) is a weakly compact cardinal then \( 
(\pow(\kappa), \subseteq ) \) embeds into the structure of the 
(bi-)embeddability relations under \( \leq_B \), i.e.\ that there is a map \( f 
\colon \pow(\kappa) \to \L_{\kappa^+ \kappa} \) such that \( {X \subseteq 
Y} \iff {{\equiv \restriction \Mod^\kappa_{f(X)}} \leq_B {\equiv \restriction 
\Mod^\kappa_{f(Y)}}} \iff {{\sqsubseteq \restriction \Mod^\kappa_{f(X)}} 
\leq_B {\sqsubseteq \restriction \Mod^\kappa_{f(Y)}}}   \), where \( \equiv \) 
denotes the relation of bi-embeddability and \( X, Y \subseteq \kappa \). This 
implies that in such model the structure of the (bi-)embeddability relations 
under \( \leq_B \) is quite rich and complicated, as it includes e.g.\ long 
antichains and long descending chains.

Theorem~\ref{theorintro} generalizes an analogous result from~\cite{frimot} 
dealing with countable models and analytic quasi-orders on the Cantor space 
\( \pre{\omega}{2} \). However, as discussed in Remark~\ref{remcountableuncountable}, in the present paper we necessarily use 
techniques which are fairly different from those employed in~\cite{frimot} (and 
in the subsequent works on invariant universality~\cite{mot,cammarmot}). Part 
of the new ideas comes from analogous results obtained  (in a different 
context) in the forthcoming~\cite{andmot}.

The paper is organized as follows. After introducing some terminology and 
basic concepts in Section~\ref{sectionbasic}, in Sections~\ref{sectionspaces},~\ref{sectioninfinitarylogic},~\ref{sectionweaklycompact}, and~\ref{sectionanalytic} we present some (old and new) results on, respectively, 
standard Borel \( \kappa \)-spaces (a generalization of the notion of standard Borel space from descriptive set theory), 
infinitary logics, weakly compact cardinals, 
and analytic quasi-orders and equivalence relations. In Section~\ref{sectionmax} we prove a technical result dealing with the quasi-order \( 
\leq_{\max} \) which will be crucial for the proof of the main results, while in 
Section~\ref{sectionlabels} we introduce some particular structures, called 
labels, which will be used in the main construction. Sections~\ref{sectioncomplete} and~\ref{sectionmain} contain the main results of the 
paper: in the first one we show that the embeddability relation \( 
\sqsubseteq^\kappa_{\mathsf{TREE}} \) on trees of size a weakly compact 
cardinal \( \kappa \) is complete for analytic quasi-orders, while in the second 
one we strengthen this result by showing that \( 
\sqsubseteq^\kappa_{\mathsf{TREE}} \) is in fact strongly invariantly universal.
Finally, in Section~\ref{sectionquestions} we collect some questions and open problems related to the results of this paper.

\section{Notation and basic definitions}\label{sectionbasic}

Throughout the paper we will work in \( \mathsf{ZFC} \), i.e.\ in Zermelo-Fraenkel set theory together with the Axiom of Choice.

Let \( \On \) be the class of all ordinals. The Greek letters \( \alpha, \beta, \gamma, \delta \) (possibly 
with various decorations) will usually denote ordinals, while the letters \( \nu, \lambda, 
\kappa \) will usually denote cardinals. 
Given two sets \( A, B \), we denote by \( \pre{B}{A} \) the set of all  
sequences of elements of \( A \) indexed by elements of \( B \), i.e.\ the set of 
all (total) functions from \( B \) to \( A \). If \( B' \subseteq B \) and \( s \in 
\pre{B}{A} \) we let \( s \restriction B' \) be the restriction of \( 
s \) to \( B' \). In particular, given \( \gamma \in \On \) the set \( 
\pre{\gamma}{A} \) is the set of all \( \gamma \)-sequences from \( A \). Moreover, we set \( \pre{< \gamma}{A} = \bigcup_{\alpha < \gamma} \pre{\alpha}
{A} \) and \( \pre{< \On}{A} = \bigcup_{\alpha \in \On} \pre{\alpha}{A} 
\). For \( s  \in \pre{< \On}{A} \) we denote by \( \leng{s} \) the 
\emph{length of \( s \)}, i.e.\ the unique \( \gamma \in \On \) such that \( s 
\in \pre{\gamma}{A} \). We also set
\[ 
\pre{\SUCC(<\gamma)}{A} = \left\{ s \in \pre{< \gamma}{A} \mid \leng{s}\text{ is a successor ordinal} \right\} = \bigcup\nolimits_{\alpha + 1 <  \gamma} \pre{\alpha+ 1}{A}.
\]
For \( \gamma \in 
\On \) and \( a \in A \), we denote by \( a^{(\gamma)} \) the \( \gamma \)-sequence constantly equal to \( a \). If \( s,t \in \pre{< \On}{A} \), then \( s {}^\smallfrown t \) denotes the concatenation of \( s \) and \( t \). When \( s = \langle a \rangle \) for some \( a \in A \), we will often simply write \( a {}^\smallfrown t \) in place of \( \langle a \rangle {}^\smallfrown t \). 

Given a set \( A \), we let \( |A| \) be the \emph{cardinality of \( A \)}, i.e.\ 
the unique cardinal \( \kappa \) such that \( A \) is in bijection with \( \kappa 
\). Given a cardinal \( \kappa \), we denote with \( \kappa^+ \) the smallest 
cardinal (strictly) greater than \( \kappa \). Moreover, we let \( [A]^\kappa \) 
be the collection of all subsets of \( A \) of cardinality \( \kappa \), and \( 
[A]^{< \kappa}  = \bigcup_{\gamma < \kappa} [A]^\gamma \) be the collection of all subsets of \( A \) of cardinality \( < 
\kappa \).

If \(f \) is a function between two sets \( X \) and \( Y \) and \( A \subseteq X \) we set
\[ 
f``\, A = \{ y \in Y \mid \exists a \in A \left (f(a) = y \right) \}.
 \]

Let \( \Hes \colon \On \times \On \to \On \) be
the Hessenberg pairing function for the class of all ordinals
\( \On \) (see e.g.\ \cite[p.\ 30]{jech}), i.e.\ the unique surjective function such that for all \( \alpha, \alpha',\beta, \beta' \in \On \)
\begin{align*} 
\Hes((\alpha,\beta)) \leq \Hes((\alpha',\beta')) \iff &\max { \{ \alpha,\beta \} < \max \{ \alpha', \beta' \} }\vee \\
& \left({\max\{ \alpha,\beta \} = \max \{ \alpha', \beta' \}} \wedge {(\alpha , \beta ) \leq_{\mathsf{lex}} (\alpha',\beta')}\right),
 \end{align*} 
where \( \leq_{\mathsf{lex}} \) is the lexicographical ordering on \( \On \times \On \).
Then define by induction the bijections \( \Hes_n \colon \pre{n}{\On} \to \On \)
(for \( n \geq 2 \)) by letting \( \Hes_2 = \Hes \) and
\( \Hes_{n+1}(\alpha_0, \dotsc , \alpha_n) = \Hes( \alpha_0 ,
\Hes_n (\alpha_1, \dotsc, \alpha_n)) \) for \( \alpha_0, \dotsc,
\alpha_n \in \On \).

\begin{defin} \label{deftree}
A \emph{(generalized) tree} is a structure \( T = (T, \preceq) \) such that:
\begin{enumerate}
\item \label{deftree1}
 \( T \) is a \emph{set}, and 
\item \label{deftree2}
\( \preceq \) is a partial order (i.e.\ a reflexive, transitive, antisymmetric binary relation) on \( T \)  such that the set
\[ 
\pred(x) = \{ y \in T \mid y \preceq x \wedge y \neq x \}
 \] 
of predecessors of any point \( x \in T \) is linearly ordered by \( \preceq \).
\end{enumerate}
\end{defin}
\noindent
Notice that, in particular, any linear order is a tree. The elements of a tree are called indifferently points or nodes.
Given a tree \( T \) as above and a point \( x \in T \), the \emph{upper cone above \( x \)} is the set
\[ 
\cone(x) = \{ y \in T \mid x \preceq y \}.
 \] 
Two elements \( x,y \in T \) are said \emph{comparable} if 
\( {x \preceq y} \vee {y \preceq x} \) (i.e.\ if
\( {x \in \pred(y)} \vee {y \in \pred(x)} \)), and are said 
\emph{compatible} if they have a common predecessor, i.e.\ there is 
\( z \in \pred(x) \cap \pred(y) \) (given such a \( z \), we will also say that 
\( x \) and \( y \) are \emph{compatible via \( z \)}).  A tree \( T \)  is \emph{connected} if every two points in \( T \) are compatible. Let \( T \) be a tree: a tree \( T' \) is a \emph{maximal connected component of \( T \)} if   
it is a connected subtree of \( T \) such that all points in \( T \) which are comparable (equivalently, compatible) with an element of \( T' \) must belong to \( T' \) themselves.

Notice that it \( T_0, T_1 \) are trees and \( i \) is an embedding of \( T_0 \) 
into \( T_1 \) then for every point \( x \) of \( T_0 \) we must have 
\( i``\,\pred(x) \subseteq \pred(i(x)) \) and 
\( i``\, \cone(x) \subseteq \cone(i(x)) \). Moreover, \( i \) preserves 
(in)comparability, that is for every pair of points \( x,y \) of \( T_0 \), \(x \) 
and \( y \) are comparable if and only if \( i(x) \) and \( i(y) \) are comparable. 
As for compatibility, if \( x,y \) are compatible then the same is true for 
\( i(x),i(y) \), but the converse does not necessarily hold. 

Generalized trees are quite popular in the literature, see e.g.\ Steel's paper 
on the proof of a restricted form of Vaught's conjecture~\cite{ste}. However, in (descriptive) set theory the word ``tree'' 
often refers to a special kind of tree, namely to trees \( T = (T, \preceq) \) 
such that \( T \subseteq \pre{< \On}{X} \), \( T \) is closed under initial 
subsequences, and \( \preceq \) is the inclusion relation, i.e.\ for every \( s,t \in T \) we have \( s \preceq t \) if and only if \( s \) is an initial segment of \( t \). To distinguish this particular kind of trees 
from the general case, we will call such structures \emph{descriptive set-theoretical 
trees} (\emph{DST-trees} for short) \emph{on \( X \)}. Notice that every DST-tree can 
be unambiguously identified with its domain. Moreover,  a DST-tree has no 
infinite  \( \preceq \)-descending chains, so that a linear order is (isomorphic 
to) a DST-tree if and only if it is well-founded. It is easy to check that a 
(generalized) tree can be embedded into a DST-tree if and only if \( \pred(x) \) 
is a well-order for every \( x \in T \). Moreover, if \( T \) has the further 
properties that there is a unique \( \preceq \)-least element and that \( x=y \) 
whenever \( x,y \in T \) are such that \( \pred(x) = \pred(y) \) and the order type of \( \pred(x) \) is a limit ordinal, then \( T \) is isomorphic to a DST-tree.

Given a DST-tree \( T \) on \( X \), we call \emph{height of \( T \)} the minimal \( \alpha \in \On \) such that \( \leng{x} < \alpha \) for every \( x \in T \) (such an ordinal must exist because by definition \( T \) is a set). Let \( \kappa \) be a cardinal. If \( T \subseteq \pre{< \kappa}{X} \) is a DST-tree, we call \emph{branch} (of \( T \)) any maximal linearly ordered subset of \( T \). A branch \( b \subseteq T\) is called \emph{cofinal} if the set \( \{ \leng{s} \mid s \in b \} \) is cofinal in \( \kappa \), i.e.\ if \( \bigcup b \in \pre{\kappa}{X} \). We	call \emph{body of \( T \)} the set
\begin{align*}
[T] &= \{ f \in \pre{\kappa}{X} \mid \forall \alpha < \kappa \left (f \restriction \alpha \in T \right) \} \\
& = \left \{ \bigcup b \mid b \text{ is a cofinal branch of } T \right \}.
 \end{align*}
When \( X = Y \times \kappa \) we let
\[ 
\proj[T] = \left\{ f \in \pre{\kappa}{Y} \mid \exists g \in \pre{\kappa}{\kappa} \left ((f,g) \in [T] \right) \right\}
 \] 
be the \emph{projection} (on the first coordinate) of the body of \( T \).

For all other undefined notation and concepts, we refer the reader to~\cite{kechris},~\cite{kan} and~\cite{jech}.

\section{Standard Borel \( \kappa \)-spaces}\label{sectionspaces}

On the Cantor space \( \pre{\omega}{2} \), the product topology (where \( 2 \) is 
endowed with the discrete topology) coincides with the topology
generated by the basic (cl)open sets of the form \( \bN_s = \{ x \in \pre{\omega}{2}  \mid s \subseteq
x \} \),
where \( s \in 
\pre{<\omega}{2} \). As discussed in~\cite{andmot}, when replacing \( \omega \) with an uncountable
cardinal
 \( \kappa \) there are then two topologies on \( \pre{\kappa}{2} \) which 
straightforwardly generalize the topology of \( \pre{\omega}{2} \), namely
the \emph{product topology}
\( \tau_p \) (where
\( 2 \) is endowed with the discrete topology again), and the \emph{bounded
topology}
\( \tau_b \),
which is the one generated by the sets of the form 
\[
\bN_s = \left\{ x \in
\pre{\kappa}{2}  \mid  s \subseteq x \right\}
\] 
for \( s \in \pre{< \kappa}{2} \).
Notice that if \( \kappa \) is regular, 
\begin{equation}\label{eqbasistau_b}
\mathcal{B} =\left \{ \bN_s \mid s \colon d \to 2 \text{ for some } d \in [\kappa]^{< \kappa}\right \}
 \end{equation}
is another natural basis for \( \tau_b \).
Both
topologies are zero-dimensional and Hausdorff, and one 
immediately sees that \( \tau_b \) refines \( \tau_p \). However, the
two
 topologies are distinct because if \( s \in \pre{< \kappa}{2} \) is 
 infinite then \( \bN_s \) is \( \tau_b \)-open but it is a proper 
\( \tau_p \)-closed set. Similar considerations hold for the space \( \pre{\kappa}{\kappa} \).

The above definitions make sense also for any space of the form \( \pre{I}{2} \) whenever \( |I| = \kappa \). Obviously, the topology \( \tau_p \) on \( \pre{I}{2} \) is again the product of the discrete topologies on \( 2 \), i.e.\ the topology generated by the collection of all the sets of the form
\( \bN_s = \{ x \in \pre{I}{2} \mid s \subseteq x \} \),
where \( s \) is a finite partial function from \( I \) to \( 2 \). For the bounded topology \( \tau_b \), a natural approach is to choose a bijection between \( I \) and \( \kappa \) (which naturally induces a bijection between \( \pre{I}{2} \) and \( \pre{\kappa}{2} \)), and then ``transfer'' the bounded topology of \( \pre{\kappa}{2} \) onto the space \( \pre{I}{2} \) (the topology on \( \pre{I}{2} \) obtained in this way will be again called bounded topology and denoted by \( \tau_b \)). Of course, in this way the definition of \( \tau_b \) on \( \pre{I}{2} \) depends \emph{a priori}  on the chosen bijection. However, this is true only if \( \kappa \) is singular: if \( \kappa \) is regular, the resulting \( \tau_b \) is exactly the topology generated by the basis
\begin{equation}\label{eqbasic} 
\mathcal{B} =\left \{ \bN_s = \left\{ x \in \pre{I}{2} \mid s \subseteq x \right\} \mid s \colon d \to 2 \text{ for some } d \in [I]^{< \kappa} \right\}. 
\end{equation} 
It is easy to check that in the last case any bijection between \( I \) and \( \kappa \) canonically induces a homeomorphism between \( \pre{I}{2} \) and \( \pre{\kappa}{2} \) as long as both spaces are endowed with the same kind of topology, i.e.\ either both with \( \tau_p \) or else both with \( \tau_b \).

When \( X \) is a finite product of copies of \( \pre{\kappa}{2} \) and \( \pre{\kappa}{\kappa} \), we denote by \( \tau_b \) (respectively, \( \tau_p \)) the topology on \( X \) obtained as the product of the topologies \( \tau_b \) (respectively, \( \tau_p \)) on each of the factors of \( X \). 
Notice in particular that a set \( C \subseteq X \) is
\( \tau_b \)-closed if and only if it is the body of a pruned DST-tree
\( T \) of height \( \kappa \) on the corresponding space, e.g.\ on the space \( 2 
\times 2 \times \kappa \) if \( X = \pre{\kappa}{2} \times \pre{\kappa}{2} \times 
\pre{\kappa}{\kappa} \).

\begin{defin}
Let \( X \) be a topological space and \( \lambda \) be an infinite
cardinal. A subset of \( X \) is called \emph{\( \lambda \)-Borel} if
and only if it belongs to the minimal class \( \bB_\lambda \subseteq \pow(X) \) which contains all open sets and is closed under
the operations of complementation and unions of size smaller
than \( \lambda \).
\end{defin}

When the space \( X \) (or its topology) is not clear from the
context, we will add references to it in all the terminology and
symbols introduced above. Notice that \( \bB_\lambda \) can also
be equivalently characterized as the smallest class containing
all open and closed sets of \( X \) and closed under unions and
intersections of size \( < \lambda \).

\begin{defin}
Let \( X,Y \) be arbitrary topological spaces and \( \lambda \) be an
infinite cardinal. A function \( f \colon X \to Y \) is said
\emph{\( \lambda \)-Borel (measurable)} if \( f^{-1}(U) \in
\bB_\lambda(X) \) for every open set \( U \subseteq Y \) (equivalently, if
\( f^{-1}(B) \in \bB_\lambda(X) \) for every \( B \in \bB_\lambda(Y) \)).
\end{defin}

In this paper, we will fix a cardinal \( \kappa \) and be concerned with the class \( \bB_{\kappa^+}(X) \) of  \( \kappa ^+ \)-Borel sets of various topological spaces \( X \). For ease of notation and terminology, 
the \( \kappa^+ \)-Borel
subsets of \( X \) will be simply called \emph{Borel} sets and, similarly,
\( \kappa^+ \)-Borel functions will be simply called \emph{Borel}
functions.
Every Borel subset \( B \) of \( X \) will be always endowed with the
relative topology inherited from \( X \): in this way, a set \( C \subseteq B \) is Borel in \( B \) if and only if it is a Borel subset of the whole \( X \).

Let now \( X \) be either \( \pre{\kappa}{2} \) or \( \pre{\kappa}{\kappa} \). If \( \kappa \) is such that \( \kappa^{< \kappa} = \kappa \) 
(equivalently, \( \kappa \) is regular and \( 2^{< \kappa} = \kappa \)),
then both \( \tau_p \) and \( \tau_b \)  give rise to the same collection of
\( \kappa^+ \)-Borel
subsets of \( X \): this is because \( \tau_b \) refines \( \tau_p \), and on the other hand the basis \( \mathcal{B} \) of \( \tau_b \) defined in~\eqref{eqbasistau_b} consists of \( \tau_p \)-closed sets and (under the cardinal assumption above) is such that
\( | \mathcal{B}| = \kappa \). The condition \( \kappa^{< \kappa} = \kappa \) has
been considered in many papers dealing with the topological
spaces \( \pre{\kappa}{2} \) or \( \pre{\kappa}{\kappa} \) (see e.g.\ \cite{frihytkul,halshe}), and it is a very natural condition to be
asked when looking for ``positive'' results in the context of
uncountable cardinals \( \kappa \): in fact, if \( \kappa^{< \kappa} >
\kappa \) many nice results which are true in the countable
case, like e.g.\ the Lopez-Escobar theorem (see the discussion
in Section~\ref{sectioninfinitarylogic}), can fail for
such a \( \kappa \).
For this reason, \emph{unless otherwise specified, from this
point onward we will tacitly assume that \( \kappa \) is an infinite
uncountable cardinal such that }
\begin{equation} \label{eqcardinalcondition}
\kappa^{< \kappa} = \kappa.
\end{equation}

Since all the results
of this paper will only depend on the structure of the
\( \kappa^+ \)-Borel subsets of the spaces under consideration, in the particular case of a finite product \( X \) of copies of \( \pre{\kappa}{2} \) and \( \pre{\kappa}{\kappa} \) it will be enough to concentrate on
one of the two natural topologies mentioned above: therefore \emph{we will always tacitly assume that such an \( X \) (as well as any other space of the form \( \pre{I}{2} \) with \( |I| = \kappa \)) is endowed with the bounded
topology \( \tau_b \)}.\footnote{The preference for the topology \( \tau_b \) is due to technical reasons.} 
Notice that since~\eqref{eqcardinalcondition} implies that \( \kappa \) is regular, the intersection of less than
\( \kappa \)-many \( \tau_b \)-basic open sets of \( X \) is still open by e.g.\
\cite[Lemma 3.8]{andmot}. 

\begin{defin}
Two topological spaces \( X \) and \( Y \) are said \emph{(\( \kappa^+ \)-)Borel isomorphic} if there is a bijection \( f \colon X \to Y \) such that both \( f \) and \( f^{-1} \) are \(\kappa^+\)-Borel functions.
\end{defin}

As in the case \( \kappa = \omega \), one can easily show the
following:

\begin{fact}\label{factisomorphic}
The spaces \( \pre{\kappa}{2} \) and \( \pre{\kappa}{\kappa} \) are Borel 
isomorphic.
\end{fact}

It immediately follows than any two spaces which
are finite products of copies of \( \pre{\kappa}{2} \) and \( \pre{\kappa}{\kappa} \) (in fact, even products of \( \kappa \)-many copies of such spaces) are Borel isomorphic. However, contrarily to the
countable case, it is no more true that every Borel subset of e.g.\ 
\( \pre{\kappa}{\kappa} \) of cardinality \( > \kappa \) is Borel isomorphic to \( \pre{\kappa}{2} \):  if \( 2^\kappa > \kappa^+ \) and there is a \( \kappa \)-Kurepa tree
\( T \) (i.e.\ a DST-tree \( T \subseteq \pre{< \kappa}{ \kappa} \) of height
\( \kappa \) with all levels of size \( < \kappa \) and exactly
\( \kappa^+ \)-many cofinal branches), then \( [T] \) is a
\( \tau_b \)-closed subset of \( \pre{\kappa}{\kappa} \) which has cardinality \( >\kappa \) but cannot contain a copy of \( \pre{\kappa}{2} \). 

\begin{remark}\label{remkurepa}
In view of the results which will be considered in this
paper, it is maybe worth noting that such a situation may happen
 also when \( \kappa \) is weakly compact (see Section~\ref{sectionweaklycompact}). In fact, assume that \( \kappa \) is a
weakly compact cardinal such that \( (\kappa^+)^L = \kappa^+ \) and
\( 2^\kappa > \kappa^+ \): then consider the DST-tree \( T \) consisting of
all sequences in \( \pre{< \kappa}{2} \cap L \). Since for every \( x
\in \pre{\kappa}{2} \) we have 
\[ 
\forall \lambda < \kappa \left (x
\restriction \lambda \in L \right ) \Rightarrow x \in L, 
\]
\( [T] = (\pre{\kappa}{2})^L \), and hence \( [T] \) has size \( (2^\kappa)^L = (\kappa^+)^L =
\kappa^+ < 2^\kappa \).
\end{remark}

\begin{defin}\label{defstandard}
A topological space is called a \emph{\( \kappa \)-space} if its topology admits a basis of size \( \leq \kappa \). 
A \( \kappa \)-space is called \emph{standard Borel} if it is Borel isomorphic to a Borel subset of \( \pre{\kappa}{\kappa} \).
\end{defin}

In particular, under~\eqref{eqcardinalcondition} all Borel subsets of the 
spaces \( \pre{\kappa}{\kappa} \) , \( \pre{\kappa}{2} \), and 
\( \pre{I}{2} \) (where \( |I| = \kappa \)) are standard Borel \( \kappa \)-spaces. Notice however that standard 
Borel \( \kappa \)-spaces may also have a highly nonstandard behavior as 
topological spaces, e.g.\ they can be non Hausdorff, and even not \( T_1 \). 

\begin{remark}\label{rmk:standardBorel}
When \( \kappa  = \omega \), the Definition~\ref{defstandard} coincides with the usual notion of standard Borel space considered in descriptive set theory, see e.g.\ \cite[Chapter 12]{kechris}. Notice also that, by definition, the class of standard Borel \( \kappa \)-spaces is closed under taking Borel subspaces: this is why in our definition we required \( X \) to be Borel isomorphic to a Borel subset of \( \pre{\kappa}{\kappa} \) and not to the entire space \( \pre{\kappa}{\kappa} \) (see Remark~\ref{remkurepa} and the observation preceding it). Finally, the class of standard Borel 
\( \kappa \)-spaces is easily seen to be also closed under products of size \( \leq \kappa \).
\end{remark}

\begin{defin}\label{defanalytic}
Let \( X \) be a standard Borel \( \kappa \)-space. A set \( A \subseteq X \) is said 
\emph{analytic} if it is a continuous image of a closed 
subset of \( \pre{\kappa}{\kappa} \). The class of
all analytic subsets of \( X \) is denoted by \( \boldsymbol{\Sigma}^1_1(X) \).
\end{defin}

It may seem that our definition of analytic depends on the topology \( \tau \)	of \( X \), but we will see that in fact it depends only on the Borel structure of \( X \): in fact, by~\ref{propanalytic4}) of Proposition~\ref{propanalytic} we have that if \( \tau' \) is another topology on \( X \) such that \( \bB_{\kappa^+}(X,\tau) = \bB_{\kappa^+}(X,\tau') \) then \( \boldsymbol{\Sigma}^1_1(X,\tau) = \boldsymbol{\Sigma}^1_1(X, \tau') \).

\begin{lemma}\label{lemmaclosureofanalytics}	
Let \( X \) be an \emph{Hausdorff} standard Borel \( \kappa \)-space. Then the class \( \boldsymbol{\Sigma}^1_1 (X) \) is closed under unions and intersections of size \( \leq \kappa \).
\end{lemma}

\begin{proof}
Let \( \{ A_\alpha \mid \alpha < \kappa \} \) be a collection of analytic subsets of \( X \), and for each \( \alpha < \kappa \) let \( f_\alpha \colon C_\alpha \to X \) be a continuous surjection from some closed set \( C_\alpha \subseteq \pre{\kappa}{\kappa} \) onto \( A_\alpha \). Let
\[ C = \bigcup_{\alpha< \kappa} \alpha {}^\smallfrown{} C_\alpha , \]
where \( \alpha {}^\smallfrown{} C_\alpha = \{ \alpha {}^\smallfrown{} x \mid x \in C_\alpha \} \). Then \( C \subseteq \pre{\kappa}{\kappa} \) is closed and the function 
\[ 
f \colon C \to X \colon x \mapsto f_{x(0)}(x^-),
\]
where \( x^- = \langle x(\gamma) \mid 1 \leq \gamma < \kappa \rangle \), is well-defined, continuous, and onto \( \bigcup_{\alpha < \kappa} A_\alpha \).

Next we show that there is a continuous surjection \( g \colon C \to X \) from some closed \( C \subseteq \pre{\kappa}{\kappa} \) onto \( \bigcap_{\alpha < \kappa} A_\alpha \).
Let   \( f' \) be the function 
\[ 
f' \colon C' \to \pre{\kappa}{X} \colon \langle x_\alpha \mid \alpha < \kappa \rangle \mapsto \langle f_\alpha(x_\alpha) \mid \alpha< \kappa \rangle,
 \] 
where \( C' \) is the closed set \( \prod_{\alpha<\kappa} C_\alpha \subseteq \pre{\kappa}{(\pre{\kappa}{\kappa})} \).
It is easy to check that \( f' \) is continuous. Since the diagonal 
\[ 
\Delta = \{ \langle x_\alpha \mid \alpha< \kappa \rangle \mid \forall \alpha, \beta < \kappa \left (x_\alpha = x_{\beta} \right) \} 
\] 
of \( \pre{\kappa}{X} \) is a closed set (because 
\( X \) is Hausdorff), the set \( C'' = 
f'^{-1}(\Delta) \) is closed in \( C' \), and hence it is closed in \( \pre{\kappa}{(\pre{\kappa}{\kappa})} \) as well. Let \( \pi \colon \pre{\kappa}
{X} \to X \) be the projection on 
the first coordinate. Then \( \pi \circ ( f' \restriction C'' ) \colon C'' \to 
X \) is a continuous surjection onto \( \bigcap_{\alpha 
< \kappa } A_\alpha \). But since \( \pre{\kappa}{(\pre{\kappa}{\kappa})} \) 
and \( \pre{\kappa}{\kappa} \) are homeomorphic, this means that there is a 
closed \( C \subseteq \pre{\kappa}{\kappa} \) and a continuous 
surjection \( g \colon C \to \bigcap_{\alpha< \kappa} A_\alpha\), as required. 
\end{proof}

The next proposition is an illuminating example of a simple and natural result which apparently holds only under~\eqref{eqcardinalcondition}.

\begin{proposition}\label{propborelanalytic}
Let \( X \) be a standard Borel \( \kappa \)-space. Then every Borel subset of \( X \) is analytic.
\end{proposition}

\begin{proof}
First we consider the case \( X = \pre{\kappa}{\kappa} \). Note that every closed subset of \( \pre{\kappa}{\kappa} \) is trivially
analytic. Moreover, by~\eqref{eqcardinalcondition} there is a clopen basis of 
\( \pre{\kappa}{\kappa} \) of size \( \kappa \) (e.g.\ we can take the basis \( \mathcal{B} \) defined
in~\eqref{eqbasistau_b}): since each element of such a  basis, being closed, is analytic,
then every open set is analytic as well by Lemma~\ref{lemmaclosureofanalytics}. Thus, by
Lemma~\ref{lemmaclosureofanalytics} again, every Borel set is analytic as well because the collection
of all Borel
sets is the minimal class of subsets of \( \pre{\kappa}{\kappa} \) which contains all open and
closed sets and is closed under unions and intersections of size \( \leq \kappa \).

Let us now consider an arbitrary standard Borel \( \kappa \)-space \( X \).
We first show the following claim.

\begin{claim}\label{claimborelanalytic}
Let \( D \subseteq \pre{\kappa}{\kappa} \) be an analytic set and 
\( \set{ D_\alpha }{\alpha< \kappa } \) be a family of analytic subsets of 
\( D \) such that each \( D \setminus D_\alpha \) is analytic as well. Then there is a closed \( C \subseteq \pre{\kappa}{\kappa} \) and 
a continuous surjection \( h \colon C \to D \) such that \( h^{-1}(D_\alpha) \) is open (relatively to \( C \))
for every \( \alpha<\kappa \).
\end{claim}

\begin{proof}[Proof of the Claim]
The argument is similar to the one used in Lemma~\ref{lemmaclosureofanalytics} to show that \( \boldsymbol{\Sigma}^1_1(X) \) is closed under intersection of size \( \leq \kappa \) (for \( X \) an Hausdorff standard Bored \( \kappa \)-space). For each \( \alpha<\kappa \), let 
\( h_\alpha \colon C_\alpha \to D \) be a continuous surjection from some closed set \( C_\alpha \subseteq \pre{\kappa}{\kappa} \) onto \( D \) such 
that \( h_\alpha^{-1}(D_\alpha) \) is (cl)open relatively to \( C_\alpha \): such \( C_\alpha \)'s and \( h_\alpha \)'s can be easily constructed using the fact that both 
\( D_\alpha \) and \( D \setminus D_\alpha \) are analytic subsets of \( \pre{\kappa}{\kappa} \).
Set \( C' = \prod_{\alpha<\kappa} C_\alpha \subseteq \pre{\kappa}{(\pre{\kappa}{\kappa})} \) and define 
\( h' \colon C' \to {}^{\kappa}
(\pre{\kappa}{\kappa}) \) using the \( h_\alpha \)'s coordinatewise, i.e.\ set 
\( h'(\seq{x_\alpha}{\alpha< \kappa}) = \seq{h_\alpha(x_\alpha)}{\alpha < \kappa} \). Clearly, \( C'' \) is closed in \( \pre{\kappa}{(\pre{\kappa}{\kappa})} \) and the  function \( h' \) is  continuous. Since the diagonal \( \Delta \) is a closed set, the set \( C'' = (h')^{-1}(\Delta) \) is closed in \( C' \), and hence also in \( \pre{\kappa}{(\pre{\kappa}{\kappa})} \). Let 
\( h'' \colon C'' \to \pre{\kappa}{\kappa} \) be the function  which sends
\( \seq{x_\alpha}{\alpha< \kappa} \) to the unique \( y \) such 
that \( y = h_\alpha(x_\alpha) \) for some/all \( \alpha < \kappa \). Notice 
that \( h'' \) is continuous, and is onto \( D \) because all \( h_\alpha \)'s were onto 
\( D \) as well. Moreover, 
\[ 
h''^{-1}(D_\alpha) = C'' \cap \left({}^{\alpha}(\pre{\kappa}{\kappa}) \times 
h_\alpha^{-1}(D_\alpha) \times {}^{\kappa}(\pre{\kappa}{\kappa}
)\right), 
\]
and hence \( h''^{-1}(D_\alpha) \) is an open set in (the relative topology of) 
\( C'' \). Since \( \pre{\kappa}{(\pre{\kappa}{\kappa})} \) is homeomorphic to \( \pre{\kappa}{\kappa} \), this means that there is a closed set \( C \subseteq \pre{\kappa}{\kappa} \) and a continuous surjection \( h \colon C \to D \) such that \( h^{-1}(D_\alpha) \) is open in (the relative topology of) \( C \) for every \( \alpha < \kappa \), as required. 
\end{proof}

 Since Borel subsets of standard Borel \( \kappa \)-spaces are standard Borel \( \kappa \)-spaces themselves (Remark~\ref{rmk:standardBorel}), it is enough to show that for every standard Borel \( \kappa \)-space \( X \) there is a continuous surjection from some closed \( C \subseteq \pre{\kappa}{\kappa} \) onto \( X \).
Let  \( B \subseteq \pre{\kappa}{\kappa} \) be a Borel set and \( i \colon B \to X \) be a Borel isomorphism witnessing the fact that \( X \) is standard Borel. Let \( \mathcal{B} = \{ B_\alpha \mid \alpha < \kappa \} \) be a basis of size \( \kappa \) for the topology of \( X \). Let \( \alpha < \kappa \). Then both \( i^{-1}(B_\alpha) \) and \( i^{-1}(X \setminus B_\alpha) \) are Borel in \( B \), and hence also in \( \pre{\kappa}{\kappa} \). Let \( C \subseteq \pre{\kappa}{\kappa} \) and \( h  \colon C \to B \) be obtained by applying Claim~\ref{claimborelanalytic} with \( D = B \) and \( D_\alpha = i^{-1}(B_\alpha) \) (this can be done because we already showed that Borel subsets of \( \pre{\kappa}{\kappa} \) are analytic). Then \( i \circ h \) is a continuous surjection of \( C \) onto \( X \), as required.
 \end{proof}

In particular, from Lemma~\ref{propborelanalytic} it follows that a subset of a
Borel subset \( B \) of a standard Borel \( \kappa \)-space \( X \) is analytic (in \( B \)) if and only if it is analytic in the whole space 
\( X \). 

\begin{proposition}\label{propanalytic}
Let \( X \) be a standard Borel \( \kappa \)-space. Given a nonempty \( A \subseteq X \), the following are equivalent:
\begin{enumerate}[i)]
\item \label{propanalytic1}
\( A \) is analytic;
\item \label{propanalytic3}
\( A \) is a continuous image of a Borel \( B \subseteq
\pre{\kappa}{\kappa} \);
\item \label{propanalytic4}
\( A \) is a Borel image of a Borel set \( B \subseteq \pre{\kappa}{\kappa} \);
\item \label{propanalytic5}
\( A \) is the image of an analytic set \( D \subseteq \pre{\kappa}{\kappa} \) 
under a bianalytic-measurable map \( f \colon D \to X \) (i.e.\  \( f \) is such that 
\( f^{-1}(U) \) and \( f^{-1}(X \setminus U) \) are analytic for every open set \( U \subseteq X \)).
\end{enumerate}
\end{proposition}

\begin{proof}
Clearly,~\ref{propanalytic1})~\( \Rightarrow \)~\ref{propanalytic3}) and~\ref{propanalytic3})~\( \Rightarrow \)~\ref{propanalytic4}).  \ref{propanalytic4})~\( \Rightarrow \)~\ref{propanalytic5}) follows from Proposition~\ref{propborelanalytic}. 
Finally, to prove~\ref{propanalytic5})~\( \Rightarrow \)~\ref{propanalytic1}) first observe that  
by Definition~\ref{defstandard} there is a basis \( \mathcal{B} = \{ B_\alpha \mid \alpha < \kappa \} \) of size 
\( \kappa \) for the 
topology of \( X \). Then apply Claim~\ref{claimborelanalytic} with \( D_\alpha = f^{-1}(B_\alpha) \): the composition of the resulting function \( h \colon C \to X \) with \( f \) gives a continuous surjection from the closed set \( C \subseteq \pre{\kappa}{\kappa } \) onto \( A \), which thus witnesses that \( A \) is analytic.
\end{proof}

\begin{remark} \label{rem:luckeschlicht}
When \( \kappa = \omega \), one may add to the equivalent conditions in Proposition~\ref{propanalytic} the following one: 
\begin{enumerate}[\emph{\phantom{ii}v)}]
\item
\emph{\( A \) is either empty, or a continuous image of \( \pre{\kappa}{\kappa} \). }
\end{enumerate}
This is because in the usual Baire space every closed set \( C \subseteq \pre{\omega}{\omega} \) is a \emph{retract} of the entire space, i.e.\ there is a continuous surjection \( r \colon \pre{\omega}{\omega} \to C \) which is the identity on \( C \). However, L\"ucke and Schlicht observed that if \( \kappa \) has uncountable cofinality this is no more true: the set \( C = \{ x \in \pre{\kappa}{\kappa} \mid x(\alpha) = 1 \text{ for finitely many \( \alpha \)'s} \} \) is closed in \( \pre{\kappa}{\kappa} \) but it is not a retract of \( \pre{\kappa}{\kappa} \). In fact, they proved (private communication) that if e.g.\ we assume \( \mathsf{V = L} \), so that~\eqref{eqcardinalcondition} is automatically satisfied for every regular \( \kappa \) by \( \mathsf{GCH} \), then for every inaccessible \( \kappa \) (or even just \( \kappa = \omega_1 \)) there is a nonempty closed \( C \subseteq \pre{\kappa}{\kappa} \) which is not a continuous image of \( \pre{\kappa}{\kappa} \): this shows that the collection of all subsets of \( \pre{\kappa}{\kappa} \) which are continuous images of \( \pre{\kappa}{\kappa} \) may form a proper subclass of \( \boldsymbol{\Sigma}^1_1(\pre{\kappa}{\kappa}) \).
\end{remark}

Using characterization~\ref{propanalytic4}) of Proposition~\ref{propanalytic} one can transfer all
results concerning analytic sets from one standard Borel \( \kappa \)-space to the others: in fact,
if \( X,Y \) are two standard Borel \( \kappa \)-spaces and \( i \colon X \to Y \) is a Borel
isomorphism, then \( A \subseteq X \) is analytic if and only if \( i``\, A \subseteq Y\) is analytic.
In particular, since we showed in Lemma~\ref{lemmaclosureofanalytics} that if \( X \) is an
Hausdorff standard Borel \( \kappa \)-space then \( \boldsymbol{\Sigma}^1_1(X) \) is closed under
unions and intersections of size \( \leq \kappa \), the same is true of 
\( \boldsymbol{\Sigma}^1_1(X) \) also for a non-Hausdorff standard Borel \( \kappa \)-space \( X \).

Let now \( X \) be a \emph{Hausdorff} standard Borel \( \kappa \)-space, and let
\( A \subseteq X \) be analytic. Then \( A \) is  the image of some closed \( C' \subseteq \pre{\kappa}{\kappa} \) 
under a continuous map \( f \colon C' \to X \). It follows that \( A \) is the projection on the first coordinate of the set 
\( C = \set{ (x,y) \in X \times \pre{\kappa}{\kappa} }{ {y \in C'} \wedge {f(y) = x} } \), which is a 
closed subset of \( X \times \pre{\kappa}{\kappa} \) because the graph of \( f \) is closed in \( X \times C' \) (as \( f \) is continuous and \( X \) is 
 Hausdorff) and \( C' \subseteq \pre{\kappa}{\kappa} \) is closed as well. Therefore we have shown that when \( X \) is as above then every analytic set \( A \subseteq X \) is the 
projection \( \proj(C) = \{ x \in X \mid \exists y \in \pre{\kappa}{\kappa} 
\left((x,y) \in C \right) \} \) of a closed set \( C \subseteq X \times \pre{\kappa}{\kappa} 
\). It turns out from Proposition~\ref{propborelanalytic} that since the projection function is continuous then 
the converse to the above statement is true as well (for an arbitrary standard Borel \( \kappa \)-space \( X \)):
for every closed \( C \subseteq X \times \pre{\kappa}{\kappa} \) the set \( 
\proj(C) \) is an analytic subset of \( X \).
In particular, we get that if \( X \) is a finite product of copies of the spaces \( \pre{\kappa}{2} \) and \( \pre{\kappa}{\kappa} \), then \( A \subseteq X \) is analytic if and only if \( A = \proj[T] \) for some DST-tree \( T \) of height \( \kappa \) on the corresponding space, e.g.\ on the space \( 2 \times 2 \times \kappa \) if \( X = \pre{\kappa}{2} \times \pre{\kappa}{2} \times \pre{\kappa}{\kappa} \); this will be our actual ``working definition'' of analytic sets for the rest of the paper.

\section{Infinitary logics and models of size $\kappa$}
\label{sectioninfinitarylogic}

The topic of this section are infinitary logics\footnote{See
e.g.\ \cite{barwise}.} and (spaces of) models of infinitary sentences --- see also~\cite{andmot} for more details.
For the rest of this section, fix a countable languege \( \L = \set{R_i }{i \in I } \) (\( |I| \leq \omega \)) consisting of finitary
relational symbols, and let \( n_i \) be the arity of \( R_i \). If \( X \)
is an \( \L \)-structure, we denote by \( R^X_i \) the interpretation of
the symbol \( R_i \) in \( X \), so that \( R^X_i \subseteq \pre{n_i}{X} \). 
With a little abuse of notation, when there is no danger of
confusion the domain of \( X \) will be denoted by \( X \) again (the
real meaning of the symbol \( X \) will be clarified by the
context). Therefore, unless otherwise specified, the \( \L \)-structure denoted by \( X \) will be of the form \( \left(X, \left\{ R_i^X \mid i \in I \right\} \right) \), where \( X \) is a set and each \( R^X_i \) is an \( n_i \)-ary relation on \( X \). If  \( X \) is an \( \L \)-structure and \( Y \subseteq X \), we will denote by \( X \restriction Y \) the substructure of \( X \) with domain \( Y \), i.e.\ the \( \L \)-structure \( \left( Y, \left\{ R_i^X \cap {}^{n_i} Y \mid i \in I \right\}\right) \).

Given\footnote{For the sake of simplicity, in addition to the explicitly mentioned logical symbols and variables we will freely use various kind of parentheses and punctuation --- this is a minor abuse since it is well-known that the use of such symbols can be entirely avoided.} an arbitrary infinite cardinal \( \kappa \), a
set \( \{ \neg, \bigwedge, \exists \} \) of logical symbols,
a binary relational symbol ``\( = \)'' for equality,
and a sequence \( \seq{\V_\alpha}{\alpha < \kappa} \)
of variables (meaning that all of these objects are distinct
elements of the model of set theory we are working in),
we define the infinitary logic \( \L_{\kappa^+ \kappa} \) as the minimal set
of formulas closed under the following conditions:
\begin{description}
\item[Atomic formulas]
for each \( i \in I \), \( \alpha, \beta < \kappa \), \( \langle \alpha_j \mid j < n_i - 1 \rangle \in \pre{n_i}{\kappa} \), the expressions \( \V_\alpha = \V_\beta \) and \( R_i(\V_{\alpha_0}, \dotsc, \V_{\alpha_{n_i - 1}}) \) are in \( \L_{\kappa^+ \kappa} \);\footnote{As usual, if \( n_i = 2 \) we will often write \( \V_{\alpha_0} \mathrel{R_i} \V_{\alpha_1} \) instead of \( R_i(\V_{\alpha_0}, \V_{\alpha_1}) \).}
\item[Negation]
if \( \upvarphi \) is in \( \L_{\kappa^+ \kappa} \) then also \( \neg \upvarphi \) is in \( \L_{\kappa^+ \kappa} \);	
\item[(Infinitary) Conjunction] if \( \Gamma \subseteq \L_{\kappa^+
\kappa} \) is a set of formulas of size \( \leq \kappa \)  
such that the total number of free variables appearing in \( \Gamma \)
 has size \( < \kappa \), then \( \bigwedge \Gamma \) is in \( \L_{\kappa^+
\kappa} \);
\item[(Infinitary) Existential quantification] if \( u \in
[\kappa]^{< \kappa} \) and \( \upvarphi \in \L_{\kappa^+ \kappa } \),
then the formula \( \exists \langle \V_\alpha \mid \alpha \in u \rangle \, \upvarphi \) (sometimes denoted by \( \exists \V_{u_0} \exists
\V_{u_1} \dotsc \exists \V_{u_\alpha} \dotsc \, \upvarphi \)) is in
\( \L_{\kappa^+ \kappa} \).
\end{description}

A formula in \( \L_{\kappa^+ \kappa} \) will also be called \emph{\( \L_{\kappa^+ \kappa} \)-formula}. Moreover, for \( \upvarphi, \uppsi \in \L_{\kappa^+ \kappa} \) we write \( \upvarphi \wedge \uppsi \) as an abbreviation for \( \bigwedge \{ \upvarphi, \uppsi \} \), and for \( J \) a set of indexes of size \( \leq \kappa \) we will write \( \bigwedge_{j \in J } \upvarphi_j \) in place of the more formally correct \( \bigwedge \{ \upvarphi_j \mid j \in J \} \). Similarly, the expression \( \exists \V_\beta \) (for \( \beta < \kappa \)) will be used as an abbreviation for \( \exists \langle \V_\alpha \mid \alpha \in u \rangle \), where \( u = \{ \beta \} \).

The disjunctions \( \bigvee, \vee \), the binary connectives of implication ``\( \Rightarrow \)'' and 
bi-implication (or equivalence)  ``\( \iff \)'', and
the universal quantifications \( \forall \seq{\V_\alpha}{\alpha \in u} \) (for \( u \in [\kappa]^{< \kappa} \)) and \( \forall \V_\alpha\) are
defined using negation,
conjunctions and existential quantifications in the
usual way. 

Notice that by construction each formula in \( \L_{\kappa^+ \kappa} \) has always \( < \kappa \) free variables occurring in it. A straightforward
computation shows that 
\begin{equation}\label{eqlogic}
 | \L_{\kappa^+ \kappa} | =  \left(\kappa^{< \kappa}\right)^\kappa =
\kappa^\kappa = 2^\kappa. 
\end{equation}

We will often use the symbols \( x,y,z,x_j,y_j,z_j \)
(where the \( j \)'s are elements of some set of indexes \( J \) of size \( \leq \kappa \)) as 
\emph{meta-variables}, i.e.\ to denote
arbitrary variables of the language \( \L_{\kappa^+ \kappa} \). In writing \( \upvarphi(\seq{x_j}{j \in J}) \)
(\( |J| < \kappa \)) we will always tacitly assume that the
variables \( x_j \) are distinct (and similarly with the \( y_j \)'s and the
\( z_j \)'s in place of the \( x_j \)'s). The semantics of
\( \L_{\kappa^+ \kappa} \), as well as all other standard logic
notions like e.g.\ subformula, sentence, universal closure,
derivation and so on, are then defined in the obvious way. In particular, if \( \upvarphi(\langle x_j \mid j \in J \rangle) \) is an \( \L_{\kappa^+ \kappa} \)-formula and \( \langle a_j \mid j \in J \rangle \in \pre{J}{X} \) is a sequence of elements of an \( \L \)-structure \( X \), then
\[ 
X \vDash \upvarphi[\langle a_j \mid j \in J \rangle]
 \] 
will have the usual meaning, i.e.\ that the formula obtained by replacing each variable \( x_j \) with the corresponding \( a_j \) is true in \( X \).

Each \( \L \)-structure \( X \) of size \( \kappa \) can be naturally
identified (up to isomorphism) with an element \( y^X = \seq{y^X_i}{i \in I} \) of \( \Mod^\kappa_\L = \prod_{i \in I}
\pre{({}^{n_i} \kappa)}{2} \): in fact, we can assume the domain of
\( X \) be \( \kappa \) itself, and then for every \( i \in I \)
 let \( y^X_i \) be the characteristic function of \( R^X_i \), i.e.\  \( y^X_i(s) = 1 \) if and only if \( X \vDash R_i[s] \)  (for every \( s \in \pre{n_i}{\kappa} \)).

By our convention, the space
\( \Mod^\kappa_\L \)  is endowed with the product of the bounded topologies on its factors \( \pre{({}^{n_i} \kappa)}{2} \) (this is possible because \( |\pre{n_i}{\kappa}| = \kappa \)): such topology is again called bounded topology and is denoted by \( \tau_b \). Notice that since by~\eqref{eqcardinalcondition} we have \( \kappa^\omega \leq \kappa^{< \kappa} = \kappa \), any bijection \( \nu \colon I \times
\kappa \to \kappa \) (together with the bijections \( \Hes_{n_i} \colon \pre{n_i}{\kappa} \to \kappa \))
provides a natural homeomorphism \label{homeomorphism} between
\( \Mod^\kappa_\L \) and \( \pre{\kappa}{2} \).

\begin{defin}\label{defModkappaphi}
Given an infinite cardinal \( \kappa \) and an
\( \L_{\kappa^+ \kappa} \)-sentence \( \upvarphi \), we denote by
\( \Mod_\upvarphi^\kappa \) the set of those \( \L \)-structures with domain
\( \kappa \) (i.e.\ of those elements of \( \Mod^\kappa_\L \)) which
satisfy \( \upvarphi \).
\end{defin}

A famous theorem of Lopez-Escobar (see e.g.\ \cite[Theorem
16.8]{kechris}) shows that \( B \subseteq \Mod^\omega_\L \) is Borel
and closed under isomorphism if and only if there is an
\( \L_{\omega_1 \omega} \)-sentence \( \upvarphi \) such that \( B =
\Mod^\omega_\upvarphi \). This result was later generalized by Vaught
to \( \omega_1 \) under the Continuum Hypothesis, and then further
generalized in \cite[Theorem 5.13]{andmot} and \cite[Theorem
24]{frihytkul} to an arbitrary infinite cardinal \( \kappa \)
satisfying~\eqref{eqcardinalcondition}. Summing up we have the
following:

\begin{theorem}\label{theorlopezescobar}
Assume \( \kappa^{< \kappa} = \kappa \). A set \( B \subseteq
\Mod^\kappa_\L \) is Borel and
closed under isomorphism if and only if there is an
\( \L_{\kappa^+ \kappa} \)-sentence \( \upvarphi \) such that \( B =
\Mod^\kappa_\upvarphi \).
\end{theorem}

\noindent
Since by Theorem~\ref{theorlopezescobar} the set \( \Mod^\kappa_\upvarphi \) is a Borel set, when considered as a topological space it will be endowed with the relative topology inherited
from \( \Mod^\kappa_\L \).

As observed in \cite[Remark 25]{frihytkul}, the direction right
to left of Theorem~\ref{theorlopezescobar} may
fail if we discard the assumption \( \kappa^{< \kappa} = \kappa \).
This is because in \cite[Corollary 17]{shevaa} it was
(essentially) shown that if \( \kappa = \lambda^+ \), \( \kappa^{<
\kappa} > \kappa \), \( \lambda^{< \lambda} = \lambda \) and a forcing
axiom holds (and \( \omega_1^L = \omega_1 \) if \( \lambda = \omega \))
then there is an \( \L_{\kappa \kappa} \)-sentence \( \upvarphi \) such that
\( \Mod^\kappa_\upvarphi \) does not have the Baire property, and hence
cannot be Borel by~\cite{halshe}.
After this observation, it was addressed as an open problem to
determine whether also the direction from left to right may fail
when \( \kappa^{< \kappa} > \kappa \). The answer to this question 
is positive, as it is shown 
by the counterexample below.

\begin{proposition}\label{propcounterexample}
Let \( \lambda < \kappa \) be arbitrary infinite cardinals with
\( \lambda \) regular, and let \( \L \) be the graph language. Then
there are \( 2^{2^\lambda} \)-many (\( \tau_b \)-)open subsets of \( \Mod^\kappa_\L \) closed under
isomorphism.
\end{proposition}

\begin{proof}
Let \( \L = \{ \mathbf{E} \} \), where \( \mathbf{E} \) is a binary relational symbol. Call an 
\( \L \)-structure \( X \) \emph{full} if for every \( a \in X \) there is \( b \in X \)
such that either \( a \mathrel{\mathbf{E}^X} b \) or \( b \mathrel{\mathbf{E}^X}
a \). First notice that since \(\lambda\) is regular there are
\( 2^\lambda \)-many full \( \L \)-structures \( \set{ X_r }{ r \in {}^\lambda 2 } \) with domain \(\lambda\) such that \( X_r \not\sqsubseteq X_s \)
whenever \( r \neq s \). This is given by \cite[Theorem 3.1]{louros}
applied to the identity relation on \( \pre{\omega}{2} \) if \( \lambda =
\omega \), and by \cite[Corollary 5.4]{bau} if \( \lambda \) is
uncountable (in the first case the \( X_r \)'s are combinatorial trees, i.e.\ connected acyclic graphs, while in the second case the \( X_r \)'s are linear orders).

Now for each \( r \in {}^\lambda 2 \), let \( t_r \in \Mod^\lambda_\L \) be a code for \( X_r \), i.e.\ \( t_r \) is the sequence in
\( {}^{\lambda \times \lambda} 2 \) defined by \( t_r(\alpha,\beta) = 1 \)
if and only if \( X_r \vDash \alpha \mathrel{\mathbf{E}} \beta \) (for \( \alpha,\beta <
\lambda \)). Let \( \mathrm{Sym}(\kappa) \) be the set of all permutations of \( \kappa \). Every \( p \in \mathrm{Sym}(\kappa) \) canonically induces the homeomorphism \( j_p \colon \Mod^\kappa_\L \to \Mod^\kappa_\L \) given by 
\[ 
j_p(X) \vDash \alpha \mathrel{\mathbf{E}} \beta \iff  X \vDash p^{-1}(\alpha) \mathrel{\mathbf{E}} p^{-1}(\beta) 
\] 
for every \( X \in \Mod^\kappa_\L \) and \( \alpha,\beta < \kappa \), i.e.\ \( j_p(X) \) is the unique \( Y \in \Mod^\kappa_\L \) such that \( p \) is an isomorphism from \( X \) to \( Y \). For every \( A \subseteq {}^\lambda 2 \) let 
\[ 
O_A = {\bigcup\nolimits_{\substack{r \in A \\ p \in \mathrm{Sym}(\kappa)}} (j_p `` \, \bN_{t_r})} \subseteq \Mod^\kappa_\L,
 \] 
i.e.\ \( O_A \) is the union over \( r \in A \) of the closure under isomorphism of the
basic open sets \( \bN_{t_r} \) of \( \Mod^\kappa_\L \). Clearly, \( O_A \) is closed under isomorphism, and is also
\( \tau_b \)-open because each \( j_p`` \, \bN_{t_r} \) is open.

We now want to show that if \( A \neq A' \subseteq \pow({}^\lambda
2) \) then \( O_A \neq O_{A'} \). For this it suffices to show that for
every \( r \in {}^\lambda 2 \) there is \( Y_r \in \Mod^\kappa_\L \)  which belongs to \( O_{\{ r \}} \setminus O_{\pre{\lambda}{2} \setminus \{ r \}} \) (this is because for every
\( A \subseteq {}^\lambda 2 \), \( O_A = \bigcup_{r \in A} O_{\{ r
\}} \)). First notice that for every \( r \in {}^\lambda 2 \), a
structure \( X \in \Mod^\kappa_\L \) belongs to \( O_{\{r \}} \) if and
only if there is a set \( B_{X,r} \subseteq \kappa \) of size \( \lambda \)
such that \( X \restriction B_{X,r} \cong X_r \). Let now \( Y_r \in \Mod^\kappa_\L \) be defined by \( Y_r \vDash \alpha
\mathrel{\mathbf{E}} \beta \) (for \( \alpha,\beta < \kappa \)) if and only if \( \alpha,\beta
< \lambda \) and \( X_r \vDash \alpha \mathrel{\mathbf{E}} \beta \). Clearly \( Y_r \in
O_{\{ r \}} \): we claim that \( Y_r \notin O_{ \{ s \} } \) for every \( s \in \pre{\lambda}{2} \) such that \( s
\neq r \). Suppose otherwise: then there would be \( B_{Y_r,s} \subseteq \kappa \)
such that \( Y_r \restriction B_{Y_r,s} \cong X_s \). Since \( X_s \) is full, \( B_{Y_r,s} \subseteq \lambda \) by
the definition of \( Y_r \). But this would imply that \( X_s \sqsubseteq {Y_r \restriction \lambda} \cong X_r \), contradicting the choice of the \( X_r \)'s.
\end{proof}

This immediately gives the following result.

\begin{theorem}\label{theorcounterexample}
Assume \( \lambda< \kappa \) are infinite cardinals with \(\lambda\) regular, and let \( \L \) be the graph
language. If \( 2^{2^{\lambda}} > 2^\kappa \) then there is a (\( \tau_b \)-)open
 subset of \( \Mod^\kappa_\L \) closed under isomorphism  which is
not of the form \( \Mod^\kappa_\upvarphi \) for \( \upvarphi \in \L_{\kappa^+
\kappa} \).
\end{theorem}

\begin{proof}
By a cardinality argument, using Proposition~\ref{propcounterexample} and~\eqref{eqlogic}.
\end{proof}

\begin{corollary}\label{corcounterexample}
Let \( \kappa \) be an infinite cardinal such that \( \kappa^{< \kappa} > \kappa \). If \( \kappa \) is the successor of a regular cardinal \(\lambda\) and either \( 2^\lambda = 2^\kappa \) or \( 2^{\kappa^+} > 2^\kappa \) (which is the case if e.g.\ \( \mathsf{GCH} \) holds at \( \kappa \)), then there is a (\( \tau_b \)-)open
 subset of \( \Mod^\kappa_\L \) closed under isomorphism  which is
not of the form \( \Mod^\kappa_\upvarphi \) for \( \upvarphi \in \L_{\kappa^+
\kappa} \).
\end{corollary}

\begin{proof}
By Theorem~\ref{theorcounterexample}, it is enough to show that under our assumptions it holds \( 2^{2^\lambda} > 2^\kappa \). If \( 2^\lambda = 2^\kappa \) this is obvious, so let us consider the other case. Since \( \kappa \) is a successor cardinal it is also regular, hence \( \kappa^{< \kappa} > \kappa \) is equivalent to \( 2^{< \kappa } > \kappa \). This implies \( 2^\lambda > \kappa \), whence \( 2^{2^\lambda} \geq 2^{\kappa^+} > 2^\kappa \), as required.
\end{proof}

Of course there are many situations in which the hypotheses of Corollary~\ref{corcounterexample} are satisfied. One significant example concerning the first uncountable cardinal is presented in the following corollary.

\begin{corollary}
Assume \( 2^{\aleph_0} = 2^{\aleph_1} \), and let \( \L \) be the graph
language. Then there is a (\( \tau_b \)-)open subset of
\( \Mod^{\omega_1}_\L \) closed under isomorphism which is not of
the form \( \Mod^{\omega_1}_\upvarphi \) for \( \upvarphi \in \L_{\omega_2
\omega_1} \).
\end{corollary}

Notice that the hypothesis \( 2^{\aleph_0} = 2^{\aleph_1} \) is satisfied in every model of 
\( \mathsf{MA}_{\aleph_1} \) (which follows e.g.\ from various forcing axioms, like \( \mathsf{PFA} \)) or
in any model where \( \mathsf{CH} \) fails but \( \mathsf{GCH} \) holds at \( \aleph_1 \).

\section{Weakly compact cardinals}\label{sectionweaklycompact}

Recall that, given cardinals \( \nu , \lambda \leq \kappa \) and \( n
\in \omega \), the notation\footnote{See e.g.\ \cite[p.\ 109]{jech}.} \(  \kappa \to (\lambda)^n_\nu \) means
that for every function \( f \colon [\kappa]^n \to \nu \) there is
an \emph{\( f \)-homogeneous} set \( X \subseteq \kappa \) of cardinality
\( \lambda \), i.e.\ a set \( X \) such that \( f \) restricted to \( [X]^n \)
is constant.

\begin{defin}\label{defweaklycompact}
(see~\cite[Definition 9.8]{jech}.)
An uncountable cardinal \( \kappa \) is called \emph{weakly compact}
if \( \kappa \to (\kappa)^2_2 \).
\end{defin}

The definition above turns out to be equivalent to many other
conditions --- in fact the name ``weakly compact'' comes from condition~\ref{propwk3}) of Proposition~\ref{propwk}. Call  \emph{\( \kappa \)-tree} any DST-tree \( T \subseteq \pre{< \kappa}{\kappa} \) of height
\( \kappa \) such that each of its levels \( L_\gamma(T) = \{ s \in T \mid \leng{s} = \gamma \} \) (for \( \gamma < \kappa \)) has size \( < \kappa \). 

\begin{defin}\label{deftreeproperty}
(see~\cite[p.\ 120]{jech}.)
Given an uncountable
\emph{regular}\footnote{It is immediate to check that if
\( \kappa \) is singular then there are \( \kappa \)-trees without
cofinal branches.} cardinal \( \kappa \), we say that \( \kappa \) has the
\emph{tree property} if every \( \kappa \)-tree has a cofinal
branch, i.e.\ a branch of length \( \kappa \).
\end{defin}

\begin{defin}\label{definaccessible}
(see~\cite[p.\ 58]{jech}.)
An uncountable cardinal is called \emph{(strongly) inaccessible} if it is regular and strong limit, i.e.\ \( 2^\lambda < \kappa \) for every \( \lambda < \kappa \).
\end{defin}

\begin{proposition} \label{propwk} 
\emph{(see~\cite[Theorem 7.8]{kan} and the comment following it.)}
Let \( \kappa \) be an uncountable\footnote{Conditions~\ref{propwk2})--\ref{propwk4}) are all true if \( \kappa = \omega \), and condition~\ref{propwk1}) is
true as well in the sense that by Ramsey theorem \( \omega \to
(\omega)^2_2 \).} cardinal. The following are equivalent:
\begin{enumerate}[i)]
\item \label{propwk1}
 \( \kappa \) is weakly compact;
\item \label{propwk2}
\( \kappa \to (\kappa)^n_\nu \) for every \( \nu < \kappa \) and
\( n \in \omega \);
\item \label{propwk3}
the logic \( \L_{\kappa \kappa} \) is compact;
\item \label{propwk4}
\( \kappa \) has the tree property and is 
inaccessible.
\end{enumerate}
\end{proposition}

Notice that, in particular, if \( \kappa \) is weakly compact then~\eqref{eqcardinalcondition} is automatically
satisfied by the fact that \( \kappa \) is inaccessible.
We now want to connect the equivalent conditions above with a
topological property of the space \( \pre{\kappa}{2} \) which will be
used later.

\begin{defin}
Let \( \kappa \) be an infinite cardinal. We say that a topological
space \( X \) is \emph{\( \kappa \)-compact} (or
\emph{\( \kappa \)-Lindel\"of}) if every open covering of \( X \) has a
subcovering of size \( < \kappa \).
Given an arbitrary topological space \( X \), a subset  \( K \subseteq
X \) is \emph{\( \kappa \)-compact} if \( K \) endowed with the relative topology
inherited from \( X \) is a \( \kappa \)-compact space.
\end{defin}

In particular, a space is compact if and only if it is
\( \omega \)-compact. It follows immediately from the definition
that any partition in nonempty (cl)open sets of a
\( \kappa \)-compact space \( X \) must have size \( < \kappa \).
The next theorem gives further support to the name ``weakly
compact'' given to the cardinals described in Definition~\ref{defweaklycompact}.

\begin{theorem}
Let \( \kappa \) be an uncountable cardinal. The space \( \pre{\kappa}{2} \)
(endowed with the bounded topology\footnote{Notice that for
every cardinal \( \kappa \), the space \( \pre{\kappa}{2} \) endowed with
the product topology \( \tau_p \) is always trivially
(\( \omega \)-)compact by Tychonoff's theorem.} \( \tau_b \)) is \( \kappa \)-compact if and only
if \( \kappa \) is weakly compact.
\end{theorem}

\begin{proof}
Assume that \( \pre{\kappa}{2} \) is \( \kappa \)-compact:
we will show that \( \kappa \) is inaccessible and has the tree
property. First notice that \( \kappa \) must be regular, for otherwise let \( \lambda =
\cof(\kappa) < \kappa \), let \( \langle \nu_i \mid i < \lambda \rangle \) be a cofinal sequence in \( \kappa \),  and consider the set \( S = \left\{ 0^{(i)}
{}^\smallfrown{}  1 {}^\smallfrown{}  t  \mid  i < \lambda, t \in
{}^{\nu_i} 2 \right\} \cup \left\{ 0^{(\lambda)} \right\} \): then \( \set{ \bN_s }{s \in
S } \) is a partition into \( \geq \kappa \)-many nonempty clopen subsets of
\( \pre{\kappa}{2} \), contradicting its
\( \kappa \)-compactness. 
That \( \kappa \) is a strong limit is immediate, as if \( \lambda < \kappa \) is such
that \( 2^\lambda \geq \kappa \), then \( \left\{ \bN_s  \mid  s \in
{}^\lambda 2 \right\} \) is a partition of \( \pre{\kappa}{2} \) in
(at least) \( \kappa \)-many clopen sets, which is again a
contradiction.

It is straightforward to check that if \( \kappa \) is regular,
then \( \kappa \) has the tree property if and only if there is no
\( \kappa \)-tree \( T \subseteq {}^{< \kappa} 2 \) without a
cofinal branch, and that if \( \kappa \) is inaccessible this last statement is equivalent to the fact that each
DST-tree \( T \subseteq {}^{< \kappa} 2 \) of height \( \kappa \) has a cofinal branch: therefore, since we already proved that \( \kappa \) is inaccessible, in order to show that \( \kappa \) has the tree property it is enough to show that the latter property holds for our \( \kappa \).	
Given a DST-tree \( T \subseteq {}^{< \kappa}2 \) of height \( \kappa \), let
\[
L(T) = \set{ s \in {}^{<\kappa} 2 }{{s \notin T} \wedge \forall
\lambda < \leng{s} \left(s \restriction \lambda \in T\right) }. 
\]
Assume towards a contradiction that \( T \) has no
cofinal branch. Then \( \mathcal{P} = \set{ \bN_s }{ s \in L(T) } \) is a
partition into nonempty clopen sets of \( \pre{\kappa}{2} \).
By \( \kappa \)-compactness, \( \mathcal{P} \) must have size \( \lambda \) for some
\( \lambda < \kappa \), hence \( |L(T)| = \lambda \)  as well. Let \( b \subseteq T \) be a branch of \( T \). Since \( b \) cannot be cofinal, then \( \leng{\bigcup b} < \kappa \) and \( \bigcup b \) admits an extension \( s_b \in L(T) \). Since the map \( b \mapsto s_b \) is injective, \( |L(T)| < \kappa \), and \( \kappa \) is regular, this implies that there is \( \nu < \kappa \) such that \( \leng{\bigcup b} < \nu \) for every branch \( b \) of \( T \). This implies that
the DST-tree \( T \) has height \(  \leq \nu < \kappa \), contradicting the choice of \( T \).

We will now show that if \( \kappa \) is inaccessible and has the
tree property then \( \pre{\kappa}{2} \) is \( \kappa \)-compact.
Let \( \mathcal{U} \) be an open covering of \( \pre{\kappa}{2} \), and define the set \( S(\mathcal{U}) \subseteq {}^{<
\kappa} 2 \) by letting \( s \in S(\mathcal{U}) \) if and only if
\( \bN_s \) is contained in some \( U \in \mathcal{U} \) and has no
proper initial segment with such a property. Notice that, in
particular, \( \set{ \bN_s }{ s \in S(\mathcal{U}) } \) is a clopen
covering of \( \pre{\kappa}{2} \) (in fact it is a clopen
partition of \( \pre{\kappa}{2} \)). Let now  \( T(\mathcal{U})
\subseteq {}^{< \kappa} 2 \) be the DST-tree generated by
\( S(\mathcal{U}) \), i.e.\ 
\[ 
T(\mathcal{U}) = \set{ t \in {}^{< \kappa}
2 }{t \subseteq s \text{ for some }s \in S(\mathcal{U}) }.
\]
Since \( T(\mathcal{U}) \) has no cofinal branch, by the fact that
\( \kappa \) is inaccessible and has the tree property we have that the
height of \( T(\mathcal{U}) \) must be \( \leq \lambda \) for some
\( \lambda < \kappa \), i.e.\ \( T(\mathcal{U}) \subseteq
{}^\lambda 2 \) for such a  \( \lambda < \kappa \). Therefore, by inaccessibility of \( \kappa \) again,
\( |T(\mathcal{U})| < \kappa \), and hence also \( |S(\mathcal{U})| < \kappa \). For each \( s \in S(\mathcal{U}) \), let \( U_s \in
\mathcal{U} \) be such that \( \bN_s \subseteq U_s \): then \( \set{ U_s
}{s \in S(\mathcal{U}) } \) is an open subcovering of
\( \mathcal{U} \) of size \( < \kappa \).
\end{proof}

The next two lemmas generalize some facts which are well-known for
the countable case, and can be proved using the same standard
arguments.

\begin{lemma}
Let \( X,Y \) be topological spaces, and \( \kappa \) be an infinite
cardinal. If \( K \subseteq X \) is \( \kappa \)-compact and \( f \colon X
\to Y \) is continuous, then \( f(K) \) is \( \kappa \)-compact as well.
\end{lemma}

\begin{lemma}
Let \( X \) be a topological space, and \( \kappa \) be an infinite
cardinal. If \( X \) is \( \kappa \)-compact and \( C \subseteq X \) is
closed then \( C \) is \( \kappa \)-compact as well.
Conversely, if \( X \) is a Hausdorff space admitting a basis \( \mathcal{B} \) such that the
intersection of less than \( \kappa \)-many basic open sets is still open, then
every \( \kappa \)-compact set \( K \subseteq X \) is closed.
\end{lemma}

\begin{corollary}\label{corclosed}
Let \( \kappa \) be a weakly compact cardinal, and \( f \colon
\pre{\kappa}{2} \to \pre{I}{2} \) be a continuous function (where \( |I| = \kappa \) and  both \( \pre{\kappa}{2} \) and \( \pre{I}{2} \) are endowed with the bounded topology). Then \( C \) is a closed subset of \( \pre{\kappa}{2} \) if and only if \( f`` \, C \) is closed in \( \pre{I}{2} \).
\end{corollary}

\section{Analytic quasi-orders and equivalence relations}\label{sectionanalytic}

\begin{defin}\label{defanalyticqo}
Let \( X \) be a standard Borel \( \kappa \)-space. We say that \( R
\subseteq X \times X \) is an \emph{analytic quasi-order on \( X \)}
if it is a quasi-order (i.e.\ a reflexive and transitive binary
relation) which is analytic as a subset of \( X \times X \).
If moreover \( R \) is also symmetric (i.e.\ an equivalence relation
on \( X \)), we call \( R \) an \emph{analytic equivalence relation}.
\end{defin}

Given the analytic quasi-order \( R \), we denote by \( E_R \) the
analytic equivalence relation canonically induced by \( R \), i.e.\
\( E_R = R \cap R^{-1} \).

\begin{example}
Let \( \L \) be a countable relational language. Then, as observed immediately before Definition~\ref{defModkappaphi}, \( \Mod^\kappa_\L \) is homeomorphic to \( \pre{\kappa}{2} \). Consider now an \( \L_{\kappa^+ \kappa} \)-sentence \( \upvarphi \). By Theorem~\ref{theorlopezescobar}, \( \Mod^\kappa_\upvarphi \) is a Borel subset of \( \Mod^\kappa_\L \), and hence a standard Borel \( \kappa \)-space. Then it is easy to check that:
\begin{enumerate}[(1)]
\item
the relation \( \sqsubseteq \) of embeddability on \( \Mod^\kappa_\upvarphi \) is an analytic 
quasi-order. The analytic equivalence relation \(E_\sqsubseteq \) canonically induced by \( \sqsubseteq \) is the bi-embeddability relation \( \equiv \);
\item
the relation \( \cong \) of isomorphism on \( \Mod^\kappa_\upvarphi \) is an analytic equivalence relation.
\end{enumerate}
\end{example}

\begin{defin} \label{defborelreducibility}
Let \( X,Y \) be two standard Borel \( \kappa \)-spaces, and \( R \) and \( S \) be analytic quasi-orders on \( X \) and \( Y \),
respectively. A function \( f \colon X \to Y \) is called a \emph{reduction} of \( R \) to \( S \) if and only if
\[ \forall x_0,x_1 \in X \left (x_0 \mathrel{R} x_1 \iff f(x_0) \mathrel{S}
f(x_1) \right). \]
If such a reduction exists we say that \( R \) is \emph{reducible} to \( S \).

Moreover, we say that \( R \) is  \emph{Borel reducible} to \( S \) (\( R
\leq_B S \) in symbols) if there is a Borel reduction of \( R \) to \( S \), and that \( R \) and \( S \) are \emph{Borel bi-reducible} (\( R \sim_B
S \) in symbols) if both \( R \leq_B S \) and \( S \leq_B R \).
\end{defin}

\begin{defin} \label{defcomplete}
An analytic quasi-order  \( S \) is called \emph{complete} if \( R \leq_B S \)  for every analytic 
quasi-order \( R \).
Similarly, an analytic equivalence relation \( F \) is \emph{complete} if \( E \leq_B F \) for every analytic equivalence relation \( E \).
\end{defin}

Intuitively, \( R \leq_B S \) means that \( R \) is not more complicated than \( S \). So an analytic 
quasi-order (or an analytic equivalence relation) is complete when it is as complicated as possible. Notice also that if a quasi-order \( R \) is complete, then the induced equivalence relation \( E_R \) is complete as well. 

A natural strengthening of the notion of completeness, called invariant universality in~\cite{cammarmot}, was implicitly introduced in~\cite{frimot}. Here we present its natural generalization to the case of embeddability between models of size \( \kappa \).

\begin{definition} \label{definvariantlyuniversal}
Let \( \kappa \) be an infinite cardinal, \( \L \) a countable relational language, and \( \upvarphi \) an \( \L_{\kappa^+ \kappa} \)-sentence. The relation of embeddability on \( \Mod^\kappa_\upvarphi \) is called \emph{invariantly universal} if for every analytic quasi-order \( R \) there is an \( \L_{\kappa^+ \kappa} \)-sentence \( \uppsi \) such that \( \Mod^\kappa_\uppsi \subseteq \Mod^\kappa_\upvarphi \) and \( R \sim_B {\sqsubseteq \restriction \Mod^\kappa_\uppsi} \).

Similarly, the relation of bi-embeddability on \( \Mod^\kappa_\upvarphi \) is called \emph{invariantly universal} if for every analytic equivalence relation \( E \) there is an 
\( \L_{\kappa^+ \kappa} \)-sentence \( \uppsi \) such that \( \Mod^\kappa_\uppsi \subseteq \Mod^\kappa_\upvarphi \) and \( E \sim_B {\equiv \restriction \Mod^\kappa_\uppsi} \).
\end{definition}

It follows from Definition~\ref{definvariantlyuniversal} that if \( \sqsubseteq \restriction \Mod^\kappa_\upvarphi \) is invariantly universal then it is also complete.

Note that if \( R \) and \( S \) are two analytic quasi-orders on, respectively, the standard Borel 
\( \kappa \)-spaces \( X \) and \( Y \), then \( R \leq_B S \) if and only if
there is an embedding \( f \colon X/_{E_R} \to Y/_{E_S} \) between
the partial orders induced by \( R \) and \( S \) on the quotient spaces
which admits a Borel lifting. From this it follows that \( R \sim_B
S \) if and only if the partial orders induced by \( R \)
and \( S \) on the quotient spaces \( X/_{E_R} \) and \( Y/_{E_S} \) are one
embeddable into the other via two functions
which admit Borel liftings. It is therefore very natural to
strengthen this last notion by substituting bi-embeddability with
isomorphism.

\begin{defin} \label{defBorelisomorphic}
Let \( X,Y \) be two standard Borel \( \kappa \)-spaces, and let \( R \) and \( S \) be analytic quasi-orders on, respectively,
\( X \) and \( Y \). We say that \( R \) and \( S \) are \emph{classwise Borel
isomorphic} (\( R \simeq_B S \) in symbols) if there is an
isomorphism \( f \colon X/_{E_R} \to Y/_{E_S} \) between the partial
orders induced by \( R \) and \( S \) on the corresponding quotient spaces such that
both \( f \) and \( f^{-1} \) admit Borel liftings.
\end{defin}

Thus, two classwise Borel isomrphic analytic quasi-order can  be reasonably identified from both the topological and the combinatorial point of view. Moreover,
Definition~\ref{defBorelisomorphic} suggests a further strengthening of the notion of invariant universality.

\begin{definition} \label{defstronginvariantlyuniversal}
Let \( \kappa \) be an infinite cardinal, \( \L \) be a countable relational language, and \( \upvarphi \) an \( L_{\kappa^+ \kappa} \)-sentence. The relation \( \sqsubseteq \restriction \Mod^\kappa_\upvarphi \) (respectively, \( \equiv \restriction \Mod^\kappa_\upvarphi \)) is called \emph{strongly invariantly universal} if for every analytic quasi-order (resp.\ equivalence relation) \( R \) there is an \( \L_{\kappa^+ \kappa} \)-sentence \( \uppsi \) such that \( \Mod^\kappa_\uppsi \subseteq \Mod^\kappa_\upvarphi \) and \( R \simeq_B {\sqsubseteq \restriction \Mod^\kappa_\uppsi} \) (resp.\ \( R \simeq_B {\equiv \restriction \Mod^\kappa_\uppsi}\)).
\end{definition}

As for the case of completeness, notice that if \( \upvarphi \) is an \( \L_{\kappa^+ \kappa} \)-sentence such that the relation of embeddability on \( \Mod^\kappa_\upvarphi \) is (strongly) invariantly universal, then \( \equiv \) on \( \Mod^\kappa_\upvarphi \) is (strongly) invariantly universal as well.
Moreover, if \( \sqsubseteq \restriction \Mod^\kappa_\upvarphi \) is strongly invariantly universal then it is in particular invariantly universal, and hence also complete.

In the countable case, when one is interested in analytic
quasi-orders and equivalence relations up to Borel reducibility,
he or she usually restrict the attention to some specific Polish space,
like e.g.\ \( \pre{\omega}{2} \): this is beacuse any uncountable Borel subset 
\( B \) of a standard Borel space is Borel isomorphic to \( \pre{\omega}{2} \) via some \( f
\colon B \to \pre{\omega}{2} \), so that one can ``copy'' on \( \pre{\omega}{2} \) any analytic
quasi-order \( R \) defined on \( B \) by letting (for \( x,y \in \pre{\omega}{2}\))
\[ 
x \mathrel{S} y \iff f^{-1}(x) \mathrel{R} f^{-1}(y), 
\]
and obviously get that \( R \simeq_B S \). However, this trick may fail in the uncountable case: for example, as already observed in Remark~\ref{remkurepa} and the observation preceding it, if there
is a \( \kappa \)-Kurepa tree then there is a closed subset of
\( \pre{\kappa}{2} \) (the body of the \( \kappa \)-Kurepa tree) which cannot be Borel isomorphic to \( \pre{\kappa}{2} \) because of cardinality considerations.
However, we can still prove the following.

\begin{lemma}\label{lemmaquasiorderisomorphic}
If \( \kappa \) is an infinite cardinal satisfying \( \kappa^{< \kappa} = \kappa \), then every analytic quasi-order (resp.\
analytic equivalence relation) \( R \) is classwise Borel isomorphic
to an analytic quasi-order (resp.\ analytic equivalence
relation) on \( \pre{\kappa}{2} \).
\end{lemma}

\begin{proof}
By Definition~\ref{defstandard} and the observation following Proposition~\ref{propanalytic}, we can
clearly assume that \( R \) is defined on a Borel set \( B \subseteq
\pre{\kappa}{2} \). Now pick any \( x_0 \in B \) and set (for \( x,y \in
\pre{\kappa}{2} \))
\begin{align*}
x \mathrel{S} y \iff & ({x,y \in B} \wedge x \mathrel{R} y) \vee ({x
\notin B} \wedge {y \in B} \wedge {x_0 \mathrel{R} y}) \\
& \vee ({x \in B} \wedge {y \notin B} \wedge x \mathrel{R} x_0) \vee
(x,y \notin B).
 \end{align*}
Clearly \( S \) is an analytic quasi-order on \( \pre{\kappa}{2} \). Now
define \( \hat{f} \colon \pre{\kappa}{2} \to B \) by setting \( \hat{f}(x) = x \) if
\( x \in B \) and \( \hat{f}(x) = x_0 \) otherwise, and let \( \hat{g} \)
be the identity function on \( B \). Clearly, \( \hat{f} \) and
\( \hat{g} \) are Borel maps which reduce, respectively, \( S \) to \( R \)
and \( R \) to \( S \). If we now let \( f \) be the map induced by
\( \hat{f} \) on the quotient space \( \pre{\kappa}{2} /_{E_S} \), one can
easily check that \( \hat{f} \) and \( \hat{g} \) are Borel liftings of
\( f \) and \( f^{-1} \), respectively.
\end{proof}

Since in what follows we will be interested in analytic quasi-orders and equivalence relations up to classwise Borel isomorphism, by Lemma~\ref{lemmaquasiorderisomorphic} we can without loss of generality restrict our attention to the special case of quasi-orders defined on \( \pre{\kappa}{2} \). Recall also that by the observation at the end of Section~\ref{sectionspaces} a quasi-order \( R \) on \( \pre{\kappa}{2} \) is analytic if and only if \( R = \proj [T] \) for some DST-tree on \( 2 \times 2 \times \kappa \) of height \( \kappa \): for this reason, from now on  when we will write \( R = \proj [T] \) (for some analytic quasi-order \( R \)) we will tacitly assume that \( T \) is such a DST-tree.

\section{The quasi-order $\leq_{\max}$}\label{sectionmax}

Let \( \rho \colon \On \times \On \to \On \setminus \{ 0 \}
\colon (\alpha, \beta) \mapsto \Hes_2(\alpha,
\beta) + 1 \), so that in particular \( \rho `` \, \omega \times \omega =
\omega \setminus \{ 0 \} \), and for every infinite cardinal
\( \kappa \) and \( \alpha, \beta < \kappa \) one has \( \alpha,
\beta < \rho(\alpha, \beta) < \kappa \). 

Fix an uncountable cardinal \( \kappa \).
We define by recursion on \( \gamma \leq \kappa \) a Lipschitz (i.e.\ a monotone and length preserving) map
\( \oplus \colon {}^{\leq \kappa} (\kappa \times \kappa) \to
{}^{\leq \kappa} \kappa \) as follows:
\begin{description}
\item[\( \gamma = 0 \)] \( \emptyset \oplus \emptyset = \emptyset \);
\item[\( \gamma = 1 \)] if \( s = \langle \alpha \rangle \) and \( t = \langle \beta \rangle \), then \( s \oplus t = \langle \rho(\alpha,\beta) \rangle \);
\item[\( \gamma = \gamma'+1 \) for \( \gamma' \neq 0 \)] let \( s , t \in {}^{\gamma} \kappa \). Then 
\[ s \oplus t = (s' \oplus t') {}^\smallfrown{}  \rho\left( \sup_{\alpha \leq \gamma'}s(\alpha) + \omega, \rho(s(\gamma'),t(\gamma'))\right), \]
where \( s' = s \restriction \gamma' \) and \( t' = t \restriction \gamma' \);
\item[\( \gamma \) limit] \( s \oplus t = \bigcup_{\alpha < \gamma} (s \restriction \alpha \oplus t \restriction \alpha ) \).
\end{description}

The relevant properties of the map \( \oplus \) are summarized in the following proposition.

\begin{proposition}\label{propoplus}
Let \( \kappa \) be an uncountable cardinal, and let \( \oplus \) be defined as above.
\begin{enumerate}[(1)]
\item \label{propoplusc1}
The map \( \oplus \) is injective.
\item \label{propoplusc2}
If \( \kappa \) is inaccessible, then there is a map \( \# \colon \pre{\SUCC(< \kappa)}{\kappa} \to \kappa \) such that
\begin{enumerate}[(a)]
\item \label{propoplusc2a}
for every \( s,t \in \pre{\SUCC(<\kappa)}{\kappa} \) such that \( \leng{s} = \leng{t} \)
\[ 
\# s \leq \#(s \oplus t);
 \] 
\item \label{propoplusc2b}
for every \( \gamma < \kappa \), \( \# \restriction \pre{\gamma+1}{\kappa} \) is a bijection between \( \pre{\gamma+1}{\kappa} \) and \( \kappa \).
\end{enumerate}
\end{enumerate}
\end{proposition}

\begin{proof}
Part~(\ref{propoplusc1}) follows from the injectivity of \( \rho \).
For part~(\ref{propoplusc2}), we will define the function \( \# \) separately on each \( \pre{\gamma+1}{\kappa} \) (for \( \gamma < \kappa \)).

The case \( \gamma = 0 \) is trivial, as one can simply take \( \# \restriction \pre{1}{\kappa} \) to be the map sending each sequence of length \( 1 \) to its unique value, so let us assume that \( \gamma \geq 1 \).
Fix any bijection \( \vartheta \colon {}^{\gamma+1}\kappa \to \kappa \) and set
\[ 
s \lhd  t \iff \sup_{\alpha \leq \gamma} s(\alpha) < \sup_{\alpha \leq \gamma} t(\alpha) \vee \left(\sup_{\alpha \leq \gamma} s(\alpha) = \sup_{\alpha \leq \gamma} t(\alpha) \wedge \vartheta(s) < \vartheta(t)\right). 
\]
Clearly \( \lhd \) is a well-founded strict linear order on \( {}^{\gamma+1} \kappa \): we claim that it has order type \( \kappa \). Since \( \sup_{\alpha \leq \gamma} s(\alpha) < \kappa \)  for every \( s \in {}^{\gamma+1} \kappa \) (by regularity of \( \kappa \)), \( \lhd \) is the sum of the \( \kappa \)-many 
well-founded quasi-orders \( \lhd_\beta \) (\( \beta < \kappa \)), where \( \lhd_\beta \) is the restriction of \( \lhd \) to those \( s \in {}^{\gamma+1}\kappa \) such that \( \sup_{\alpha \leq \gamma} s(\alpha) = \beta \). Obviously, for every \( \beta < \kappa \) there are less than 
\( \lambda^\lambda = 2^\lambda \)-many such sequences (where \( \lambda = |\max\{ \beta + 1, \gamma+1 \}| \)), so by inaccessibility of \( \kappa \) each \( \lhd_\beta \) has order type \( < \kappa \). This implies that the order type of \( \lhd \) is \( \leq \kappa \). On the other hand, \( | {}^{\gamma+1} \kappa | = \kappa \), hence \( \lhd \) has order type \( \kappa \).

Now let \( \# \restriction \pre{\gamma + 1}{\kappa} \) be the function enumerating \( {}^{\gamma+1} \kappa \) with respect to the order \( \lhd \), i.e.\ let \( \#(s) = \delta \) if and only if \( s \) is the \( \delta \)-th element of \( {}^{\gamma+1}\kappa \) with respect to \( \lhd \). Since for every \( s,t \in {}^{\gamma+1}\kappa \)
\[ \sup_{\alpha \leq \gamma } s(\alpha) < \rho\left(\sup_{\alpha \leq \gamma } s(\alpha) + \omega, \rho(s(\gamma),t(\gamma))\right) = (s \oplus t)(\gamma) \leq \sup_{\alpha \leq \gamma}(s \oplus t)(\alpha),\]
we get that \( s \lhd s \oplus t \), hence \( \# s \leq \#(s \oplus t) \), as required.
\end{proof}

The following construction has been essentially introduced (in a simplified form) in~\cite{andmot} and 
is based on the proof of~\cite[Theorem 2.4]{louros}. 
 For \( \emptyset \neq u \in \pre{<
\kappa}{2} \) set \( u^- = u \restriction \leng{u}-1 \) if \( \leng{u} \) is finite
and \( u^- = u \) otherwise. 
Moreover, consider the variant \( \widetilde{\oplus} \colon  \pre{ \leq \kappa}{(\kappa \times \kappa)} \to \pre{\leq \kappa}{\kappa} \) of \( \oplus \) defined by
\[ 
s \mathrel{\widetilde{\oplus}} t = \langle ({0 {}^\smallfrown{} s} \oplus {0 {}^\smallfrown{} t})(1+\gamma) \mid \gamma < \leng{s}. \rangle
 \] 
Then \( \widetilde{\oplus} \) is an injective Lipschitz map (since \( \oplus \) is such a function), and it is straightforward  to check that for every \( n,m \in \omega \) and \( s,t \in \pre{\leq \kappa}{\kappa} \) it holds
\begin{equation} \label{eq:variantoplus}
{( n {}^\smallfrown{} s )} \oplus {( m {}^\smallfrown{}  t )} = \rho(n,m) {}^\smallfrown{}  ( s \mathrel{\widetilde{\oplus}} t ).
\end{equation}

Given a DST-tree \( T \) on \( 2 \times 2 \times \kappa \) of
height \(  \leq \kappa \), let \( \hat{T} = T \cup \set{ (u,u,s) \in
{}^{< \kappa} 2 \times {}^{< \kappa} 2 \times {}^{< \kappa}
\kappa }{\leng{u} = \leng{s} } \).
Then inductively define \( S^T_n \) as
follows:
\begin{align*}
S^T_0 = & \{ \emptyset \} \cup \left \{ (u,v,0 {}^\smallfrown{}  s)  \mid  (u^-,v^-,s) \in
\hat{T} \right\}\\
S^T_{n+1} = & \{ \emptyset \} \cup \left\{ (u,v,(n+1) {}^\smallfrown{}  s)  \mid  (u,v,n
{}^\smallfrown{}  s) \in S^T_n \right\} \cup \\
& \cup \left\{ (u,w,(n+1) {}^\smallfrown{}  s \mathrel{\widetilde{\oplus}} t)  \mid  \exists v
\left[(u,v,n {}^\smallfrown{}  s),(v,w,n {}^\smallfrown{}  t) \
\in S^T_n \right] \right\} .
\end{align*}
Finally, set 
\begin{equation} \label{eqS_T}
S_T = \left ( \bigcup\nolimits_n S^T_n \right )\setminus \left\{ (u,v,0^{(\leng{u})})
\in \pre{< \kappa}{2} \times \pre{< \kappa}{2} \times \pre{< \kappa}
{\kappa}  \mid  u \neq v \right\}. 
\end{equation}
Notice that \( \bigcup_n S^T_n \) and  \( S_T \) are  DST-trees on
\( 2 \times 2 \times \kappa \) (because an easy induction on \( n \in
\omega \) shows that each \( S^T_n \) is a DST-tree on the same space) of
height \( \leq \kappa \), and that if \( (u,v,s) \in S_T \setminus \{ \emptyset \} \) then \( s(0) \in \omega \).

\begin{lemma}\label{lemmanormalform}
Let \( \kappa \) be a weakly compact cardinal, and \( R =
\proj[T] \) be an analytic quasi-order. Then the following conditions hold:
\begin{enumerate}[i)]
\item \label{lemmanormalformc1}
\( R = \proj[S_T] = \set{ (x,y) \in \pre{\kappa}{2} \times
\pre{\kappa}{2} }{ \exists \xi \in \pre{\kappa}{\kappa} \left
((x,y,\xi) \in [S_T] \right) } \);
\item \label{lemmanormalformc2}
for every \( u \in {}^{< \kappa} 2 \) and \( s \in \pre{< \kappa}{\kappa}  \) such that \( \leng{u} = \leng{s} > 0 \) and
\( s(0) \in \omega \), \( (u,u, s) \in S_T \);
\item \label{lemmanormalformc3}
if \( u,v,w \in {}^{< \kappa} 2 \) and \( s,t
\in {}^{< \kappa} \kappa \) are such that \( (u,v,s) , (v,w,t) \in S_T \) then \( (u,w,s \oplus t) \in S_T \);
\item\label{lemmanormalformc4} 
if \( (u,v,0^{(\leng{u})}) \in S_T \) then \( u=v \).
\end{enumerate}
\end{lemma}

\begin{proof}
The DST-tree \( S_T \) defined in~\eqref{eqS_T} clearly satisfies~\ref{lemmanormalformc2}) and~\ref{lemmanormalformc4}). To see that it also satisfies~\ref{lemmanormalformc3}), observe that we can assume that both \( s \) and \( t \) are nonempty, and hence let \( s',t' \in \pre{< \kappa}{\kappa} \) and \( n,m \in \omega \) be such that \( s = n {}^\smallfrown{}  s' \) and \( t = m {}^\smallfrown{}  t' \). Notice that if \( (u,v,i {}^\smallfrown{}  r) \in S_T \) then \( (u,v,i
{}^\smallfrown{}  r) \in S^T_i \setminus \bigcup_{j \neq i} S^T_j \), so
 \( (u,v,s) \in S^T_n \) and
\( (v,w,t) \in S^T_m \). Since \( n,m \leq
\rho(n,m)-1 = k \), we have that both \( (u,v, k {}^\smallfrown{}  s') \) and
\( (v,w,k {}^\smallfrown{}  t') \) belong to \( S^T_k \): hence 
\[ 
(u,w,(k+1)
{}^\smallfrown{}  (s' \mathrel{\widetilde{\oplus}} t')) = (u,w, \rho(n , m)
{}^\smallfrown{}  (s' \mathrel{\widetilde{\oplus}} t')) = (u,w,s \oplus t) \in S^T_{\rho(n,m)} \subseteq \bigcup\nolimits_l S^T_l 
\] 
by
definition of the \( S^T_n \)'s and~\eqref{eq:variantoplus}. Since \( (s \oplus t) (0) = \rho(n,m) \neq 0 \), \( (u,w, s \oplus t) \in S_T \), as required.

It remains to prove~\ref{lemmanormalformc1}). We first prove the following claim.

\begin{claim}
\( R = \proj[\bigcup_n S^T_n] \).
\end{claim}

\begin{proof}[Proof of Claim]
One direction is easy as \( R = \proj[\hat{T}] \) (since \( R \) is
reflexive) and
\( \proj[\hat{T}] = \proj[S^T_0] \subseteq \proj[\bigcup_n
S^T_n] \), so assume \( (x,y)
\in \proj[\bigcup_n S^T_n] \) and let \( \xi \in  
\pre{\kappa}{\kappa} \) be
a witness of this fact. Then \( \forall k (x \restriction k , y
\restriction k , \xi
\restriction k) \in S^T_{\xi(0)} \), so that \( (x,y) \in
\proj[S^T_{\xi(0)}] \).
Therefore it is enough to prove by induction on \( n \) that
\( \proj[S^T_n] \subseteq R \)
(in fact, one has \( \proj[S^T_n] = R \) because \( R \subseteq
\proj[S^T_0] \subseteq
\proj[S^T_n] \)).
The case \( n=0 \) is obvious because \( \proj[S^T_0] = \proj[\hat{T}] =
R \), so assume
\( \proj[S^T_n] \subseteq R \), choose an arbitrary \( (x,y) \in
\proj[S^T_{n+1}] \) and
let \( \xi \in \pre{\kappa}{\kappa} \) be such that \( (x,y,(n+1)
{}^\smallfrown{} 
\xi) \in [S^T_{n+1}] \). We distinguish two cases: if for
cofinally many \( \gamma
< \kappa \) we have \( (x \restriction \gamma, y \restriction
\gamma , n
{}^\smallfrown{}  {\xi \restriction \gamma - 1}) \in S^T_n \)
(where we set \( \gamma -1 = \gamma \) if \( \gamma \geq \omega \)) then \( (x,y,n
{}^\smallfrown{}  \xi)
\in [S^T_n] \), so that \( (x,y) \in \proj[S^T_n] \subseteq R \) by inductive hypothesis.
Otherwise, for almost
all \( \gamma < \kappa \) (hence for every \( \gamma < \kappa \),
since \( S^T_n \) is a
DST-tree) there is a \( v_\gamma \in {}^{< \kappa} 2 \) such that \( (x
\restriction
\gamma , v_\gamma, n {}^\smallfrown{}  \xi_0 \restriction
\gamma-1), (v_\gamma, y
\restriction \gamma, n {}^\smallfrown{}  \xi_1 \restriction
\gamma-1) \in S^T_n \),
where \( \xi_0,\xi_1 \in \pre{\kappa}{\kappa} \) are  such that
\( \xi = \xi_0 \mathrel{\widetilde{\oplus}} \xi_1 \) (such  \( \xi_0 \) and \( \xi_1 \) exist and are unique by the fact that \( \mathrel{\widetilde{\oplus}} \) is injective and that clearly \( (s \mathrel{\widetilde{\oplus}} t ) \restriction \alpha = (s \restriction \alpha) \mathrel{\widetilde{\oplus}} (t \restriction \alpha) \) for every \( \alpha \leq \leng{s} = \leng{t} \)). Now notice that the DST-tree \( V = \{ s \in \pre{< \kappa}{2} \mid s \subseteq v_\gamma \text{ for some } \gamma < \kappa \} \)
generated by all these
\( v_\gamma \)'s  is a subtree of \( {}^{< \kappa} 2 \) of
height \( \kappa \) (as
\( \leng{v_\gamma} = \gamma \)). Since \( \kappa \) is inaccessible and has the tree
property,  there is a cofinal branch \( z \in \pre{\kappa}{2} \) through \( V \),
which clearly has the property that \( (x \restriction \gamma, z
\restriction
\gamma, n {}^\smallfrown{}  \xi_0 \restriction \gamma-1), (z
\restriction \gamma,
y \restriction \gamma, n {}^\smallfrown{}  \xi_1 \restriction
\gamma-1) \in S^T_n \)
for every \( \gamma < \kappa \). Therefore \( (x,z),(z,y) \in
\proj[S^T_n] \subseteq R \), hence \( (x,y) \in R \) by the transitivity of \( R \). 
\end{proof}

It remains to show that \( \proj[S_T] = \proj[\bigcup_n S^T_n] \).
One direction is
obvious because \( S_T \subseteq \bigcup_n S^T_n \), so assume
\( (x,y,\xi) \in
[\bigcup_n S^T_n] \). If \( \xi(0) \neq 0 \)
then \( (x,y,\xi) \in [S_T] \) and hence \( (x,y) \in \proj[S_T] \). If instead \( \xi(0) = 0 \), we use the fact
that \( \bigcup_n
S^T_n \) satisfies (the analogous of) condition~\ref{lemmanormalformc3}), the proof being identical
to the one for
\( S_T \). Moreover, it is straightforward to check that \( (x,x,0^{(\kappa)}) \in [\bigcup_n S^T_n] \) (in fact \( \bigcup_n S^T_n \) satisfies the analogous of condition~\ref{lemmanormalformc2})): therefore we get
\( (x,y, \zeta) \in
[\bigcup_n S^T_n] \) for \( \zeta = 0^{(\kappa)} \oplus \xi \), and
since \( \zeta(0) = \rho(0, \xi(0)) = \rho(0,0) \neq
0 \) we obtain \( (x,y,\zeta) \in
[S_T] \), whence \( (x,y) \in \proj[S_T] \) again. Thus \( \proj[\bigcup_n S^T_n ] \subseteq \proj[S_T] \) and we are done.
\end{proof}

\begin{defin}\label{definmax}
Let \( \kappa \) be an infinite cardinal. Given two DST-trees \( \T,\T' \subseteq \pre{< \kappa}{(2
\times \kappa)} \), we let \( \T \leq_{\max} \T' \) if and only if there
is a Lipschitz (i.e.\ a monotone and length-preserving)
\emph{injective} function \( \varphi \colon {}^{< \kappa} \kappa \to {}^{<
\kappa} \kappa \) such that for all \( (u,s) \in \pre{< \kappa}{(2 \times \kappa)} \)
\[ (u,s) \in \T \Rightarrow (u, \varphi(s)) \in \T'. \]
\end{defin}

Notice that every Lipschitz map \( \varphi \colon {}^{< \kappa} \kappa \to {}^{< \kappa} \kappa \) is completely determined by its values on \( \pre{\SUCC(< \kappa)}{\kappa} \).

When \( \kappa \) satisfies~\eqref{eqcardinalcondition}, the quasi-order \( \leq_{\max} \) is easily seen to be analytic once
we naturally identify each DST-tree \( T \subseteq \pre{< \kappa}{(2 \times \kappa)} \)   through its characteristic function (and the bijection between \( \pre{< \kappa}{(2 \times \kappa)} \) and \( \kappa \) provided by the assumption \( \kappa^{< \kappa} = \kappa \)) with an element of \( \pre{\kappa}{2} \) --- under this identification, 
the set of such codes is a (\( \tau_b \)-)closed subset of \( \pre{\kappa}{2} \).

Given a DST-tree \( T \subseteq \pre{< \kappa}{(2 \times 2 \times \kappa)} \) of height \( \leq
\kappa \), let \( S_T \) be the DST-tree defined in~\eqref{eqS_T}. Then define the map \( s_T \) from \( \pre{\kappa}{2} \)
to the space of the DST-subtrees of \( \pre{< \kappa} {(2 \times \kappa)} \) by setting 
\begin{equation} \label{eqs_T}
s_T(x) =
S_T^x = \left\{ (u,s)  \mid  (u,x \restriction \leng{u}, s) \in S_T \right\}.
\end{equation}
Notice that by part~\ref{lemmanormalformc4}) of Lemma~\ref{lemmanormalform} the map \( s_T \) is injective in a strong sense, that is for every \( x,y \in \pre{\kappa}{2} \)
\begin{equation} \label{eqs_Tinjective}
x \neq y \Rightarrow s_T(x) \cap \pre{\SUCC(< \kappa)}{(2 \times \kappa)} \neq s_T(y) \cap \pre{\SUCC(< \kappa)}{(2 \times \kappa)}.
\end{equation}

\begin{lemma}\label{lemmamax}
Let \( \kappa \) be a weakly compact cardinal, and let 
\( R = \proj[T] \) be an analytic quasi-order. Then \( s_T \) reduces \( R \) to
\( \leq_{\max} \).
In particular, \( \leq_{\max} \) is complete for analytic
quasi-orders.
\end{lemma}

\begin{proof}
The proof is identical to the one of~\cite[Theorem 2.5]{louros}, but we give it here in full details for the reader's convenience. Suppose
first that \( \varphi \)
witnesses
\( S_T^x \leq_{\max} S_T^y \). Fix an arbitrary
\( \xi \in \pre{\kappa}{\kappa} \) with \( \xi(0) \in \omega \), and set \( \zeta = \bigcup_{\gamma < \kappa} \varphi(\xi
\restriction \gamma) \). By~\ref{lemmanormalformc2}) of Lemma~\ref{lemmanormalform},
\( (x \restriction \gamma, \xi
\restriction \gamma) \in S_T^x \) for all \( \gamma < \kappa \), hence \( (x \restriction
\gamma, \varphi(\xi \restriction \gamma) \in S_T^y \): but this
means that \( (x,y,\zeta) \in [S_T] \), hence \( (x,y) \in R \) by~\ref{lemmanormalformc1}) of
Lemma~\ref{lemmanormalform}.

Assume now that \( \xi \in \pre{\kappa}{\kappa} \) witnesses \( (x,y)
\in \proj[S_T] = R \) (in particular this means \( \xi(0) \in
\omega \)). For \( s \in \pre{< \kappa}{\kappa} \), define the Lipschitz
map \( \varphi(s) = s \oplus ({\xi \restriction \leng{s}}) \). First notice
that since the function
\( \oplus \) is injective, then the map \( \varphi \) is injective as well.
If now \( u \) and \( s \) are such that \( (u,s) \in S_T^x \) (which, in
particular, means \( s(0) \in \omega \) and \( \leng{u} = \leng{s} \)), then \( (u,
x \restriction \leng{s}, s) \in S_T \) by definition. But on the other
hand \( (x \restriction \leng{s}, y \restriction \leng{s}, \xi
\restriction \leng{s}) \in S_T \), therefore \( (u, y \restriction \leng{s}, s
\oplus ({\xi \restriction \leng{s}})) \in S_T \) by~\ref{lemmanormalformc3}) of Lemma~\ref{lemmanormalform}, which implies \( (u,\varphi(s)) \in S_T^y \). Therefore \( \varphi \) witnesses \( S_T^x \leq_{\max} S_T^y \).
\end{proof}

\begin{remark} \label{remspecialphi}
The proof of Lemma~\ref{lemmamax} actually gives that if \( \kappa, R, T \) are as in the assumptions of the lemma and \( x \mathrel{R} y\), then there is a witness \(\varphi\) of \( s_T(x) \leq_{\max} s_T(y) \) such that for all \( s \in \pre{\SUCC(<\kappa)}{\kappa} \), \( \varphi(s) = s \oplus t \) for some \( t \in \pre{\SUCC(< \kappa)}{\kappa} \).
\end{remark}

\section{Some labels}\label{sectionlabels}

\emph{For the rest of this section, fix an inaccessible cardinal \( \kappa \)}. We will define various (generalized) trees which will be used as labels in the next section. In particular, we will define three kinds of labels (type I, II and III), and then discuss some of their basic properties, e.g.\ that a label of a certain type cannot be embedded into a label of a different type.

Let 
\begin{equation}\label{eqtheta}
 \theta \colon 2 \times \kappa \times \kappa \to \kappa \setminus \{ 0 \} \colon (i,\gamma,\alpha) \mapsto 1+ 2 \cdot \Hes_2(\gamma,\alpha) + i. 
\end{equation}
The map \( \theta \) is a bijection such that for every \( i = 0,1 \) and \( \alpha, \beta, \gamma < \kappa \)
\begin{equation} \label{eqthetamonotone} 
 \alpha \leq \beta \iff \theta(i,\gamma,\alpha) \leq \theta(i,\gamma,\beta) .
\end{equation}

Using \cite[Corollary 5.4.]{bau}, fix a sequence 
\( \langle L_\gamma = ( \kappa, \preceq^{L_\gamma}) \mid \gamma < 
\kappa \rangle \) of linear orders such that 
\( L_\gamma \not\sqsubseteq L_{\gamma'} \) whenever 
\( \gamma \neq \gamma' \). Notice that we can assume that all the 
\( L_\gamma \) have no greatest element (if this is not the case, append a 
copy of the linear order of the
integer numbers \( (\ZZ, \leq ) \) at the end of each of the \( L_\gamma \)'s, and check that the 
resulting linear orders still have the required properties). Then for \( \gamma < \kappa \) let 
\( \sL_\gamma = (D_\gamma, \preceq_\gamma) \) be defined as follows:

\begin{itemize}
\item
\( D_\gamma = \kappa \cup \{ (\alpha, \beta)  \mid \alpha < \kappa \wedge \beta < \theta(0,\gamma,\alpha) \} \);
\item
\( \preceq_\gamma \) is the partial order on \( D_\gamma \) defined by 
\begin{enumerate}
\item
\( \forall \alpha,\alpha' < \kappa  \left({\alpha \preceq_\gamma \alpha'} \iff {\alpha \preceq^{L_\gamma} \alpha'}\right) \)
\item
\( \forall \alpha, \alpha' < \kappa \, \forall \beta < \theta(0,\gamma,\alpha)  \left({\alpha' \preceq_\gamma (\alpha, \beta)} \iff {\alpha' \preceq_\gamma
 \alpha}\right) \)
\item
\( \forall \alpha < \kappa \, \forall \beta , \beta' < \theta(0,\gamma,\alpha)  \left({(\alpha,\beta) \preceq_\gamma (\alpha, \beta')} \iff {\beta \leq \beta'}\right) \).
\item
no other \( \preceq_\gamma \)-relation holds.
\end{enumerate}
\end{itemize}

Trees of the form \( \sL_\gamma \) are called \emph{labels of type I}. We also say that a tree is a \emph{code for \( \gamma \)} if it is isomorphic to \( \sL_\gamma \). Notice that labels of type I have always cardinality \( \kappa \).

\medskip

Let \( \gamma < \kappa \). Given \( s \in \pre{\gamma+1}{\kappa} \), let \( \sL_s = (D_s, \preceq_s) \) be the tree defined as follows:
\begin{itemize}
\item
\( D_s \) is the disjoint union of \( \theta(1, \gamma, \# s) \), \( \omega^* = \{ n^* \mid n \in \omega \} \), and \( A_s = \{ a,a^+,a^-,b,b^+,b^-,c,c^+,c^- \} \), where \( \# \) is as in Proposition~\ref{propoplus}(\ref{propoplusc2});
\item
\( \preceq_s \) is the partial order on \( D_s \) defined by
\begin{enumerate}
\item
\( \forall \alpha,\beta < \theta(1, \gamma, \# s)  \left({\alpha \preceq_s \beta} \iff {\alpha \leq \beta}\right) \)
\item
\( \forall n,m \in \omega  \left({n^* \preceq_s m^*} \iff {n \geq m}\right) \)
\item
\( x \preceq_s x^+, x^- \) for \( x \in \{ a,b,c \} \)
\item
\( \forall \alpha < \theta(1, \gamma, \# s) \, \forall n \in \omega \, \forall x \in A_s  \left({\alpha \preceq_s n^*} \wedge {n^* \preceq_s x}\right) \)
\item
no other \( \preceq_s \)-relation holds.
\end{enumerate}
\end{itemize}

Trees of the form \( \sL_s \) are called \emph{labels of type II}, and a tree isomorphic to \( \sL_s \) is called a \emph{code for \( s \)}. The restriction of \( \sL_s \) to \( \theta(1, \gamma, \#s) \) is called the \emph{initial part} of \( \sL_s \). In particular, the initial part of \( \sL_s \) is isomorphic to the ordinal \( \theta(1, \gamma, \# s) \) (hence it is well-founded). Notice that labels of type II have always cardinality \( < \kappa \).

\medskip 

Let \( \mu \colon \kappa \to \kappa \colon \gamma \mapsto (\max \{ \omega_1, |\gamma + 1| , \sup \{ \mu(\alpha) \mid \alpha < \gamma \} \})^+ \) (\(\mu\) is well-defined because \( \kappa \) is inaccessible), so that \(\mu\) is an injective map, \( \mu(\gamma) \) is always an uncountable regular cardinal, and \( 2^{\mu(\gamma)} \geq | \pre{\gamma + 1}{2}| \). Fix \( \gamma < \kappa \). Using \cite[Corollary 5.4.]{bau} again, fix a sequence \( \langle L^*_u  \mid u \in \pre{\gamma+1}{2} \rangle \) of linear orders of size \( \mu(\gamma) \) such that \( L^*_u \not\sqsubseteq L^*_v \) for distinct \( u,v \in \pre{\gamma+1}{2} \). Notice that we can assume that none of these \( L^*_u \) is a well-order. Then for every \( u \in \pre{\gamma+1}{2} \) define the tree \( \sL^*_u = (D^*_u, \preceq^*_u ) \) as follows:
\begin{itemize}
\item
\( D^*_u \) is the disjoint union of \( L^*_u \), \( \omega^* = \{ n^* \mid n \in \omega \} \), and \( A^*_u = \{ a,a^+,a^-,b,b^+,b^- \} \);
\item
\( \preceq^*_u \) is the partial order on \( D^*_u \) defined as follows
\begin{enumerate}
\item
\( \forall x,y \in L^*_u  \left({\alpha \preceq^*_u \beta} \iff {\alpha \preceq^{L^*_u} \beta}\right) \)
\item
\( \forall n,m \in \omega  \left({n^* \preceq^*_u m^*} \iff {n \geq m}\right) \)
\item
\( x \preceq^*_u x^+, x^- \) for \( x \in \{ a,b \} \)
\item
\( \forall x \in L^*_u \, \forall n \in \omega \, \forall y \in A^*_u  \left({x \preceq^*_u n^*} \wedge {n^* \preceq^*_u y}\right) \)
\item
no other \( \preceq^*_u \)-relation holds.
\end{enumerate}
\end{itemize}

Trees of the form \( \sL^*_u \) are called \emph{labels of type III}. Similarly to the previous cases, a tree isomorphic to \( \sL^*_u \) is called a \emph{code for \( u \)}. The restriction of \( \sL^*_u \) to \( L^*_u \) is called the \emph{initial part} of \( \sL^*_u \). In particular, the initial part of \( \sL^*_u \) is an 
ill-founded linear order. Notice that also the labels of type III have always cardinality \( < \kappa \): in fact, \( | \sL^*_u| = \mu(\leng{u}-1) \) for every \( u \in \pre{\SUCC(< \kappa)}{2} \).

\begin{lemma}\label{lemmatypeIIandIII}
Assume \( \sL, \sL' \) are labels of type II or III. If \( i \) is an embedding of \( \sL \) into \( \sL' \), then \( i \) maps the initial part of \( \sL \) into the initial part of \( \sL' \).
\end{lemma}

\begin{proof}
Notice that \( x \) is in the initial part of a label of type II or III if and only if \( \cone(x) \) is infinite. Therefore, if \( x \) is in the initial part of \( \sL \), then \( i(x) \) must be in the initial part of \( \sL' \) because \( i`` \, \cone(x) \subseteq \cone(i(x)) \) implies that  \( \cone(i(x)) \) is infinite.
\end{proof}

\begin{lemma} \label{lemmalabelsdifferenttypes}
Let \( \sL, \sL' \) be two labels of different type. Then \( \sL \not\sqsubseteq \sL' \).
\end{lemma}

\begin{proof}
First assume that \( \sL \) is of type I and \( \sL' \) is either of type II or of type III. Then \( \sL \not\sqsubseteq \sL' \) because  \( |\sL| = \kappa > | \sL' | \).

Now assume that \( \sL \) is either of type II or of type III, and \( \sL' \) is of type I. Then \( \sL \not\sqsubseteq \sL' \) because in \( \sL \) the points \(a,b\) are incomparable and both \( \cone(a) \) and \( \cone(b) \) are \emph{not} linear orders, while if \( x,y \) are incomparable points in \( \sL' \), then at least one of \( \cone(x) , \cone(y) \) is a linear order.

Next assume \( \sL \) is of type II and \( \sL' \) is of type III. Notice that \( \sL \) contains six pairwise incomparable elements (the points \( a^+,a^-,b^+,b^-,c^+,c^- \)), while in \( \sL' \) there are at most four pairwise incomparable elements (namely, the points \( a^+, a^-, b^+, b^- \)). Since embeddings preserves (in)comparability, \( \sL \not\sqsubseteq \sL' \).

Finally, we assume that \( \sL \) is of type III while \( \sL' \) is of type II.  Assume towards a contradiction that \( i \) is an embedding of \( \sL \) into \( \sL' \). By Lemma~\ref{lemmatypeIIandIII}, \( i \) maps the initial part of \( \sL \) into the initial part of \( \sL' \), which is a contradiction because the initial part of \( \sL \) is ill-founded while the initial part of \( \sL' \) is a well-order. 
\end{proof}

Notice that Lemma~\ref{lemmalabelsdifferenttypes} implies that if \( \sL \) is a label (of any type) and \( \sL',\sL'' \) are two labels of different type one from the other, then \( \sL',\sL'' \) cannot be simultaneously embedded into \( \sL \). This simple fact will be used in Claim~\ref{claimstem}.

\begin{lemma} \label{lemmaL_gamma}
Let \( \gamma , \gamma' < \kappa \). If \( \gamma \neq \gamma'\) then \( \sL_\gamma \not\sqsubseteq \sL_{\gamma'} \).
\end{lemma}

\begin{proof}
Assume towards a contradiction that \( i \colon D_\gamma \to D_{\gamma'} \) is an embedding of \( \sL_\gamma \) into 
\( \sL_{\gamma'} \). Let \( \alpha \in \kappa \, \left(\subseteq D_\gamma\right) \), and let \( \alpha' \in \kappa \setminus \{ \alpha \} \) 
be such that \( \alpha \preceq_\gamma \alpha' \) (such an \( \alpha' \) exists because we assumed that \( L_\gamma \) has 
no maximal element). Then \( \alpha', (\alpha,0) \in \cone(\alpha) \) are incomparable (\( (\alpha, 0 ) \in D_\gamma \) 
because \( \theta(0,\gamma,\alpha) > 0 \) for every \( \gamma, \alpha < \kappa \)). Since  
\( i``\, \cone(\alpha) \subseteq \cone(i(\alpha)) \), and \( \cone(x) \) is a linear order for every 
\( x \in D_{\gamma'} \setminus \kappa \), we must conclude that \( i(\alpha) \in \kappa\, \left(\subseteq D_{\gamma'}\right) \). 
Therefore \( i \restriction \kappa \) is an embedding of \( L_\gamma \) into \( L_{\gamma'} \), which contradicts the 
choice of the \( L_\gamma \)'s.
\end{proof}

\begin{lemma}\label{lemmaL_s}
Let \( \gamma < \kappa \) and \(s,t \in \pre{\gamma+1}{\kappa} \). Then 
\( \sL_s \sqsubseteq \sL_t \iff \# s \leq \# t \). Moreover, \( \sL_s \cong \sL_t \iff s = t \) for every \( s,t \in \pre{\SUCC(< \kappa)}{\kappa} \).
\end{lemma}

\begin{proof}
By Lemma~\ref{lemmatypeIIandIII}, any embedding of \( \sL_s \) into \( \sL_t \) would map the initial part of \( \sL_s \) into the initial part of \( \sL_t \). This implies \( \theta(1,\gamma,\#s) \leq \theta(1,\gamma,\# t) \), which in turn implies \( \#s \leq \# t \) by~\eqref{eqthetamonotone}.
Conversely, if \( \# s \leq \# t \) then \( D_s \subseteq D_t \) (by~\eqref{eqthetamonotone} again), and the identity map on \( D_s \) is an embedding of \( \sL_s \) into \( \sL_t \).

The nontrivial direction of the second part follows from the fact that if \( \sL_s \cong \sL_t \) for some \( s,t \in \pre{\SUCC(< \kappa)}{\kappa} \), then \( \theta(1, \leng{s}-1, \# s) = \theta(1, \leng{t}-1,\# t) \) by Lemma~\ref{lemmatypeIIandIII}. Since \(\theta\) is a bijection, \( \leng{s} = \leng{t} \) and \( \# s = \# t \), which implies \( s = t \) by Proposition~\ref{propoplus}(\ref{propoplusc2b})).
\end{proof}

\begin{lemma} \label{lemmaL_u}
Let \( \gamma < \kappa \) and \( u,v \in \pre{\gamma+1}{\kappa} \). If \( u \neq v \) then \( \sL^*_u \not\sqsubseteq \sL^*_v \). Moreover,  \( {\sL^*_u \cong \sL^*_v} \iff {u=v} \) for every \( u,v \in \pre{\SUCC(< \kappa)}{2} \).
\end{lemma}

\begin{proof}
Assume \( u \neq v \).
By Lemma~\ref{lemmatypeIIandIII}, any embedding of \( \sL^*_u \) into \( \sL^*_v \) would map the initial part of \( \sL^*_u \) into the initial part of \( \sL^*_v \). Therefore, the restriction of such an embedding to the initial part of \( \sL^*_u \) would be an embedding of \( L^*_u \) into \( L^*_v \), contradicting the choice of the \( L^*_u \)'s.

The nontrivial direction of the second part, follows from the first part and the fact that if \( u,v \in \pre{\SUCC(< \kappa)}{2} \) have different lengths then \( | \sL^*_u | = \mu(\leng{u}-1) \neq \mu(\leng{v}-1) =  | \sL^*_v| \) by the injectivity of the function \(\mu\).
\end{proof}

\section{Completeness of the embeddability relation}\label{sectioncomplete}

Given an \emph{inaccessible} cardinal \( \kappa \) and a DST-tree \( \mathcal{T} \) on 
\( 2 \times \kappa \) of height \( \kappa \), we will now define a (generalized) 
tree \( G_{\mathcal{T}} \) of size \( \kappa \).  To define 
\( G_{\mathcal{T}} \) we will in particular use the labels defined in Section~\ref{sectionlabels}. Formally, two distinct labels (even of different 
types) may have nondistinct domains. However, considering suitable isomorphic copies, we can assume without loss of generality that for every 
\( \gamma < \kappa \), \( s,t \in \pre{\SUCC(< \kappa)}{\kappa} \), and 
\( u,v \in \pre{\SUCC(< \kappa)}{2} \) the following conditions hold:
\begin{enumerate}[i)]
\item
\( \sL_\gamma\), \( \sL_s \) and \( \sL^*_u \) have pairwise disjoint 
domains;
\item
\( \sL^*_u \) and \( \sL^*_v \) have disjoint domains if and only if 
\( u \neq v \);
\item
 if \( \leng{s} = \leng{t} \), then the domain of \( \sL_s \) is contained in the 
domain of \(\sL_t \) if and only if \( \# s \leq \# t \) (this requirement can be satisfied by~\eqref{eqthetamonotone}).
\end{enumerate}

\noindent
These technical assumptions will ensure that the trees \( G_{\mathcal{T}} \) are well-defined avoiding unnecessary complications in the notation.

Let us now first define the tree \( G_0 \) (which is independent of the choice of 
\( \mathcal{T} \)). Roughly speaking, \( G_0 \) will be constructed by appending 
to the nodes of the tree \( \left(\pre{\SUCC(< \kappa)}{\kappa}, \subseteq \right) \)
 some labels as follows. Let \( \bar{\gamma} \colon \pre{\SUCC(<\kappa)}{\kappa} \to \kappa \colon s \mapsto \leng{s}-1 \). For every \( s \in 
\pre{\SUCC(< \kappa)}{\kappa} \) we fix a distinct copy of \( (\ZZ , \leq ) \)  and append it to \( s \): each of these copies of \( \ZZ \) will be called a
\emph{stem}, and if such a copy is appended to \( s \) it will be called the 
\emph{stem of \( s \)}. Then for every such \( s \) we fix also distinct 
copies \( \sL_{\bar{\gamma}(s),s} \) and \( \sL_{s,s} \) of, respectively, \( \sL_{\bar{\gamma}(s)} \) and \( \sL_s \), and then append both of them to the stem of \( s \).   
More formally, we have the following definition.

\begin{defin}
The tree \( G_0  \) is defined by the following conditions:
\begin{itemize}
\item
\( G_0 = \pre{\SUCC(< \kappa)}{\kappa}\cup \bigcup_{ s \in \pre{\SUCC(< \kappa)}{\kappa}} ( \{ (s,x) \mid x \in \ZZ \cup D_{\bar{\gamma}(s)}  \cup  D_s \}   ) \), where \( D_{\bar{\gamma}(s)} \)'s and \( D_s \) are the domains of, respectively, \( \sL_{\bar{\gamma}(s)} \) and \( \sL_s \);
\item
the partial order \( \preceq^{G_0} \) on \( G_0 \) is defined as follows:
\begin{enumerate}
\item
\( \forall s,t \in \pre{\SUCC(<\kappa)}{\kappa}  \left({s \preceq^{G_0} t} \iff {s \subseteq t}\right) \)
\item
\( \forall s \in \pre{\SUCC(<\kappa)}{\kappa} \, \forall z,z' \in \ZZ  \left({(s,z) \preceq^{G_0} (s,z')} \iff {z \leq z'}\right) \)
\item
\( \forall s \in \pre{\SUCC(< \kappa)}{\kappa}\, \forall x,x' \in D_{\bar{\gamma}(s)}  \left({(s,x) \preceq^{G_0} (s,x')} \iff {x \preceq_{\bar{\gamma}(s)} x'}\right) \)
\item
\( \forall s \in \pre{\SUCC(<\kappa)}{\kappa}\, \forall x,x' \in D_s  \left({(s,x) \preceq^{G_0} (s,x')} \iff {x \preceq_s x'}\right) \)
\item
\( \forall s,t\in \pre{\SUCC(<\kappa)}{\kappa}\, \forall x \in \ZZ \cup D_{\bar{\gamma}(t)} \cup D_t  \left({s \preceq^{G_0} (t,x)} \iff {s \subseteq t}\right) \)
\item
\( \forall s \in \pre{\SUCC(< \kappa)}{\kappa}\,  \forall z \in \ZZ \, \forall x \in D_{\bar{\gamma}(s)} \cup D_s  \left({(s,z) \preceq^{G_0} (s,x)} \right) \)
\item
no other \( \preceq^{G_0} \)-relation holds.
\end{enumerate}
\end{itemize}
\end{defin}

\noindent
So the \emph{stem of \( s \)} is \( G_0 \restriction \{ s \} \times \ZZ \). Substructures of the form \( G_0 \restriction \{ s \} \times D_{\bar{\gamma}(s)} \) and \( 
G_0 \restriction \{ s \} \times D_s \) will be called \emph{labels} (of type I 
and II, respectively). 

Let now \( \mathcal{T} \) be a DST-tree on \( 2 \times \kappa \) of height \( \kappa \). The tree \( G_{\mathcal{T}} \) will be constructed by appending a distinct copy of the label \( \sL^*_u \) to the stem of \( s \) for every \( (u,s) \in \mathcal{T} \).

\begin{defin}
The tree \( G_{\mathcal{T}} = (D_{\mathcal{T}}, \preceq_{\mathcal{T}}) \) is defined as follows:
\begin{itemize}
\item
\( D_{\mathcal{T}} = G_0 \cup \bigcup_{\substack{(u,s) \in \mathcal{T} \\ s \in \pre{\SUCC(< \kappa)}{\kappa}}} \{ (s,x) \mid x \in D^*_u \} \), where \( D^*_u \) is the domain on \( \sL^*_u \);
\item
\( \preceq_{\mathcal{T}} \) is the partial order on \( D_{\mathcal{T}} \) defined by:
\begin{enumerate}
\item
\( \forall x,y \in G_0  \left({x \preceq_{\mathcal{T}} y } \iff {x \preceq^{G_0} y}\right) \)
\item
\(  \forall (u,s) \in \mathcal{T}  \left[s \in \pre{\SUCC(<\kappa)}{\kappa} \Rightarrow \forall x,y \in D^*_u \left ({(s,x) \preceq_{\mathcal{T}} (s,y)} \iff {x \preceq^*_u y}\right) \right]\)
\item
\( \forall t \in \pre{\SUCC(<\kappa)}{\kappa} \, \forall (u,s) \in \mathcal{T}  \left [s \in \pre{\SUCC(<\kappa)}{\kappa} \Rightarrow \forall x \in D^*_u  \left ({t \preceq_{\mathcal{T}} (s,x) } \iff {t \subseteq s} \right)\right]  \)
\item
\( \forall (u,s) \in \mathcal{T}  \left[s \in \pre{\SUCC(<\kappa)}{\kappa} \Rightarrow \forall x \in D^*_u \, \forall z \in \ZZ \left ((s,z) \preceq_{\mathcal{T}} (s,x)\right)\right] \)
\item
no other \( \preceq_{\mathcal{T}} \)-relation holds.
\end{enumerate}
\end{itemize}
\end{defin}

Substructures of the form \( G_{\mathcal{T}} \restriction \{ s \} \times \ZZ \),  \( G_{\mathcal{T}} \restriction \{ s \} \times D_{\bar{\gamma}(s)} \), and \( G_{\mathcal{T}} \restriction \{ s \} \times D_s \) will again be called, respectively, \emph{stem of \( s \)},  \emph{labels of type I} and \emph{labels of type II}, and be denoted by, respectively, \( \sS^{\mathcal{T}}_{s} \), \( \sL^{\mathcal{T}}_{\bar{\gamma}(s),s} \), and \( \sL^{\mathcal{T}}_{s,s} \). Similarly, substructures of the form \( G_{\mathcal{T}} \restriction \{ s \} \times D^*_u \) (for \( (u,s) \in \mathcal{T} \)) will be called \emph{labels of type III}, and be denoted by \( \sL^{\mathcal{T}}_{u,s} \).
Notice that if \( \sL \) is a label of \( G_{\mathcal{T}} \) with domain \( D_{\sL} \) and \( x  \in D_{\sL} \), then \( \cone(x) \subseteq D_{\sL} \).

For \( s \in \pre{\SUCC(< \kappa)}{\kappa} \), we let
\[ 
\cone(\sS^{\mathcal{T}}_s) = \bigcap_{z \in \ZZ} \cone((s,z)).
 \] 
Therefore,  \( \cone(\sS^{\mathcal{T}}_s) \) consists of a disjoint union of 
labels of various type. In particular, it contains exactly one label of type I 
(namely, \( \sL^{\mathcal{T}}_{\bar{\gamma}(s),s} \)), one label of type II (that is, 
\( \sL^{\mathcal{T}}_{s,s} \)) and, depending on \( \mathcal{T} \), a variable 
number of labels of type III (namely, a label of the form 
\(\sL^{\mathcal{T}}_{u,s} \) for every \( (u,s) \in \mathcal{T} \)). Notice 
also that every label \( \sL \subseteq \cone(\sS^\mathcal{T}_s) \) is a 
maximal connected component of \( G_\mathcal{T} \restriction 
\cone(\sS^\mathcal{T}_s) \).

\begin{theorem} \label{theorcomplete}
Let \( \kappa \) be an inaccessible cardinal, let \( \mathcal{T}, \mathcal{T}' \) be two DST-trees on \( 2 \times \kappa \) of height \( \kappa \), and let \( \# \) be as in 
Proposition~\ref{propoplus}(\ref{propoplusc2}).
\begin{enumerate}[(1)]
\item\label{theorcomplete1}
\( {G_{\mathcal{T}} \sqsubseteq G_{\mathcal{T}'}} \iff \) there is a witness \( \varphi \colon \pre{< \kappa}{\kappa} \to \pre{< \kappa}{\kappa} \) of \( {\mathcal{T} \leq_{\max} \mathcal{T}'} \) such that \( \forall  s\in \pre{\SUCC(< \kappa)}{\kappa} \left (\#s \leq \# \varphi(s) \right)  \);
\item \label{theorcomplete2}	
\( {G_{\mathcal{T}} \cong G_{\mathcal{T}'}} \iff {\mathcal{T} \cap \pre{\SUCC(<\kappa)}{(2 \times \kappa)} = \mathcal{T}' \cap \pre{\SUCC(< \kappa)}{(2 \times \kappa)}}  \).
\end{enumerate}
\end{theorem}

\begin{proof}
We first prove part~(\ref{theorcomplete1}). The \( \Leftarrow \) direction is easy: if \(\varphi\) 
is as in~(\ref{theorcomplete1}) then the map \( i \colon D_{\mathcal{T}} \to D_{\mathcal{T}'} \) 
defined by \( i(s) = \varphi(s)\) and \( i((s,x)) = (\varphi(s),x) \) (for every \( s 
\in \pre{\SUCC(<\kappa)}{\kappa} \) and \( (s,x) \in D_{\mathcal{T}} \)) is 
an embedding. In fact, the unique nontrivial thing that must be checked is that 
\( i \) is well-defined, i.e.\ that \( (\varphi(s),x) \in D_{\mathcal{T}'} \) for 
every \( (s,x) \in D_{\mathcal{T}} \). If \( (s,x) \) belongs to \( \sS^\mathcal{T}_s \), then the above claim is obvious. If \( (s,x) \) belongs to a label of type I, 
the claim follows from the fact that \(\varphi\) is a Lipschitz map and hence \( 
\leng{s} = \leng{\varphi(s)} \). If \( (s,x) \) belongs to a label of type II, the 
claim follows from the assumption that \( \# s \leq \# \varphi(s) \) and~\eqref{eqthetamonotone}. Finally, if \( (s,x) \) belongs to a label of type III, the claim 
follows from the fact that \(\varphi\) witnesses \( \mathcal{T} \leq_{\max} 
\mathcal{T}' \).

To prove the \( \Rightarrow \) direction of part~(\ref{theorcomplete1}), let 
\( i \colon D_{\mathcal{T}} \to D_{\mathcal{T}'} \) be an embedding of 
\( G_{\mathcal{T}} \) into \( G_{\mathcal{T}'} \). First observe that 
 \( x \in \pre{\SUCC(< \kappa)}{\kappa} \) if and only if \( \cone(x) \) 
contains a copy of the tree 
\( \left(\pre{\SUCC(< \kappa)}{\kappa}, \subseteq \right) \) if and only if, in particular, there are 
\( \kappa \)-many \( y \in \cone(x)  \) with \( \cone(y) \) not a linear order. 
Conversely, \( x \notin \pre{\SUCC(< \kappa)}{\kappa} \) if 
and only if \( \pred(x) \) is ill-founded. Therefore, since \( i \)  must preserve 
(in)comparability and descending chains,  
\[ 
x \in \pre{\SUCC(< \kappa)}{\kappa}\,  \left(\subseteq D_\mathcal{T}\right) \iff i(x) \in \pre{\SUCC(< \kappa)}{\kappa} \, \left(\subseteq D_{\mathcal{T}'}\right) .
\]

\begin{claim}\label{claimstem}
For every \( s \in \pre{\SUCC(<\kappa)}{\kappa} \) there is 
\( t = \varphi(s) \in \pre{\SUCC(< \kappa)}{\kappa} \) such that 
\( i``\,\sS^{\mathcal{T}}_s \subseteq \sS^{\mathcal{T}'}_t \). Moreover, 
\( i``\,\sS^{\mathcal{T}}_s \) is cofinal and coinitial in 
\(  \sS^{\mathcal{T}'}_t \), i.e.\ for every \( z \in \ZZ \) there are 
\( z_0, z_1 \in \ZZ \) such that \( i((s,z_0)), i((s,z_1)) \neq (t,z) \) and 
\( i((s,z_0)) \preceq_{\mathcal{T}'} (t,z) \preceq_{\mathcal{T}'} i((s,z_1)) \). 
\end{claim}

\begin{proof}[Proof of the Claim]
Let \( s \in \pre{\SUCC(< \kappa)}{\kappa} \).
By the previous observation, for all \( z \in \ZZ \) we have \( i((s,z)) \notin \pre{\SUCC(< \kappa)}{\kappa} \), therefore there is \( t_z \in \pre{\SUCC(< \kappa)}{\kappa} \) such that \( i((s,z)) \in \sS^{\mathcal{T}'}_{t_z} \cup \cone(\sS^{\mathcal{T}'}_{t_z}) \). But since for \( z,z' \in \ZZ \) the points \( (s,z) \) and \( (s,z') \) are \( \preceq_{\mathcal{T}} \)-comparable, \( t_z = t_{z'} \): let \( t \) denote this unique sequence. We claim that \( i``\,\sS^{\mathcal{T}}_s \subseteq \sS^{\mathcal{T}'}_t \). 

Suppose that \( z \in \ZZ \) is such that \( i((s,z)) \notin \sS^{\mathcal{T}'}_t\), so that \( i((s,z)) \in \cone(\sS^{\mathcal{T}'}_t) \). Then there would be a label \( \sL \subseteq \cone(\sS^{\mathcal{T}'}_t) \) containing \( i((s,z)) \), and in particular we would have \( i``\,\cone((s,z)) \subseteq \cone(i((s,z))) \subseteq \L \). But this contradicts (the comment following) Lemma~\ref{lemmalabelsdifferenttypes}, as \( \cone((s,z)) \) contains both a label of type I and a label of type II, and these labels cannot be simultaneously embedded in the label \( \sL \).

The last part of the claim follows from the fact that \( \sS^{\mathcal{T}}_s \) has order type \( \ZZ \) and therefore cannot be embedded in linear orders of type \( (\omega, \leq) \) or \( (\omega, \geq) \).
\end{proof}

Claim~\ref{claimstem} yields in particular a map \( \varphi \colon \pre{\SUCC(< \kappa)}{\kappa} \to \pre{\SUCC(< \kappa)}{\kappa} \), namely the map where \( \varphi(s) = t \iff i``\, \sS^\mathcal{T}_s \subseteq \sS^{\mathcal{T}'}_t \). Notice also that, as already observed, two points in \( \cone(\sS^{\mathcal{T}}_s) \) (or, respectively, in \( \cone(\sS^{\mathcal{T}'}_{\varphi(s)}) \)) are compatible in such substructure if and only if they belong to the same label. Since \( i``\, \sS^{\mathcal{T}}_s \) is cofinal in 
\(  \sS^{\mathcal{T}'}_t \), this implies that every label contained in \( \cone(\sS^{\mathcal{T}}_s) \) is mapped by \( i \) into a single label contained in \(  \sS^{\mathcal{T}'}_t \).

\begin{claim}\label{claimlength}
For every \( s \in \pre{\SUCC(< \kappa)}{\kappa} \), \( \leng{\varphi(s)} = \leng{s} \).
\end{claim}

\begin{proof}[Proof of the Claim]
The structure \( \cone(\sS^\mathcal{T}_s) \) contains the label \( \sL^\mathcal{T}_{\bar{\gamma}(s),s} \). By the observation following Claim~\ref{claimstem}, \( i``\, \sL^\mathcal{T}_{\bar{\gamma}(s),s} \subseteq \sL \) for some label \( \sL \) contained in \( \cone(\sS^{\mathcal{T}'}_{\varphi(s)}) \). By Lemma~\ref{lemmalabelsdifferenttypes}, \( \sL \) must be of type I: hence \( \sL =  \sL^{\mathcal{T}'}_{\bar{\gamma}(\varphi(s)),\varphi(s)} \), the unique label of type I contained in \( \cone(\sS^{\mathcal{T}'}_{\varphi(s)}) \). Therefore \(  \leng{\varphi(s)} = \leng{s} \) by Lemma~\ref{lemmaL_gamma}.
\end{proof}

\begin{claim}\label{claimLipschitz}
For every \( s \in \pre{\SUCC(< \kappa)}{\kappa} \), \( \varphi(s) = i(s) \). In 
particular, \( \varphi \) is monotone.
\end{claim}

\begin{proof}
First notice that  every point in \( \sS^{\mathcal{T}}_s \) is \( 
\preceq_{\mathcal{T}} \)-above \( s \), and therefore by Claim~\ref{claimstem} every point of \( \sS^{\mathcal{T}'}_{\varphi(s)} \) must be 
above \( i(s) \). This implies \( i(s) \subseteq \varphi(s) \) by definition of \( 
G_{\mathcal{T}'} \). Since an easy induction on \( \leng{s} \) shows that \( 
\leng{s} \leq \leng{i(s)} \), we have \( \leng{\varphi(s)} \leq \leng{i(s)} \) by 
Claim~\ref{claimlength}, which in turn implies \( i(s) = \varphi(s) \).

For the second part, it is enough to further observe that \( i \restriction \pre{\SUCC(< \kappa)}{\kappa} \colon \pre{\SUCC(< \kappa)}{\kappa} \to \pre{\SUCC(< \kappa)}{\kappa} \) is monotone (with respect to the subsequence relation). In fact, for every \( s,s' \in \pre{\SUCC(< \kappa)}{\kappa} \)
\[ 
{s \subseteq s'} \iff {s \preceq_{\mathcal{T}} s'} \iff {i(s) \preceq_{\mathcal{T}'} i(s')} \iff {i(s) \subseteq i(s')}. \qedhere
 \] 
\end{proof}

Now extend \( \varphi \) to \( \pre{< \kappa}{\kappa } \) by setting \( \varphi(s) = 
\cup_{\gamma +1< \leng{s}} \varphi(s \restriction (\gamma+1)) \) for every 
\(s\) of limit length. The definition is well-posed by Claim~\ref{claimLipschitz}. 
Moreover, the map \( \varphi \colon \pre{<\kappa}{\kappa} \to 
\pre{<\kappa}{\kappa} \) is Lipschitz by Claim~\ref{claimlength} and Claim~\ref{claimLipschitz}. 

\begin{claim}\label{claimmax}
\( \varphi \) witnesses \( \mathcal{T} \leq_{\max} \mathcal{T}' \).
\end{claim}

\begin{proof}
We use an argument similar to that of Claim~\ref{claimlength}.
Let \( (u,s) \in \mathcal{T} \). Then \( \cone(\sS^{\mathcal{T}}_s) \) contains the label \( \sL^\mathcal{T}_{u,s} \). By the observation following Claim~\ref{claimstem}, \( i``\,\sL^\mathcal{T}_{u,s} \subseteq \sL \) for some label \( \sL \) contained in \( \cone(\sS^{\mathcal{T}'}_{\varphi(s)}) \). By Lemma~\ref{lemmalabelsdifferenttypes}, \( \sL \) must be of type III, hence by Lemma~\ref{lemmaL_u} we have \( \sL = \sL^{\mathcal{T}'}_{u,\varphi(s)} \). This implies \( (u,\varphi(s)) \in \mathcal{T}' \), as required. 
\end{proof}

The next claim concludes the proof of~(\ref{theorcomplete1}).

\begin{claim}\label{claimnumber}
For every \(s \in \pre{\SUCC(< \kappa)}{\kappa} \), \( \# s \leq \# \varphi(s) \).
\end{claim}

\begin{proof}
We use again an argument similar to that of Claim~\ref{claimlength}. Observe that the label \( \sL^{\mathcal{T}}_{s,s} \) is contained in \( \cone(\sS^{\mathcal{T}}_s) \), therefore by the observation following Claim~\ref{claimstem} we have \( i``\,\sL^{\mathcal{T}}_{s,s} \subseteq \sL \) for some label  \( \sL \) contained in \( \cone(\sS^{\mathcal{T}'}_{\varphi(s)}) \). By Lemma~\ref{lemmalabelsdifferenttypes}, \( \sL \) must be of type II, i.e.\ \( \sL = \sL^{\mathcal{T}'}_{\varphi(s), \varphi(s)} \). Hence by Lemma~\ref{lemmaL_s} we have \( \#s \leq \# \varphi(s) \), as required. 
\end{proof}

\medskip

We now prove  part~(\ref{theorcomplete2}).
The \( \Leftarrow \) direction is trivial, so
we will prove just the \( \Rightarrow \) direction. Suppose \( G_{\mathcal{T}} \cong G_{\mathcal{T}'} \) via some isomorphism \( i \colon D_\mathcal{T} \to D_{\mathcal{T}'} \). By Claim~\ref{claimlength}, Claim~\ref{claimLipschitz}, and Claim~\ref{claimnumber} applied to both \( i \) and \( i^{-1} \), we have that \( \# s = \# i(s) \) for every \( s \in \pre{\SUCC(< \kappa)}{\kappa} \): hence \( \varphi = i \restriction \pre{\SUCC(< \kappa)}{\kappa} \) is the identity function. The argument contained in the proof of Claim~\ref{claimmax} shows that
\begin{align*}
(u,s) \in \mathcal{T} \cap \pre{\SUCC(<\kappa)}{( 2 \times \kappa )} & \iff \cone(\sS^{\mathcal{T}}_s) \text{ contains the label } \sL^{\mathcal{T}}_{u,s} \\
& \iff \cone(\sS^{\mathcal{T}'}_{\varphi(s)}) = \cone(\sS^{\mathcal{T}'}_s) \text{ contains the label } \sL^{\mathcal{T}}_{u,s} \\
& \iff (u,s) \in \mathcal{T}' \cap \pre{\SUCC(< \kappa)}{( 2 \times \kappa )}.
\end{align*}
Thus \( \mathcal{T} \cap \pre{\SUCC(< \kappa)}{( 2 \times \kappa )} = \mathcal{T'} \cap \pre{\SUCC(< \kappa)}{( 2 \times \kappa )} \).
\end{proof}

Now let \( \kappa \) be a weakly compact cardinal and \( R  = \proj[T] \) be an analytic quasi-order on \( \pre{\kappa}{2} \). Recall that in~\eqref{eqs_T} we defined a map \( s_T \) sending \( x \in \pre{\kappa}{2} \) into a DST-tree on \( 2 \times \kappa \) of height \( \kappa \) denoted by \( s_T(x) \). 
Since each tree \( G_{s_T(x)} \) can be easily Borel-in-\( T \) coded into a tree with domain \( \kappa \), henceforth \( G_{s_T(x)} \) will be tacitly identified with such a copy. With this notational convention, the composition of \( s_T \) with the map sending \( \mathcal{T} \) into \( G_{\mathcal{T}} \) gives the function
\begin{equation}\label{eqf} 
f \colon \pre{\kappa}{2} \to \Mod^\kappa_\L \colon x \mapsto G_{s_T(x)}, 
\end{equation}
which will be our reduction of \( R \) to the embeddability relation \( \sqsubseteq \restriction \Mod^\kappa_\L \), for \(\L\) the language of trees. 

\begin{defin}
Given a cardinal \( \kappa \), let \( \sqsubseteq^\kappa_{\mathsf{TREE}} \)  (\( \sqsubseteq^\kappa_{\mathsf{GRAPH}} , \sqsubseteq^\kappa_{\mathsf{LO}} \))  denote  the relation of embeddability between trees (respectively, graphs, linear orders) of size \( \kappa \), 
\end{defin}

\begin{corollary} \label{corcompletetree}
Let \( \kappa \) be a weakly compact cardinal. The relation \( 
\sqsubseteq^\kappa_{\mathsf{TREE}} \) is complete for analytic quasi-orders.
\end{corollary}

\begin{proof}
Let \( R = \proj[T] \) be an analytic quasi-order, let \( f \) be the map defined 
in~\eqref{eqf}, and assume \( x \mathrel{R} y \). By Remark~\ref{remspecialphi},
there is a Lipschitz \( \varphi \) witnessing \( s_T(x) \leq_{\max} s_T(y) \) 
such that for every \( s \in \pre{\SUCC(< \kappa)}{\kappa} \), \( \varphi(s) = s \oplus 
t \) for some \( t \in \pre{\SUCC(< \kappa)}{\kappa} \). By Proposition~\ref{propoplus}(\ref{propoplusc2a}) we then have \( \# s \leq \# \varphi(s) \) for every \( 
s \in \pre{\SUCC(< \kappa)}{\kappa} \). Therefore, by Theorem~\ref{theorcomplete}(\ref{theorcomplete1})
\[
 f(x) = G_{s_T(x)} \sqsubseteq G_{s_T(y)} = f(y). 
\]
Conversely, by Theorem~\ref{theorcomplete}(\ref{theorcomplete1}) and 
Lemma~\ref{lemmamax}
\[ 
f(x) \sqsubseteq f(y) \iff s_T(x) \leq_{\max} s_T(y) \iff x \mathrel{R} y. 
\] 

Finally, a straightforward routine computation shows that \( f \) is Borel, 
hence \( R \leq_B {\sqsubseteq^\kappa_{\mathsf{TREE}}} \).
\end{proof}

\begin{remark} \label{rembaumgartner}
Corollary~\ref{corcompletetree} can be seen as a generalization of the following result from Baumgartner's~\cite{bau} (see also Section~\ref{sectionquestions} for a further discussion on this topic). Let \( \kappa \) be a regular uncountable cardinal, let \( \mathsf{STAT}  \subseteq \pre{\kappa}{2} \) be the collection of all stationary subsets of \( \kappa \) (where a subset of \( \kappa \) is identified with its characteristic function), and let \( \subseteq^{\mathsf{NSTAT}} \) be the relation of inclusion modulo a nonstationary set, i.e.\ for every \( X,Y \in \pre{\kappa}{2} \)
\[ 
{X \subseteq^{\mathsf{NSTAT}} Y} \iff {X \setminus Y \notin \mathsf{STAT}}.
 \] 
In~\cite{bau}, it is shown that there is a map assigning to each  \( X \in \mathsf{STAT} \) a linear order \( L_X \) (hence, in particular, a tree) of size \( \kappa \) in such a way that that  for \( X,Y \in \mathsf{STAT} \), 
\[ 
{X \subseteq^{\mathsf{NSTAT}} Y} \iff {L_X \sqsubseteq L_Y}.
\] 
Therefore, such construction yields a reduction of the quasi-order \( \subseteq^{\mathsf{NSTAT}} \) on \( \mathsf{STAT}  \)
to the relation \( \sqsubseteq^\kappa_{\mathsf{LO}} \). 
However, the set \( \mathsf{STAT} \) is not Borel (in fact, it is a proper coanalytic set), hence the 
quasi-order under discussion is not an analytic quasi-order according to our Definition~\ref{defanalyticqo}. Nevertheless, if one consider the relation \( \subseteq^{\mathsf{NSTAT}} \) on the whole \( \pre{\kappa}{2} \), one gets an analytic quasi-order (denoted by \( \subseteq^{\mathsf{NSTAT}} \) again) which is very close to the one considered in~\cite{bau} --- in fact, all sets  \( X \in \pre{\kappa}{2} \setminus \mathsf{STAT} \) are in \( \subseteq^{\mathsf{NSTAT}} \)-relation with any  \( Y \in \pre{\kappa}{2} \), while no \( X \in \mathsf{STAT} \) can be in \( \subseteq^{\mathsf{NSTAT}} \)-relation with a \( Y \in \pre{\kappa}{2} \setminus \mathsf{STAT} \). Therefore Baumgartner's result can be interpreted as a slight weakening of the assertion \( {\subseteq^{\mathsf{NSTAT}}} \leq_B {\sqsubseteq^\kappa_{\mathsf{LO}}} \), and hence of \( {\subseteq^{\mathsf{NSTAT}}} \leq_B {\sqsubseteq^\kappa_{\mathsf{TREE}}} \): but when \( \kappa \) is weakly compact, this last statement is just an instantiation of Corollary~\ref{corcompletetree}.
\end{remark}

\begin{remark} \label{rembiinterpretation}
Let \( \L \) be the graph language consisting of just one binary relational 
symbol. Using the construction contained in \cite[Theorem 5.5.1]{hodges}, 
one sees that for every countable relational language \( \L' \) and any \( 
\L'_{\kappa^+ \kappa} \)-sentence \( \upvarphi \) there is an \( 
\L_{\kappa^+ \kappa} \)-sentence \(\uppsi\) such that all models of \( \uppsi 
\) are connected graphs and \( \Mod^\kappa_\upvarphi \) and \( 
\Mod^\kappa_\uppsi \) are bi-interpretable. In particular, there is a bijection \( 
b \colon \Mod^\kappa_\upvarphi \to \Mod^\kappa_\uppsi \) such that for all 
\( X,Y \in \Mod^\kappa_\upvarphi \)
\begin{align*}
X \cong Y & \iff b(X) \cong b(Y), \text{ and} \\
X \sqsubseteq Y & \iff b(X) \sqsubseteq b(Y).
\end{align*}
A straightforward computation shows that \( b \) is actually a   
homeomorphism. 
\end{remark}

Using Remark~\ref{rembiinterpretation}, we can extend Corollary~\ref{corcompletetree} to the relation \( \sqsubseteq^\kappa_{\mathsf{GRAPH}} \).

\begin{corollary} \label{corcompletegraph}
Let \( \kappa \) be a weakly compact cardinal. Then 
\( \sqsubseteq^\kappa_{\mathsf{GRAPH}} \)  is complete for analytic 
quasi-orders.
\end{corollary}

\begin{proof}
Let \( R \) be an analytic quasi-order. By Corollary~\ref{corcompletetree}, \( R \leq_B {\sqsubseteq^\kappa_{\mathsf{TREE}}} \). Therefore the composition of any witness of this fact with the function \( b \) of Remark~\ref{rembiinterpretation} (where \( \L' = \L \) and \( \upvarphi \) is the \( \L_{\kappa^+ \kappa} \)-sentence axiomatizing trees) is a Borel reduction of \( R \) into \( \sqsubseteq^\kappa_{\mathsf{GRAPH}} \).
\end{proof}

\section{The main result: \( \sqsubseteq^\kappa_{\mathsf{TREE}} \) is (strongly) invariantly universal} \label{sectionmain}

Let \( \L = \{ \preceq \} \) be the tree language consisting of one binary relational symbol, and let \( \kappa \) be an inaccessible cardinal. For the rest of this section, \( X , Y \) will denote arbitrary \( \L \)-structures of size \( \leq \kappa \).

As a first step, we provide an \( \L_{\kappa^+ \kappa} \)-sentence \( \Uppsi \) such that
\begin{enumerate}[(1)]
\item
for every \( X \in \Mod^\kappa_\Uppsi \), \( X \) is a tree;
\item
for every DST-tree \( \mathcal{T} \) on \( 2 \times \kappa \) of height \( \kappa \), \( G_{\mathcal{T}} \vDash \Uppsi \);
\item
the relation of isomorphism \( \cong \) on \( \Mod^\kappa_\Uppsi \) is \emph{smooth}, i.e.\ there is a map \( h \colon \Mod^\kappa_\Uppsi \to \pre{\chi}{2} \), where \( \chi = \pre{\SUCC(< \kappa)}{2} \times \pre{\SUCC(< \kappa)}{\kappa} \), such that for every \( X,Y \in \Mod^\kappa_\Uppsi \)
\[ 
X \cong Y \iff h(X) = h(Y).
 \] 
\end{enumerate}
 
We let \( x \prec y \), \( x \not\preceq y \), \( x \perp y \), and \( x \not\perp y \)  be abbreviations for, respectively, \( {x \preceq y} \wedge {x \neq y} \), \( \neg (x \preceq y ) \), \( x \not\preceq y \wedge y \not\preceq x \), and \( x \preceq y \vee y \preceq x \).
Let \( X \) be an \( \L \)-structure of size \( \leq \kappa \), and let \( i \colon X \to \kappa \) be an injection. We denote by 
\[ 
\uptau^i_{\mathsf{qf}}(X)  (\langle \V_\alpha \mid \alpha \in \range(i) \rangle )
 \] 
the \emph{quantifier free type of \( X \) (induced by \( i \))}, i.e.\ the formula
\[ 
\bigwedge_{\substack{x,y \in X \\  x \neq y}} (\V_{i(x)} \neq \V_{i(y)}) \wedge \bigwedge_{\substack{x,y \in X \\ x \preceq^X y}} (\V_{i(x)} \preceq \V_{i(y)}) \wedge \bigwedge_{\substack{x,y \in X \\ x \not\preceq^X y}} \V_{i(x)} \not\preceq \V_{i(y)}.
 \] 
Notice that \( \uptau^i_{\mathsf{qf}}(X)(\langle \V_\alpha \mid \alpha \in 
\range(i) \rangle) \) 
is an \( \L_{\kappa^+ \kappa} \)-formula if and only if \( |X| < \kappa \). 
Moreover, if \( Y \) is an \( \L \)-structure and \( \langle a_\alpha \mid \alpha 
\in \range(i) \rangle , \langle b_\alpha \mid \alpha \in \range(i) \rangle \) 
are two sequences of elements of 
\( Y \) such that both 
\( Y \vDash \uptau^i_{\mathsf{qf}}(X)[\langle a_\alpha \mid \alpha \in 
\range(i) \rangle]\) and 
\( Y \vDash \uptau^i_{\mathsf{qf}}(X)[\langle b_\alpha \mid \alpha \in 
\range(i) \rangle] \), 
then \( Y \restriction \{ a_\alpha \mid \alpha \in \range(i) \} \) and 
\( Y \restriction \{ b_\alpha \mid \alpha \in \range(i) \} \) are isomorphic (in 
fact, they are 
isomorphic to \( X \)). In order to simplify the notation, since the choice 
of \( i \) is often irrelevant we will drop the reference to \( i \), replace variables 
with metavariables, and call the resulting expression \emph{qf-type of \( X \)}. 
Hence in general  we will denote the qf-type of an \( \L \)-structure \( X \) by
\[ 
\uptau_{\mathsf{qf}}(X) (\langle x_i \mid i \in X \rangle ).
 \]

First let \( \Upphi_0 \) be the \( \L_{\kappa^+ \kappa} \)-sentence axiomatizing 
trees, i.e.\ the first order sentence

\begin{multline}\tag{\( \Upphi_0 \)}
\forall x  \left(x \preceq x \right) \wedge
\forall x \, \forall y  \left({{x \preceq y} \wedge {y \preceq x}} \Rightarrow {x = y}\right ) \wedge \\
\forall x \, \forall y \, \forall  z  \left({{x \preceq y} \wedge {y \preceq z}} \Rightarrow {x \preceq y}\right) \wedge 
\forall x \, \forall y \, \forall z  \left({{y \preceq x} \wedge {z \preceq x}} \Rightarrow {y \not\perp z}\right).
\end{multline}

Let \( \mathsf{Seq}(x) \) be the \( \L_{\kappa^+ \kappa} \)-formula
\begin{equation}\tag{\( \mathsf{Seq} \)} 
\neg \exists \langle x_n \mid n < \omega \rangle  \bigwedge_{n < m < \omega} \left ({x_n \preceq x} \wedge {x_m \prec x_n} \right),
 \end{equation}
and let \( \mathsf{Root}(x,y) \) be the \( \L_{\kappa^+ \kappa} \)-formula
\begin{equation} \tag{\( \mathsf{Root} \)}
{\mathsf{Seq}(x)} \wedge {\neg \mathsf{Seq}(y)} \wedge {x \preceq y} \wedge
\neg \exists w \left({x \prec w} \wedge {w \preceq y} \wedge {\mathsf{Seq}(w)}\right).
\end{equation}
\begin{remark}\label{remroot}
Note that if \( X \) is a tree and \( a \in X \), \( X \vDash \mathsf{Seq}[a] \) if and only if \( \pred(a) \) is well-founded, and that \( X_{\mathsf{Seq}} = \{ a \in X \mid \pred(a) \text{ is well-founded} \} \) is necessarily \( \preceq^X \)-downward closed.
Moreover, if \( a,a',b  \in X \) are such that \( X \vDash \mathsf{Root}[a,b] \) and \( X \vDash \mathsf{Root}[a',b] \), then \( a = a' \). This is because \( X \vDash \mathsf{Root}[a,b] \wedge \mathsf{Root}[a',b] \) implies \( a,a' \preceq^X b \), hence, since \( X \) is a tree, \( a \) and \( a' \) are comparable. Assume without loss of generality that \( a \preceq^X a' \): since \( \pred(a') \) is 
well-founded, \( a \neq a' \) would contradict \( X \vDash \mathsf{Root}[a,b] \). Therefore \( a  = a' \). 
\end{remark}

Let \( \Upphi_1 \) be the \( \L_{\kappa^+ \kappa} \)-sentence
\begin{equation} \tag{\( \Upphi_1 \)}
\forall y  \left[ \mathsf{Seq}(y) \vee \exists x \, \mathsf{Root}(x,y) \right] .
\end{equation}

\begin{remark}\label{remPhi_1}
Let \( X \) be a tree. Given \( a \in X_{\mathsf{Seq}} \), let \( X_a \) be the substructure of \( X \) with domain 
\begin{align*}
X_a  &= \left\{ b \in X \mid a \preceq^X b \wedge \neg \exists c  \left(a \prec^X c \preceq^X b \wedge c \in X_{\mathsf{Seq}}\right) \right\} \\
& = \left\{ b \in X \mid X \vDash \mathsf{Root}[a,b] \right\}.
\end{align*}
Assume now that \( X \vDash \Upphi_1 \). Then for every \( b \in X \) either \( b \in X_{\mathsf{Seq}} \) or \( b \) belongs to \( X_a \) for some \( a \in X_{\mathsf{Seq}} \). Moreover, each \( X_a \) is obviously \( \preceq^X \)-upward closed (i.e.\ for every \(a,b,c \in X \), if \( X \vDash \mathsf{Root}[a,b] \) and \( b \preceq^X c \) then \( X \vDash \mathsf{Root}[a,c] \)). This implies that:
\begin{itemize}
\item
if \( a,a' \in X_{\mathsf{Seq}} \) are distinct, \( b \in X_a \), and \( b' \in X_{a'} \), then \( b,b' \) are incomparable;
\item
for every \( a,a' \in X_{\mathsf{Seq}} \) and \( b \in X_a \), 
\[ 
a' \preceq^X b \iff a' \preceq^X a;
\]
\item
by Remark~\ref{remroot}, for \( a,a',b \) as above \( b \not\preceq^X a' \) (otherwise \( b \in X_{\mathsf{Seq}} \), contradicting \( b \in X_a \)).
\end{itemize}
\end{remark}

Consider now the linear order \( \ZZ  = (\ZZ, \leq ) \). Let \( \mathsf{Stem} (\langle x,x_z \mid z \in \ZZ \rangle ) \) be the \( L_{\kappa^+ \kappa} \)-formula
\begin{multline}\tag{\( \mathsf{Stem} \)}
\uptau_{\mathsf{qf}}(\ZZ)(\langle x_z \mid z \in \ZZ \rangle) \wedge  \bigwedge\nolimits_{z  \in \ZZ} \mathsf{Root}(x,x_z) \wedge \\
 \forall y \left [{\mathsf{Root}(x,y)} \Rightarrow \left (  {\bigvee\nolimits_{z \in \ZZ} y = x_z} \vee {\bigwedge\nolimits_{z \in \ZZ} x_z \prec y}\right ) \right ].
 \end{multline}

We also let \( \mathsf{Stem^\in}(x,y) \) be the \( \L_{\kappa^+ \kappa} \)-formula 
\begin{equation}\tag{\( \mathsf{Stem^\in} \)}
\exists \langle x_z \mid z \in \ZZ \rangle  \left ({\mathsf{Stem}(\langle x,x_z \mid z \in \ZZ \rangle)} \wedge {\bigvee\nolimits_{z \in \ZZ} y = x_z} \right ).
 \end{equation}

\begin{lemma} \label{lemmastemisunique}
Let \( X  \) be a tree, \( a \in X \) and \( \langle a_z \mid z \in \ZZ \rangle , \langle b_z \mid z \in \ZZ \rangle\ \in \pre{\ZZ}{X}\). If \( X \vDash \mathsf{Stem}(\langle a,a_z \mid z \in \ZZ \rangle) \) and \( X \vDash \mathsf{Stem}(\langle a,b_z \mid z \in \ZZ \rangle) \), then there is \( k \in \ZZ\) such that \( a_z = b_{z+k} \) for every \( z \in \ZZ \). In particular, \( \{ a_z \mid z \in \ZZ \} = \{ b_z \mid z \in \ZZ \} \).
\end{lemma}

\begin{proof}
Fix \( z \in \ZZ \). We claim that there is \( i \in \ZZ \) such that \( a_z = b_i \). Since \( X \vDash \mathsf{Stem}(\langle a,a_z \mid z \in 
\ZZ \rangle) \), then \( X \vDash \mathsf{Root}(a,a_z) \). 
Suppose toward a contradiction that \( b_i \preceq^X a_z\) for every \( i \in 
\ZZ \) (which in particular would imply \( b_i \neq a_j \) for every \( i,j \in \ZZ \) 
with \( z < j \)). Since \( X \restriction \{ b_i \mid i 
\in \ZZ \} \) has order type \( (\ZZ, \leq) \), then \( b_i \neq a_z \) for every \( i  \in 
\ZZ \) (otherwise \( a_z \prec^X b_{i+1} \), contradicting our assumption on \( a_z \) and the \( b_i \)'s), and thus, in particular, \( a_z \not\preceq^X b_i \). Moreover, since \( X \restriction \{ a_j \mid j<z \} \) has order type \( 
(\omega, \geq) \not\cong (\ZZ, \leq)\), there is \( \bar{\imath} \in \ZZ\) such that \( b_{\bar{\imath}} \neq a_j \) for every \( j 
< z \), and hence also \( b_{\bar{\imath}} \neq a_j \) for every \( j \in \ZZ \). Since \( X \vDash \mathsf{Stem}
(\langle a,b_z \mid z \in \ZZ \rangle) \), then \( X \vDash \mathsf{Root}[a,b_{\bar{\imath}}] \): this fact, together with the choice of \( \bar{\imath} \),
contradicts \( X \vDash {\bigvee_{z \in \ZZ} b_{\bar{\imath}} = a_z} \vee 
{\bigwedge_{z \in \ZZ} a_z \preceq b_{\bar{\imath}}} \). Therefore, \( X \not\vDash \bigwedge_{i \in \ZZ} b_i \preceq a_z \). Since \( X \vDash 
{\bigvee_{i \in \ZZ} a_z = b_i} \vee {\bigwedge_{i \in \ZZ} b_i \preceq a_z} 
\), there is \( i \in \ZZ \) such that \( a_z = b_i \), as required. 

A similar argument shows
that for every \( i \in \ZZ \) there is \( z \in \ZZ \) such that \( b_i = a_z \). 
Hence there is a bijection \( f \colon \ZZ \to \ZZ \) such that \( a_z = b_{f(z)} 
\) for every \( z \in \ZZ \). Since \( X \vDash \uptau_{\mathsf{qf}}(\ZZ) [\langle a_z  \mid z \in \ZZ 
\rangle] \wedge \uptau_{\mathsf{qf}}(\ZZ) [\langle b_z  \mid z \in \ZZ 
\rangle] \), \( f \) must be of the form \( i \mapsto i + k \) for some \( k \in 
\ZZ \).
\end{proof}

Let \( \Upphi_2 \) be the \( \L_{\kappa^+ \kappa} \)-sentence
\begin{equation}\tag{\( \Upphi_2 \)}
\forall x \left(\mathsf{Seq}(x) \Rightarrow \exists \langle x_z \mid z \in \ZZ  \rangle \, \mathsf{Stem}(\langle x, x_z \mid z \in \ZZ \rangle \right).
 \end{equation}
\begin{remark}\label{remPhi_2}
If \( X \) is a tree such that \( X \vDash \Upphi_2 \), then 
at the bottom of each \( X_a \) (for \( a \in X_{\mathsf{Seq}} \)) there is an isomorphic copy \( \sS^X_s \) of \( \ZZ \) (which from now on will be called \emph{stem of \( a \)}) such that all other points in \( X_a \) are \( \preceq^X \)-above (all the points of)  \( \sS^X_a \). Moreover, the stem of \( a \) is unique by Lemma~\ref{lemmastemisunique}.
Notice also that for \( a \in X_{\mathsf{Seq}} , b \in X \) 
\[ 
X \vDash \mathsf{Stem^\in}[a,b] \iff b \in \sS^X_a.
 \] 

\end{remark}

Given \( s \in \pre{\SUCC(< \kappa)}{\kappa} \), we let \( \mathsf{Lab}_s(\langle x,x_i \mid i \in \sL_s \rangle) \) be the \( \L_{\kappa^+ \kappa} \)-formula
\begin{multline}\tag{\( \mathsf{Lab}_s \)}
\uptau_\mathsf{qf}(\sL_s)(\langle x_i \mid i \in \sL_s \rangle)
\wedge 
\bigwedge\nolimits_{i \in \sL_s} \left(\mathsf{Root(x,x_i)}
\wedge
\neg \mathsf{Stem}^\in (x,x_i)\right) \wedge  \\
\forall y \left [\left (  {\mathsf{Root}(x,y)} \wedge {\neg \mathsf{Stem}^\in(x,y)} \wedge {\bigvee\nolimits_{i \in \sL_s} x_i \not\perp y} \right)  \Rightarrow \bigvee\nolimits_{i \in \sL_s} y = x_i  \right ].
\end{multline}
We also let \( \mathsf{Lab}_s^\in(x,y) \) be the \( \L_{\kappa^+ \kappa} \)-formula
\begin{equation}\tag{\( \mathsf{Lab}_s^\in \)}
\exists \langle x_i \mid i \in \sL_s \rangle \left (\mathsf{Lab}_s(\langle x,x_i \mid i \in \sL_s \rangle) \wedge \bigvee\nolimits_{i \in \sL_s} y = x_i \right ).
 \end{equation}

\begin{remark} \label{remLab_s}
If \( X \) is a tree, \( a \in X \), and \( \langle a_i \mid i \in \sL_s  \rangle \) is a sequence of elements of \( X \) such that \( X \vDash \mathsf{Lab}_s[\langle a,a_i \mid i \in \sL_s \rangle] \) (which implies \( a_i \in X_a \) for every \( i \in \sL_s \)), then the structure \( X \restriction \{ a_i \mid i \in \sL_s \} \) is a label of type II which is a code for \( s \). Moreover, if \( X \vDash \Upphi_2 \) then \( X \restriction \{ a_i \mid i \in \sL_s \} \) is above \( \sS^X_a \) and is a maximal connected component of \( X_a \setminus \sS^X_a \)  (in particular, \( X \restriction \{ a_i \mid i \in \sL_s \} \) is \( \preceq^X \)-upward closed in both \( X_a \) and  \( X \)). This fact immediately yields the following lemma.
\end{remark}

\begin{lemma}\label{lemmadisjointorequal}
Let \( X \) be a tree, \( s,t \in \pre{\SUCC(< \kappa)}{\kappa} \), and \( \langle a,a_i \mid i \in \sL_s \rangle, \langle a,b_j \mid j \in \sL_t \rangle \) be sequences of elements of \( X \) such that both \( X \vDash \mathsf{Lab}_s[\langle a,a_i \mid i \in \sL_s \rangle] \) and \( X \vDash \mathsf{Lab}_t[\langle a,b_j \mid j \in \sL_t \rangle] \). Then either the sets \( A = \{ a_i \mid i \in \sL_s \} \) and \( B = \{ b_j \mid j \in \sL_t \} \) are disjoint or they coincide (and in this second case \( s = t \)). 
\end{lemma}

\begin{proof}
Assume \( A \cap B \neq \emptyset \): we claim that \( A \subseteq B \) (the proof of \( B \subseteq A \) can be obtained in a similar way).
Suppose \( i_0 \in \sL_s, j_0 \in \sL_t \) are such that \( a_{i_0} = b_{j_0} \), and let \( i \in \sL_s \): we must show \( a_i \in B \). Since the labels are connected trees, there is \( i' \in \sL_s \) such that \( a_{i'} \preceq^X a_i,a_{i_0} \). Notice that \( X \vDash \mathsf{Root}[a,a_{i'}] \wedge \neg \mathsf{Stem}^\in[a,a_{i'}] \). Since \( a_{i'} \not\perp b_{j_0} \) by \( a_{i_0} = b_{i_0} \), \( X \vDash \mathsf{Lab}_t[\langle a,b_j \mid j \in \sL_t \rangle] \) implies \( a_{i'} \in B \). Similarly, since \( X \vDash \mathsf{Root}[a,a_{i}] \wedge \neg \mathsf{Stem}^\in[a,a_{i}] \) and \( a_i \not\perp a_{i'} \in B \), \( X \vDash \mathsf{Lab}_t[\langle a,b_j \mid j \in \sL_t \rangle] \) implies \( a_i \in B \), as required.

 The fact that if \( A = B \) then \( s = t \) follows from the fact that  \( X \restriction  A \) and \( X \restriction B \) are isomorphic, respectively, to \( \sL_s \) and \( \sL_t \), and that \( \sL_s \cong \sL_t \iff s=t \) by Lemma~\ref{lemmaL_s}.
\end{proof}

Now let \( \mathsf{Seq}_s(x) \) be the \( \L_{\kappa^+ \kappa} \)-formula
\begin{equation} \tag{\( \mathsf{Seq}_s \)}
\exists \langle x_i \mid i \in \sL_s \rangle \left (\mathsf{Lab}_s(\langle x,x_i \mid i \in \sL_s \rangle) \right).
\end{equation}
Notice that if \( a \) is a point of a tree \( X \), \( X \vDash \mathsf{Seq}_s[a] \) implies \( X \vDash \mathsf{Seq}[a] \) (hence \( a \in X_{\mathsf{Seq}} \)).

Let \(\Upphi_3\) be the \( \L_{\kappa^+ \kappa} \)-sentence
\begin{multline} \tag{\( \Upphi_3 \)}
\forall x \bigwedge\nolimits_{s,t \in \pre{\SUCC(< \kappa)}{\kappa}} \forall \langle x_i \mid i \in \sL_s \rangle \, \forall \langle y_j \mid j \in \sL_t \rangle \\
\left [  \mathsf{Lab}_s(\langle x, x_i \mid i \in \sL_s \rangle) \wedge \mathsf{Lab}_t(\langle x, y_j \mid j \in \sL_t \rangle)  
 \Rightarrow \bigvee\nolimits_{\substack{i \in \sL_s \\ j \in \sL_t}} x_i = y_j \right ].
 \end{multline}

\begin{lemma} \label{lemmauniqueL_s}
Let \( X \) be a tree such that \( X \vDash \Upphi_3 \). Then for every \( a \in X \) there is at most one \( s \in \pre{\SUCC(< \kappa)}{\kappa} \) such that \( X \vDash \mathsf{Seq}_s[a] \). Moreover, if \( X \vDash \mathsf{Seq}_s[a] \) then the set of witnesses  \( \{ a_i \mid i \in \sL_s \} \subseteq X \) of this fact is unique.
\end{lemma}

\begin{proof}
Let \( a \in X \) and \( s,t \in \pre{\SUCC(< \kappa)}{\kappa} \) be such that \( X \vDash \mathsf{Seq}_s[a] \) and \( X \vDash \mathsf{Seq}_t[a] \),  and let \( \langle a_i \mid i \in \sL_s \rangle, \langle b_j \mid j \in \sL_t \rangle \) be two sequences of points from \( X \) witnessing these facts. Then by \( X \vDash \Upphi_3 \) the sets \( A = \{ a_i \mid i \in \sL_s \} \) and \( B = \{ b_j  \mid j \in \sL_t \} \) are not disjoint. Therefore  \(A = B \) by Lemma~\ref{lemmadisjointorequal}, and hence \( s = t \), as required.
\end{proof}

\noindent
If \( X,s,a, \{ a_i \mid i \in \sL_s \}  \) are such that \( X \vDash \Upphi_3 \) and \( X \vDash \mathsf{Lab}_s[\langle a,a_i \mid i \in \sL_s \rangle] \), we denote \( X \restriction \{ a_i \mid i \in \sL_s \}  \) by \( \sL^X_{s,a} \).
Notice also that for \( a,b \in X \), \( X \vDash 
\mathsf{Lab}^\in_s[a,b] \iff b \in \sL^X_{s,a} \).

Let \( \Upphi_4 \) be the \( \L_{\kappa^+ \kappa} \)-sentence
\begin{multline} \tag{\( \Upphi_4 \)}
\bigwedge\nolimits_{s \in \pre{\SUCC(< \kappa)}{\kappa}} \exists x \left(\mathsf{Seq}_s(x) \wedge \forall y \left(\mathsf{Seq}_s(y) \Rightarrow y = x \right) \right) \wedge \\
\forall x \left (\mathsf{Seq}(x) \Rightarrow \bigvee\nolimits_{s \in \pre{\SUCC(< \kappa)}{\kappa}} \mathsf{Seq}_s(x) \right ).
 \end{multline}

\begin{remark}\label{remsigma_X}
If \( X \) is a tree such that \( X \vDash \Upphi_4 \), then there is a surjection \( \sigma_X \) of \( \pre{\SUCC(<\kappa)}{\kappa} \) onto \( X_{\mathsf{Seq}}  \), namely \( \sigma_X(s) = \) the unique \( a \in X_{\mathsf{Seq}} \) such that \( X \vDash \mathsf{Seq}_s[a] \). Moreover, if \( X \) further satisfies \( \Upphi_3 \), then \( \sigma_X \) is also injective, hence a bijection.
\end{remark}

Let \( \Upphi_5 \) be the \( \L_{\kappa^+ \kappa} \)-sentence
\begin{multline} \tag{\( \Upphi_5 \)}
\forall x, y \left [ {\bigwedge\nolimits_{\substack{s,t \in \pre{\SUCC(< \kappa)}{\kappa} \\ s \subseteq t }} ({\mathsf{Seq}_s(x) \wedge \mathsf{Seq}_t(y)} \Rightarrow {s \preceq t})} \, \wedge \right. \\
\left. {\bigwedge\nolimits_{\substack{s,t \in \pre{\SUCC(< \kappa)}{\kappa} \\ s \not\subseteq t }} ({\mathsf{Seq}_s(x) \wedge \mathsf{Seq}_t(y)} \Rightarrow {s \not\preceq t})} \right ].
 \end{multline}

\begin{remark} \label{remPhi_5}
If \( X \) is a tree such that \( X \vDash \Upphi_3 \wedge \Upphi_4 \wedge \Upphi_5 \), then the map \( \sigma_X \) defined in Remark~\ref{remsigma_X} is actually an isomorphism between \( \left( \pre{\SUCC(< \kappa)}{\kappa}, \subseteq \right) \) and \( X \restriction X_{\mathsf{Seq}} \).
\end{remark}

Given \( u \in \pre{\SUCC(< \kappa)}{2} \), we let \( \mathsf{Lab}^*_u(\langle x,x_i \mid i \in \sL^*_u \rangle) \) be the \( \L_{\kappa^+ \kappa} \)-formula
\begin{multline} \tag{\( \mathsf{Lab}^*_u \)}
\uptau_\mathsf{qf}(\sL^*_u)(\langle x_i \mid i \in \sL^*_u \rangle)
\wedge 
\bigwedge\nolimits_{i \in \sL^*_u}\left (\mathsf{Root(x,x_i)}
\wedge
\neg \mathsf{Stem}^\in (x,x_i)\right ) \wedge  \\
\forall y \left [\left ({\mathsf{Root}(x,y)} \wedge {\neg \mathsf{Stem}^\in(x,y)} \wedge {\bigvee\nolimits_{i \in \sL^*_u} x_i \not\perp y} \right ) \Rightarrow {\bigvee\nolimits_{i \in \sL^*_u} y = x_i} \right ].
\end{multline}
We also let \( \mathsf{Lab}^{*\, \in}_u(x,y) \) be the \( \L_{\kappa^+ \kappa} \)-formula
\begin{equation}\tag{\( \mathsf{Lab}^{* \, \in}_u \)}
\exists \langle x_i \mid i \in \sL^*_u \rangle \left (\mathsf{Lab}^*_u(\langle x,x_i \mid i \in \sL^*_u \rangle) \wedge \bigvee\nolimits_{i \in \sL^*_u} y = x_i \right ).
 \end{equation}

\begin{remark} \label{remLab_u}
Similarly to the case of \( \mathsf{Lab}_s(\langle x_i \mid i \in \sL_s \rangle) 
\), if \( X \) is a tree of size \( \kappa \), \( a \in X \), and \( \langle a_i \mid i 
\in \sL^*_u  \rangle \) is a sequence of elements of \( X \) such that \( X 
\vDash \mathsf{Lab}^*_u[\langle a,a_i \mid i \in \sL^*_u \rangle] \), then \( 
X \restriction \{ a_i \mid i \in \sL^*_u \} \) is a label of type III which is a 
code for \( u \). Moreover, if \( X \vDash \Upphi_2 \) then \( X \restriction \{ a_i 
\mid i \in \sL^*_u \} \) is above \( \sS^X_a \) and is again a maximal 
connected component of \( X_a \setminus \sS^X_a \)  (in particular, \( X 
\restriction \{ a_i \mid i \in \sL_s \} \) is \( \preceq^X \)-upward closed in both 
 \( X_a \) and  \( X \)).
\end{remark}

\begin{remark} \label{remuniqueLab_u}
Using the argument contained in the proof of Lemma~\ref{lemmadisjointorequal}, one can show that if \( u,v \in \pre{\SUCC(< \kappa)}{2} \) and \( \langle a,a_i \mid i \in \sL^*_u \rangle, \langle a,b_j \mid j \in \sL^*_v \rangle \) are sequences of elements of \( X \) such that both \( X \vDash \mathsf{Lab}_u[\langle a,a_i \mid i \in \sL^*_u \rangle] \) and \( X \vDash \mathsf{Lab}_v[\langle a,b_j \mid j \in \sL^*_v \rangle] \), then either the sets \( A = \{ a_i \mid i \in \sL^*_u \} \) and \( B = \{ b_j \mid j \in \sL^*_v \} \) are disjoint, or they coincide (and in this case \( u = v \) by Lemma~\ref{lemmaL_u}). 

Similarly, if \( u \in \pre{\SUCC(< \kappa)}{\kappa} \), 
\( s \in \pre{\SUCC(< \kappa)}{\kappa} \), and 
\( \langle a,a_i \mid i \in \sL^*_u \rangle, \langle a,b_j \mid j \in \sL_s \rangle \) 
are sequences of elements of \( X \) such that both 
\( X \vDash \mathsf{Lab}_u[\langle a,a_i \mid i \in \sL^*_u \rangle] \) and 
\( X \vDash \mathsf{Lab}_s[\langle a,b_j \mid j \in \sL_s \rangle] \), 
then the sets \( \{ a_i \mid i \in \sL^*_u \} \) and \( \{ b_j \mid j \in \sL_s \} \) 
are disjoint (they cannot coincide because \( \sL^*_u \not\cong \sL_s \) by Lemma~\ref{lemmalabelsdifferenttypes}).
\end{remark}

Let \( \Upphi_6 \) be the following \( \L_{\kappa^+ \kappa} \)-sentence:
\begin{multline}\tag{\( \Upphi_6 \)}
\forall x \bigwedge\nolimits_{u \in \pre{\SUCC(< \kappa)}{\kappa}} \forall \langle x_i \mid  i \in \sL^*_u \rangle \, \forall \langle y_j \mid j \in \sL^*_u \rangle \\
\left ( \mathsf{Lab}^*_u(\langle x, x_i \mid i \in \sL^*_u \rangle) \wedge \mathsf{Lab}^*_u(\langle x,y_j \mid j \in \sL^*_u \rangle) \Rightarrow \bigvee\nolimits_{i,j \in \sL^*_u} x_i = y_j \right ).
 \end{multline}

\begin{remark} \label{remPhi_6}
If \( X \vDash \Upphi_6 \) then for all \( a \in X_{\mathsf{Seq}} \), if \( \langle 
a_i \mid i \in \sL^*_u \rangle, \langle b_j \mid j \in \sL^*_u \rangle \) are 
sequences of elements of \( X \) such that both \( X \vDash 
\mathsf{Lab}^*_u[\langle a_i \mid i \in \sL^*_u \rangle] \) and \( X \vDash 
\mathsf{Lab}^*_u[\langle b_j \mid j \in \sL^*_u \rangle] \) then \( \{ a_i 
\mid i \in \sL^*_u \} = \{ b_j \mid j \in \sL^*_u \} \). Roughly speaking, this 
means that there can be at most one substructure of \( X_a \) which is above 
\( \sS^X_a \), is a maximal connected component of \( X_a \setminus \sS^X_a 
\), and is a code for \( u \): such a substructure, if it exists, will be denoted by 
\( \sL^{*\, X}_{u,a} \). Notice also that for \( a,b \in X \), \( X \vDash 
\mathsf{Lab}^{* \, \in}_u[a,b] \iff b \in \sL^{*\, X}_{u,a} \).
\end{remark}

Let \( \mathsf{Lab_I}^\in(x,y) \) be the \( \L_{\kappa^+ \kappa} \)-formula
\begin{multline} \tag{\( \mathsf{Lab_I}^\in \)}
\mathsf{Root}(x,y) \wedge \neg \mathsf{Stem}^\in(x,y) \wedge \\ 
\bigwedge_{s \in \pre{\SUCC(<\kappa)}{\kappa}} \neg 
\mathsf{Lab}^\in_s(x,y) \wedge \bigwedge_{u \in \pre{\SUCC(< \kappa)}
{2}} \neg \mathsf{Lab}^{*\, \in}_u(x,y).
 \end{multline}

\begin{remark} \label{remLab_I}
 Notice that if \( X \) is a tree and \( a,b,c \in X \) are such that \( X \vDash \mathsf{Lab_I}^\in[a,b] \) and \( b \preceq^X c \), then \( X \vDash \mathsf{Lab_I}^\in[a,c] \). 
\end{remark}

Given \( \alpha < \kappa \), consider the structure \( \alpha= (\alpha, \leq ) \). Then let \( \mathsf{Lab_I}^{\alpha}(\langle x,y,z_i \mid i \in \alpha \rangle) \) be the 
\( \L_{\kappa^+ \kappa} \)-sentence
\begin{multline}\tag{\( \mathsf{Lab_I}^{\alpha} \)}
{\mathsf{Lab_I}^\in(x,y)} \wedge { \exists w \, \exists w' \left({y \preceq w} \wedge {y \preceq w'} \wedge {w \perp w'}\right ) }\wedge {\bigwedge\nolimits_{i \in \alpha} (y \prec z_i)} \wedge \\
 {\tau_{\mathsf{qf}}(\alpha)(\langle z_i \mid i \in \alpha \rangle)} \wedge {
\forall w \left ( {{y \prec w} \wedge {\bigvee\nolimits_{i \in \alpha} w \not\perp z_i}} \Rightarrow {\bigvee\nolimits_{i \in \alpha} w = z_i} \right )}.
 \end{multline}

Let \( \Upphi_7 \) be the \( \L_{\kappa^+ \kappa} \)-sentence
\begin{multline}\tag{\( \Upphi_7 \)} 
\forall x \, \forall y \bigwedge\nolimits_{\alpha,\beta < \kappa} \forall \langle z_i \mid i \in \alpha \rangle \, \forall \langle w_j \mid j \in \beta \rangle \\
\left ( \mathsf{Lab_I}^\alpha(\langle x,y,z_i \mid i \in \alpha \rangle) \wedge \mathsf{Lab_I}^\beta(\langle x,y,w_j \mid j \in \beta \rangle) \Rightarrow \bigvee\nolimits_{\substack{i \in \alpha \\ j \in \beta}} z_i = w_j \right ).
\end{multline}

\begin{remark} \label{remPhi_7}
The same argument contained in the proof of Lemma~\ref{lemmadisjointorequal} gives the following:
Let \( a, b \in X \) (for \( X \) a tree). Let \( \alpha, \beta  < \kappa\) and \( \langle c_i \mid i \in \alpha \rangle, \langle d_j \mid j \in \beta \rangle \) be sequences of elements of \( X \) such that both \( X \vDash \mathsf{Lab_I}^\alpha[\langle a,b,c_i \mid i \in \alpha \rangle] \) and \( X \vDash \mathsf{Lab_I}^\beta[\langle a,b,d_j \mid j \in \beta \rangle] \). Then the sets \( C = \{ c_i \mid i \in \alpha \} \) and \( D = \{ d_j \mid j \in \beta \} \) are either disjoint or coincide.
Therefore, if \( X \vDash \Upphi_7\) then \( C = D \), and hence \( \alpha = \beta \).
\end{remark}

Let \( \Upphi_8 \) be the \( \L_{\kappa^+ \kappa} \)-sentence
\begin{multline}\tag{\( \Upphi_8 \)}
\forall x \, \forall y \left [  \mathsf{Lab_I}^\in(x,y)  \Rightarrow  \left ( \vphantom{\bigvee\nolimits_{\alpha < \kappa}}{ \exists w \, \exists w' \left({y \preceq w} \wedge {y \preceq w'} \wedge {w \perp w'}\right) } \vee \right. \right. \\ 
\left. \left. \exists z \bigvee\nolimits_{\alpha < \kappa}\exists \langle w_i \mid i \in \alpha \rangle \left (\mathsf{Lab_I}^\alpha(\langle x,z,w_i \mid i \in \alpha \rangle) \wedge \bigvee\nolimits_{i \in \alpha} y = w_i \right ) \right ) \right ].
 \end{multline}

For \( \alpha < \kappa \), let \( \mathsf{Lab_I}^{= \alpha } (x,y) \) be the 
\( \L_{\kappa^+ \kappa} \)-formula
\begin{equation}\tag{\( \mathsf{Lab_I}^{= \alpha }  \)}
\exists \langle z_i \mid i \in \alpha \rangle\left (\mathsf{Lab_I}^\alpha(\langle x,y,z_i \mid i \in \alpha \rangle)\right).
 \end{equation}

Let \( \Upphi_9 \) be the \( \L_{\kappa^+ \kappa} \)-sentence
\begin{multline}\tag{\( \Upphi_9 \)}
\forall x \left [ \mathsf{Seq}(x)  \Rightarrow  \left ( \bigwedge\nolimits_{\alpha < \kappa} \exists y \left(\mathsf{Lab_I}^{= \alpha}(x,y) \wedge \forall z \left(\mathsf{Lab_I}^{=\alpha}(x,z) \Rightarrow z = y \right)\right) \right )  \wedge \right. \\
\left. \forall y \left ( {\mathsf{Lab_I}^\in(x,y) \wedge { \exists w \, \exists w' \left({y \preceq w} \wedge {y \preceq w'} \wedge {w \perp w'}\right) }} \Rightarrow \bigvee\nolimits_{\alpha < \kappa} \mathsf{Lab_I}^{=\alpha}(x,y) \right )\right ].
 \end{multline}

\begin{remark} \label{remPhi_7-9}
Let \( X \) be a tree such that \( X \vDash \Upphi_2 \wedge \Upphi_3 \wedge \Upphi_4 
\wedge \Upphi_5 \) and \( a \in X_{\mathsf{Seq}} \). Then some of the points in 
\( X_a \) belong to the stem \( \sS^X_a \) of \( a \), while some others belong 
to maximal connected components of \( X_a \setminus \sS^X_a \) which are 
labels of type II or III (namely, to \( \sL^X_{s,a} \), where \( s \in 
\pre{\SUCC(< \kappa)}{\kappa} \) is the unique sequence such that \( X 
\vDash \mathsf{Seq}_s[a] \), or to \( \sL^{*\, X}_{u,a} \) for some \( u \in 
\pre{\SUCC(< \kappa)}{\kappa} \)). Let \( X'_a \) be the substructure of \( 
X_a \) whose domain consists of the elements of \( X_a \) which does not 
belong to any of the categories mentioned above, i.e.\ \( X'_a = \left\{ b \in X_a 
\mid X \vDash \mathsf{Lab_I}^\in[a,b] \right\} \). Notice that \( X'_a \) is \( 
\preceq^X \)-upward closed (both in \( X \) and in \( X_a \)) by Remark~\ref{remLab_I}. Then if \( X 
\vDash \Upphi_9 \) there is a surjection \( l_a \) from \( \kappa \) onto the points 
\( b \in X'_a \) such that \( \cone(b) \) is not a linear order, namely \( l_a(\alpha) = \) the unique \( b \in X'_a \) such that \( X \vDash \mathsf{Lab_I}^{=\alpha}[a,b] \) (for \(\alpha < \kappa \)). If \( X \) satisfies 
also \( \Upphi_7 \) then \( l_a \) is actually a bijection, and if moreover \( X 
\vDash \Upphi_8 \)  then each remaining point, i.e.\ each point \( c \in X'_a \) 
such that \( \cone(c) \) is a linear order, belongs to the unique (by Remark~\ref{remPhi_7}) sequence witnessing \( X \vDash \mathsf{Lab_I}^{=\alpha}[a,l_a(\alpha)] \) 
(for some \( \alpha < \kappa \)).
\end{remark}

Finally, let \( \Upphi_{10} \) be the \( \L_{\kappa^+ \kappa} \)-sentence
\begin{multline}\tag{\( \Upphi_{10} \)} 
\forall x \bigwedge_{s \in \pre{\SUCC(< \kappa)}{\kappa}} \left [ \mathsf{Seq}_s(x) \Rightarrow  \left ( \bigwedge\nolimits_{\substack{\alpha,\beta < \kappa \\ \alpha \preceq^{L_{\bar{\gamma}(s)}} \beta}} \forall y \, \forall z \left({\mathsf{Lab_I}^{= \alpha}(x,y) \wedge \mathsf{Lab_I}^{= \beta}(x,z)} \Rightarrow {y \preceq z}\right) \right ) \wedge\right. \\
\left. \left ( \bigwedge\nolimits_{\substack{\alpha,\beta < \kappa \\ \alpha \not{\preceq^{L_{\bar{\gamma}(s)}}} \beta}} \forall y \, \forall z \left ({\mathsf{Lab_I}^{= \alpha}(x,y) \wedge \mathsf{Lab_I}^{= \beta}(x,z)} \Rightarrow {y \not\preceq z}\right) \right ) \right ].
 \end{multline}

\begin{remark} \label{remPhi_10}
Let \( X \) be a tree such that \( X \vDash \Upphi_7 \wedge \Upphi_8 \wedge \Upphi_9 \wedge \Upphi_{10} \). Using Remark~\ref{remPhi_7-9}, if \( s \in \pre{\SUCC(< \kappa)}{\kappa} \) and \( a \in X_{\mathsf{Seq}} \) are such that \( X \vDash \mathsf{Seq}_s[a] \), then  \( X \restriction X'_a \)  is isomorphic to \( \sL_{\bar{\gamma}(s)} \). In this case, the structure \( X \restriction X'_a \) will be denoted by \( \sL^X_{\bar{\gamma}(s),a} \).
\end{remark}

\begin{defin}\label{defPsi}
Let now \( \Uppsi \)  be the \( \L_{\kappa^+ \kappa} \)-sentence given by the  conjunction
\begin{equation}\tag{\( \Uppsi \)}
\bigwedge\nolimits_{i \leq 10} \Upphi_i.
 \end{equation}
\end{defin}

\begin{remark} \label{remrecap}
Suppose \( X \vDash \Uppsi \). Collecting all the remarks above, we have the following description of \( X \):
\begin{enumerate}[(1)]
\item \label{condtree}
\( X \) is a tree (by \( X \vDash \Upphi_0 \));
\item \label{condsigma_X}
there is an isomorphism \( \sigma_X \) between \( \left(\pre{\SUCC(< \kappa)}
{\kappa} , \subseteq \right) \) and the substructure of \( X \) with domain \( 
X_{\mathsf{Seq}} = \{ a \in X \mid \pred(a) \text{ is well-founded} \} \), 
which is a \( \preceq^X \)-downward closed subset of \( X \) (Remarks~\ref{remsigma_X},~\ref{remPhi_5}, and~\ref{remroot});
\item \label{condX_a}
by Remark~\ref{remPhi_1}, for every point \( b \) in \( X \setminus X_{\mathsf{Seq}} \) there is a 
(unique) \( \preceq^X \)-maximal element \( a^b \) in \( \pred(b) \) which is 
in \( X_{\mathsf{Seq}} \): denote by \( X_{\sigma_X(s)} \) the collection of all
\( b  \in X \setminus X_{\mathsf{Seq}} \) such that \( a^b = \sigma_X(s) \), and notice that \( X_{\sigma_X(s)} 
\) is necessarily \( \preceq^X \)-upward closed. 
Moreover, for every \( s,t  \in \pre{\SUCC(< \kappa)}{\kappa}\) we have (see Remark~\ref{remPhi_1}):
\begin{enumerate}[(a)]
\item
if \( s,t \) are distinct then for every \( b \in X_{\sigma_X(s)} , b' \in X_{\sigma_X(t)} \), \( b \) and \( b' \) are incomparable;
\item
if \( b  \in X_{\sigma_X(s)} \) then \( \sigma_X(t) \preceq^X b \) if and only if \( t \subseteq s \) and \( b \not\preceq^X \sigma_X(t) \);
\end{enumerate}
\item \label{condstem}
at the bottom of each \( X_{\sigma_X(s)} \) there is an isomorphic copy of \( ( \ZZ, \leq ) \), called stem of \( \sigma_X(s) \) and denoted by \( \sS^X_{\sigma_X(s)} \): all other elements of \( X_{\sigma_X(s)} \) are \( \preceq^X \)-above (all the points in) \( \sS^X_{\sigma_X(s)} \) (Remark~\ref{remPhi_2});
\item \label{condabovestem}
call a substructure \( X' \) of \( X_{\sigma_X(s)} \setminus \sS^X_{\sigma_X(s)} \) \emph{maximal} if it is a maximal connected component of \( X_{\sigma_X(s)} \setminus \sS^X_{\sigma_X(s)} \). Moreover, let \( U^X_s \) be the collection of all \( u \in \pre{\SUCC(< \kappa)}{2} \) for which there is a sequence \( \langle a_i  \mid i \in \sL^*_u \rangle \) of points in \( X_{\sigma_X(s)} \) such that \( X \vDash \mathsf{Lab}^*_u[\langle \sigma_X(s),a_i \mid i \in \sL^*_u \rangle] \). Then above the stem of \( \sigma_X(s) \) there is 
\begin{enumerate}[(a)]
\item \label{condL_s}
a (unique) substructure \( \sL_{s,\sigma_X(s)}^X \) of \( X_{\sigma_X(s)} \) which is a code for \( s \) (i.e.\ it is isomorphic to \( \sL_s \)) and is maximal  (Lemma~\ref{lemmauniqueL_s} and Remark~\ref{remsigma_X});
\item \label{condL_u}
for each \( u \in U^X_s \), a (unique) substructure \( \sL^{*\, X}_{u,\sigma_X(s)} \) of \( X_{\sigma_X(s)} \) which is a code for \( u \) (i.e.\ it is isomorphic to \( \sL^*_u \)) and is maximal (Remarks~\ref{remLab_u},~\ref{remuniqueLab_u} and~\ref{remPhi_6});
\end{enumerate}
\item \label{condL_gamma}
the remaining points above  \( \sS^X_{\sigma_X(s)} \) form a substructure \( \sL^X_{\bar{\gamma}(s),\sigma_X(s)} \) of \( X_{\sigma_X(s)} \) which is a code for \( \bar{\gamma}(s) \) (i.e.\ it is isomorphic to \( \sL_{\bar{\gamma}(s)} \)): \( \sL^X_{\bar{\gamma}(s),\sigma_X(s)} \) is necessarily maximal as well (Remark~\ref{remPhi_10}).
\end{enumerate}
\end{remark}

Therefore one immediately gets that

\begin{lemma}\label{lemmarangef}
Let \( \kappa \) be a weakly compact cardinal, \( R = \proj[T] \) be an analytic quasi-order on \( \pre{\kappa}{2} \), and \( f \) be the function defined in~\eqref{eqf}. Then \( \range(f) \subseteq \Mod^\kappa_\Uppsi \). 
\end{lemma}

Moreover, if \( X,Y \) both satisfy \( \Uppsi \) then 
\(
X \cong Y \iff \forall s \in \pre{\SUCC(< \kappa)}{\kappa} \left (U^X_s = U^Y_s \right)
 \). 
More precisely: let \( \chi = \pre{\SUCC(< \kappa)}{2 \times \pre{\SUCC(< \kappa)}{\kappa}} \), and define the map 
\begin{equation*}
h \colon \Mod^\kappa_\Uppsi \to \pre{\chi}{2} \colon X \mapsto h(X), 
\end{equation*}
by setting 
\[ 
h(X)((u,s)) = 1 \iff u \in U^X_s ,
\]
for every \( (u,s) \in \chi \) and \( X \in \Mod^\kappa_\Uppsi \),  i.e.
\begin{multline} \label{eqh}
h(X)((u,s)) = 1 \iff \\ X \vDash \forall x \left[\mathsf{Seq}_s(x) \Rightarrow \exists \langle x_i \mid i \in \sL^*_u \rangle \left (\mathsf{Lab}^*_u(\langle x,x_i \mid i \in \sL^*_u \rangle)\right)\right].
 \end{multline}

\begin{proposition}\label{propreduction}
Let \( \kappa \) be an inaccessible cardinal and \( h \) be the map defined in~\eqref{eqh}. Then \( h \) reduces the relation of isomorphism on \( \Mod^\kappa_\Uppsi \) to the relation of equality on \( \pre{\chi}{2} \), i.e.\ for every \( X,Y \in \Mod^\kappa_\Uppsi \)
\begin{equation} \label{eq:smooth}
X \cong Y \iff h(X) = h(Y).
\end{equation}
\end{proposition}

\begin{proof}
Suppose first that \( X, Y \in \Mod^\kappa_\Uppsi \) and let \( f \) be an isomorphism from \( X \) to \( Y \). Given \( (u,s) \in \chi \), we need to show 
\[
 h(X)((u,s)) = 1 \iff h(Y)((u,s)) = 1.
\]
Assume that \( h(X)((u,s)) =1 \). By Remark~\ref{remrecap}(\ref{condX_a}), 
let \( a = \sigma_X(s) \in X \), so that \( X \vDash \mathsf{Seq}_s[a] \). Then 
since \( f \) is an isomorphism, \( Y \vDash \mathsf{Seq}_s[f(a)] \). Notice 
also the necessarily \( f(a) = \sigma_Y(s) \), and recall that by definition of 
\(\sigma_Y \) (Remark~\ref{remsigma_X}), \( \sigma_Y(s) \) is the  unique \( 
b \in Y \) such that \( Y \vDash \mathsf{Seq}_s[b] \). Let now \( \langle a_i 
\mid i \in \sL^*_u \rangle \) be a sequence of elements from \( X \) such that 
\( X \vDash \mathsf{Lab}^*_u[\langle a,a_i \mid i \in \sL^*_u \rangle] \). 
Then \( Y \vDash \mathsf{Lab}^*_u[\langle f(a),f(a_i) \mid i \in \sL^*_u 
\rangle] \). Hence \( Y \vDash \forall x \left[\mathsf{Seq}_s(x) \Rightarrow \exists 
\langle x_i \mid i \in \sL^*_u \rangle \left (\mathsf{Lab}^*_u(\langle x,x_i \mid 
i \in \sL^*_u \rangle)\right)\right] \), which means \( h(Y)((u,s)) = 1 \). Using a similar argument, one can
show that if \( h(Y)((u,s)) = 1 \) then \( h(X)((u,s)) = 1 \) as well, hence we are 
done.

\medskip

We now prove the \( \Leftarrow \) direction or~\eqref{eq:smooth}. Let \( X, Y \in \Mod^\kappa_\Uppsi \) be such 
that \( h(X) = h(Y) \) (this hypothesis will be used just in~\eqref{eqhypothesis}). By Remark~\ref{remrecap}(\ref{condsigma_X}), the 
map  \( f_{\mathsf{Seq}} = \sigma_Y \circ \sigma^{-1}_X \) is an 
isomorphism between \( X_{\mathsf{Seq}} \) and \( Y_{\mathsf{Seq}} \). 
Notice that \( f_{\mathsf{Seq}}(\sigma_X(s)) = \sigma_Y(s) \). We will now 
extend \( f_{\mathsf{Seq}} \) to an isomorphism between \( X \) and \( Y \). 
By Remark~\ref{remrecap}(\ref{condX_a}), it is enough to build for every 
\( s \in \pre{\SUCC(< \kappa)}{\kappa} \) an isomorphism \( f_s \) between 
the substructures \( X_{\sigma_X(s)} \) and \( Y_{\sigma_Y(s)} \), because 
then 
\[
 f = {f_{\mathsf{Seq}}} \cup {\bigcup_{s \in \pre{\SUCC(< \kappa)}{\kappa}} f_s }
\] 
would automatically be an isomorphism between \( X \) and \( Y \). Hence in 
what follows we will fix \( s \in \pre{\SUCC(< \kappa)}{\kappa} \) and define 
four functions, denoted by \( f^i_s \) for \( i \leq 3 \), such that \( f_s = 
\bigcup_{i \leq 3} f^i_s \) is the desired isomorphism between \( 
X_{\sigma_X(s)} \) and \( Y_{\sigma_Y(s)} \). This will conclude the proof of 
Proposition~\ref{propreduction}.

By Remark~\ref{remrecap}(\ref{condstem}),  both in \( X_{\sigma_X(s)} \) 
and in \( Y_{\sigma_Y(s)} \) there are substructures \( \sS^X_{\sigma_X(s)} 
\) and \( \sS^Y_{\sigma_Y(s)} \), respectively, which are isomorphic to 
\( (\ZZ, \leq) \) and such that their elements are below any other element of, 
respectively, \( X_{\sigma_X(s)} \) and \( Y_{\sigma_Y(s)} \). Let \( f_s^0 \) 
be any isomorphism between \( \sS^X_{\sigma_X(s)} \) and \( 
\sS^Y_{\sigma_Y(s)} \), and notice that by the properties of a stem, if \( 
\hat{f} \) is an isomorphism between \( X_{\sigma_X(s)} \setminus 
\sS^X_{\sigma_X(s)} \) and \( Y_{\sigma_Y(s)} \setminus 
\sS^Y_{\sigma_Y(s)} \) then \( f_s^0 \cup \hat{f} \) is automatically an 
isomorphism between \( X_{\sigma_X(s)} \) and \( Y_{\sigma_Y(s)} \). 

By Remark~\ref{remrecap}(\ref{condL_s}), there are two substructures \( 
\sL^X_{s,\sigma_X(s)} \) and \( \sL^Y_{s,\sigma_Y(s)}\) of, respectively, \( 
X_{\sigma_X(s)} \setminus \sS^X_{\sigma_X(s)} \) and \( Y_{\sigma_Y(s)} 
\setminus \sS^Y_{\sigma_Y(s)} \)  such that both of them are isomorphic to 
\( \sL_s \) and they are maximal. Let \( f_s^1 \) be any isomorphism between 
\( \sL^X_{s,\sigma_X(s)} \) and \( \sL^Y_{s,\sigma_Y(s)} \). Then  \( f^0_s 
\cup f^1_s \) is an isomorphism between \( \sS^X_{\sigma_X(s)} \cup 
\sL^X_{s,\sigma_X(s)} \) and \( \sS^Y_{\sigma_Y(s)} \cup \sL^Y_{s,
\sigma_Y(s)} \). Moreover, by the maximality of \( \sL^X_{s,\sigma_X(s)} \) 
and \( \sL^Y_{s,\sigma_Y(s)} \) we again have that if \( \hat{f} \) is any 
isomorphism between \( X_{\sigma_X(s)} \setminus (\sS^X_{\sigma_X(s)} 
\cup \sL^X_{s,\sigma_X(s)}) \) and \( Y_{\sigma_Y(s)} \setminus 
(\sS^Y_{\sigma_Y(s)} \cup \sL^Y_{s,\sigma_Y(s)}) \)  then \( f_s^0 \cup 
f_s^1  \cup \hat{f} \) is automatically an isomorphism between \( 
X_{\sigma_X(s)} \) and \( Y_{\sigma_Y(s)} \).

Let \( U^X_s \) be defined as in Remark~\ref{remrecap}(\ref{condabovestem}).
By Remark~\ref{remrecap}(\ref{condL_u}), for each \( u \in \pre{\SUCC(< 
\kappa)}{2} \) either \( u \notin U^X_s \), or else there is a unique set \( \{ a_i \mid i \in \sL^*_u \} \) of 
witnesses for \( u \in U^X_s \): let \( \sL^{*\, X}_{u,\sigma_X(s)} \) be \( X \restriction \{ 
a_i \mid i \in \sL^*_u \} \) if \( u \in U^X_s \) and the empty structure 
otherwise, so that if \( \sL^{*\, X}_{u,\sigma_X(s)} \) is nonempty then it is 
isomorphic to \( \sL^*_u \). Define the structures \( \sL^{*\, Y}_{u,\sigma_Y(s)} \) 
analogously, replacing \( U^X_s \) with \( U^Y_s \). Since by Remark~\ref{remrecap}(\ref{condsigma_X}) \( \sigma_X(s) \) (respectively, 
\( \sigma_Y(s) \)) is the unique point \( b \) of \( X \) (resp.\ of \( Y \)) such 
that \( X \vDash \mathsf{Seq}_s[b] \) (resp.\ \( Y \vDash \mathsf{Seq}_s[b] 
\)),  by~\eqref{eqh} and \( h(X) = h(Y) \) we have 
\begin{multline}\label{eqhypothesis}
U^X_s = \left\{ u \in \pre{\SUCC(< \kappa)}{2} \mid h(X)((u,s)) = 1 \right\} = \\ \left\{ u 
\in \pre{\SUCC(< \kappa)}{2} \mid h(X)((u,s)) = 1 \right\} = U^Y_s.
\end{multline}
Therefore, for every \( u \in \pre{\SUCC(< \kappa)}{2 } \) we have that \( 
\sL^{*\, X}_{u,\sigma_X(s)} \cong \sL^{*\, Y}_{u,\sigma_Y(s)}\) (which is trivially 
satisfied if both structures are empty). Choose for every \( u \in \pre{\SUCC(< 
\kappa)}{2 } \) an isomorphism \( f_s^u \) between \( 
\sL^{*\, X}_{u,\sigma_X(s)} \) and \( \sL^{*\, Y}_{u,\sigma_Y(s)}\), and let
\[ 
f^2_s=  \bigcup_{u \in U^X_s} f^u_s.
 \] 
By maximality of the structures \( \sL^{*\, X}_{u,\sigma_X(s)},
\sL^{*\, Y}_{u,\sigma_Y(s)} \), the function \( f^0_s \cup f^1_s \cup f^2_s \) is 
an isomorphism between \( \sS^X_{\sigma_X(s)} \cup \sL^X_{s,\sigma_X(s)} 
\cup \bigcup_{u \in U^X_s} \sL^{*\, X}_{u,\sigma_X(s)}  \) and \( 
\sS^Y_{\sigma_Y(s)} \cup \sL^Y_{s,\sigma_Y(s)} \cup \bigcup_{u \in 
U^Y_s} \sL^{*\, Y}_{u,\sigma_Y(s)} \), and if \( \hat{f} \) is an isomorphism 
between \( X_{\sigma_X(s)} \setminus  \left (\sS^X_{\sigma_X(s)} \cup 
\sL^X_{s,\sigma_X(s)} \cup \bigcup_{u \in U^X_s} \sL^{*\, X}_{u,\sigma_X(s)} 
\right) \) and \( Y_{\sigma_Y(s)} \setminus \left (\sS^Y_{\sigma_Y(s)} \cup 
\sL^Y_{s,\sigma_Y(s)} \cup \bigcup_{u \in U^Y_s} \sL^{*\, Y}_{u,\sigma_Y(s)} 
\right) \) then \( f^0_s \cup f^1_s \cup f^2_s \cup \hat{f} \) is  an 
isomorphism between \( X_{\sigma_X(s)} \) and \( Y_{\sigma_Y(s)} \).

Finally, by condition Remark~\ref{remrecap}(\ref{condL_gamma}), the 
substructures \( \sL^X_{\bar{\gamma}(s),\sigma_X(s)} \) and \( 
\sL^Y_{\bar{\gamma}(s),\sigma_Y(s)} \) with domain, respectively, \( 
X_{\sigma_X(s)} \setminus  \left (\sS^X_{\sigma_X(s)} \cup 
\sL^X_{s,\sigma_X(s)} \cup \bigcup_{u \in U^X_s} \sL^{*\, X}_{u,\sigma_X(s)} 
\right) \) and \( Y_{\sigma_Y(s)} \setminus \left (\sS^Y_{\sigma_Y(s)} \cup 
\sL^Y_{s,\sigma_Y(s)} \cup \bigcup_{u \in U^Y_s} \sL^{*\, Y}_{u,\sigma_Y(s)} 
\right) \) are both isomorphic to \( \sL_{\bar{\gamma}(s)} \). Therefore, letting 
\( f^3_s \) be an isomorphism between \( \sL^X_{\bar{\gamma}
(s),\sigma_X(s)} \) and \( \sL^Y_{\bar{\gamma}(s),\sigma_Y(s)} \), we have that 
\[ 
f_s = 
f^0_s \cup f^1_s \cup f^2_s \cup f^3_s 
\] 
is an isomorphism between \( 
X_{\sigma_X(s)} \) and \( Y_{\sigma_Y(s)} \), as required.
\end{proof}

Notice that~\eqref{eqcardinalcondition} implies \( | \chi | = \kappa \), hence we can 
endow \( \pre{\chi}{2} \) with the bounded topology and get a space which is homeomorphic to \( (\pre{\kappa}{2}, \tau_b) \). Recall also
that, as discussed in Section~\ref{sectionspaces}, the collection 
\begin{equation} \label{eqbasis} 
\mathcal{B} = \left\{ \bN_s \mid s \colon d \to 2 \text{ for some } d \in [\chi]^{< \kappa} \right\}
 \end{equation} 
defined  in~\eqref{eqbasic},
form a basis of size \( \kappa \) for the (bounded) topology on \( \pre{\chi}{2} \).

\begin{lemma}\label{lemmapreimage}
Let \( \kappa \) be an inaccessible cardinal. For every open subset \( U \) of \( 
\pre{\chi	}{2} \) there is an \( \L_{\kappa^+ \kappa} \)-sentence \( 
\upvarphi_U \) such that \( h^{-1}(U) = \Mod^\kappa_{\upvarphi_U} \), and 
for every closed \( C \subseteq \pre{\chi}{2} \) there is an \( \L_{\kappa^+ 
\kappa} \)-sentence \( \upvarphi_C \) such that \( h^{-1}(C) = 
\Mod^\kappa_{\upvarphi_C} \).
\end{lemma}

\begin{proof}
Let us consider the first claim of the lemma.
Let \( \mathcal{B} \) be the basis defined in~\eqref{eqbasis}:  since 
\( |\mathcal{B}| = \kappa \), it is enough to prove the claim for \( U \in 
\mathcal{B} \), so let \( d \in [\chi]^{< \kappa} \) and \( s \colon d \to 2 \) 
be such that \( U = \bN_s \). Then let \( A_i = \{ (u,t) \in d \mid s((u,t)) = i 
\} \) for \( i = 0,1 \). Let \( \upvarphi_U \) be the \( \L_{\kappa^+ \kappa} 
\)-sentence
\begin{multline}\tag{\( \upvarphi_U \)}
\Uppsi \wedge \bigwedge_{(u,t) \in A_1} \forall x \left[\mathsf{Seq}_t(x) 
\Rightarrow \exists \langle x_i \mid i \in \sL^*_u \rangle \left
(\mathsf{Lab}^*_u(\langle x,x_i \mid i \in \sL^*_u \rangle)\right)\right] \wedge \\
\bigwedge_{(u,t) \in A_0} \exists x \left[\mathsf{Seq}_t(x) \wedge \neg \exists 
\langle x_i \mid i \in \sL^*_u \rangle \left (\mathsf{Lab}^*_u(\langle x,x_i \mid 
i \in \sL^*_u \rangle)\right)\right]
 \end{multline}
By~\eqref{eqh}, one can easily check that \( h^{-1}(U) = 
\Mod^\kappa_{\upvarphi_U} \).
For the second claim of the lemma, simply let \( \upvarphi_C \) be \( \Uppsi \wedge \neg 
\upvarphi_{\pre{\chi}{2} \setminus C} \).
\end{proof}

\begin{lemma}\label{lemmacontinuous}
Let \( \kappa \) be a weakly compact cardinal, \( R = \proj[T] \) an anlytic 
quasi-order on \( \pre{\kappa}{2} \), and let \( f , h \) be the maps defined, 
respectively, in~\eqref{eqf} and~\eqref{eqh}. Then \( h \circ f \colon 
\pre{\kappa}{2} \to \pre{\chi}{2} \) is continuous.
\end{lemma}

\begin{proof}
It is enough to show that for \( d \in [\chi]^{< \kappa} \) and \( s \colon d \to 
2 \) the set \( (h \circ f)^{-1}(\bN_s) \) is open in \( \pre{\kappa}{2} \). Let 
\( A_i \), \( i = 0,1 \) be defined as in the proof of Lemma~\ref{lemmapreimage}. Notice that if there is \( (u,t) \in A_1 \) such that \( 
\leng{u} \neq \leng{s} \) then \( (h \circ f)^{-1}(\bN_s)  = \emptyset \), 
which is an open set. So let us assume that \( \leng{u} = \leng{t} \) for all \( (u,t) 
\in A_1 \). Since \( \kappa \) is regular and \( | d | < \kappa \), we can 
assume that there is \( \gamma < \kappa \) such that \( \leng{u}, \leng{t} 
\leq \gamma \) for every \( (u,t) \in d \). Let \( B_s \) be the collection of 
those \( r  \in \pre{\gamma}{2}\) such that
\begin{enumerate}[(1)]
\item
for all \( (u,t) \in A_1 \), \( (u, r \restriction \leng{u}, t ) \in T \), and
\item
for all \( (u,t) \in A_0 \) such that \( \leng{u} = \leng{t} \), \( (u, r \restriction \leng{u}, t ) \notin T \).
\end{enumerate}
It is easy to check that by definition of \( f \) and \( h \)
\[ 
(h \circ f )^{-1}(\bN_s) = \bigcup_{r \in B_s} \bN_r, 
\] 
which is therefore an open set in \( \pre{\kappa}{2} \).
\end{proof}

Now we are ready to prove the main result of the paper.

\begin{theorem}\label{theormain}
Let \( \kappa \) be a weakly compact cardinal. Then \( 
\sqsubseteq^\kappa_{\mathsf{TREE}} \) is strongly invariantly universal, i.e.\ 
for every analytic quasi-order \( R \) there is an 
\( \L_{\kappa^+ \kappa} \)-sentence \( \upvarphi \) (all of whose models are trees) such that 
\[ 
R 
\simeq_B {\sqsubseteq \restriction \Mod^\kappa_{\upvarphi}} .
\]
\end{theorem}

\begin{proof}
Let \( R = \proj [T] \) be an analytic quasi-order, and \( f,h\) be the maps 
defined in~\eqref{eqf} and~\eqref{eqh}, respectively. By Lemma~\ref{lemmacontinuous} and Corollary~\ref{corclosed}, \( (h \circ f ) 
``\, \pre{\kappa}{2}  \subseteq \pre{\chi}{2}\) is a closed set, hence by 
Lemma~\ref{lemmapreimage} there is an 
\( \L_{\kappa^+ \kappa} \)-sentence \( \upvarphi \) such that \( h^{-1}((h \circ f)``\, \pre{\kappa}{2}) = 
\Mod^\kappa_\upvarphi \).

\begin{claim}\label{claimsaturation}
\( \Mod^\kappa_\upvarphi \) is the closure under isomorphism of \( f``\, \pre{\kappa}{2}  = \range{f}\).
\end{claim}

\begin{proof}[Proof of the Claim]
Let \( X \in \Mod^\kappa_\upvarphi = h^{-1}((h \circ f)``\, \pre{\kappa}{2}) \subseteq \Mod^\kappa_\Uppsi \), 
and let \( y \in \pre{\kappa}{2} \) be such that \( h(X) = h(f(y)) \): then \( X \cong f(y) \) by Proposition~\ref{propreduction}. 
Conversely, let \( y \in \pre{\kappa}{2} \) and \( X \in \Mod^\kappa_\L \) be such that \( X \cong f(y) \). Since  
\( f(y) \in \Mod^\kappa_\Uppsi \) by
Lemma~\ref{lemmarangef}, then \( X \in \Mod^\kappa_\Uppsi \) as well. But then \( h(X) = h(f(y)) \) by
 Proposition~\ref{propreduction} again, whence \( X \in h^{-1}((h \circ f)``\, \pre{\kappa}{2} = \Mod^\kappa_\upvarphi \).
\end{proof}

As shown in Corollary~\ref{corcompletetree}, the map \( f \) is a witness of \( R \leq_B {\sqsubseteq^\kappa_{\mathsf{TREE}}} \), and hence by Claim~\ref{claimsaturation} it is also a witness of \( R \leq_B {\sqsubseteq \restriction \Mod^\kappa_\upvarphi} \). Let 
\begin{equation}\label{eqg}
g \colon \Mod^\kappa_\upvarphi \to \pre{\kappa}{2}
\end{equation}
be defined by setting \( g(X) = y \iff f(y) \cong X \). Notice that by definition~\eqref{eqf} of \( f\), the strong injectivity~\eqref{eqs_Tinjective} of \( s_T \), and 
Theorem~\ref{theorcomplete}(\ref{theorcomplete2}), the map \( g \) is well-defined. Moreover, for every \( y \in \pre{\kappa}{2} \) and \( X \in \Mod^\kappa_\upvarphi \) we have
\begin{equation} \label{eqpropertiesofg}
g(f(y))  = y \quad \text{and} \quad
f(g(X))  \cong X.
\end{equation}
\begin{claim}\label{claimgBorel}
\( g \) is a Borel map.
\end{claim}

\begin{proof}[Proof of the Claim]
Recall that \( \bN_s \subseteq \pre{\kappa}{2} \) is a clopen set (for every \( s \in \pre{< \kappa}{2} \)). Hence by 
Lemma~\ref{lemmacontinuous}, Corollary~\ref{corclosed}, Lemma~\ref{lemmapreimage}, and 
Theorem~\ref{theorlopezescobar}, we have that \( h^{-1}((h \circ f)``\, \bN_s)  \subseteq \Mod^\kappa_\upvarphi\) 
is a Borel set. Using an argument similar to that of Claim~\ref{claimsaturation}, one can show that
 \( h^{-1}((h \circ f)``\, \bN_s) \) is the closure under isomorphism of \( f`` \, \bN_s \), which implies
 \( g^{-1} (\bN_s )  =  h^{-1}((h \circ f)``\, \bN_s)  \). Thus  \( g^{-1}(\bN_s) \) is a Borel set for every \( s \in \pre{< \kappa}{2} \).
Since the collection of all such \( \bN_s \) form a 
basis of size \( \kappa \) for the (bounded) topology of \( \pre{\kappa}{2} \), this implies that \( g \) is a Borel map.
\end{proof}
Since \( f \) reduces \( R \) to \( \sqsubseteq \restriction \Mod^\kappa_\upvarphi\), by 
Claim~\ref{claimgBorel}  and~\eqref{eqpropertiesofg} one immediately gets that \( g \) is a
 witness of \( {\sqsubseteq \restriction \Mod^\kappa_\upvarphi} \leq_B R \). Moreover,~\eqref{eqpropertiesofg} 
implies that \( g(f(y)) \mathrel{E_R} y \) and \(f(g(X)) \equiv X \) for every \( y \in \pre{\kappa}{2} \) and
 \( X \in \Mod^\kappa_\upvarphi \), hence the factorings of \( f \) and \( g \) through, respectively, \( E_R \) and \( \equiv \) 
are one the inverse of the other. Therefore, \( f \) and \( g \) witness 
\( R \simeq_B {\sqsubseteq \restriction \Mod^\kappa_\upvarphi} \).
\end{proof}

Using Remark~\ref{rembiinterpretation} we immediately get also the following corollary.

\begin{corollary} \label{coruniversalgraph}
Let \( \kappa \) be a weakly compact cardinal. Then \( \sqsubseteq^\kappa_{\mathsf{GRAPH}} \)  is strongly invariantly universal for analytic quasi-orders.
\end{corollary}

\begin{remark}
Notice that the maps \( f \) and \( g \) considered in the proof of Theorem~\ref{theormain} simultaneously witness \( R \simeq_B {\sqsubseteq \restriction \Mod^\kappa_\upvarphi} \) and \( {=} \simeq_B {\cong \restriction \Mod^\kappa_\upvarphi} \), and that the 
\( \L_{\kappa^+ \kappa} \)-sentence \( \Uppsi \) of Definition~\ref{defPsi} is such that if an 
\( \L \)-structure \( X \) satisfies \( \Uppsi \) then \( |X| = \kappa \). 

Moreover, the fact that the cardinal \( \kappa \) under consideration is weakly compact is used in an essential way twice in the proof 
of the main result of this paper (namely, of Theorem~\ref{theormain}): the first occurrence (in the form of ``inaccessibility 
plus tree property'') is in the proof of Lemma~\ref{lemmanormalform}, the second one (in the form of ``\( \kappa \)-compactness 
of the space \( \pre{\kappa}{2} \)'') is in the proof of Theorem~\ref{theormain}, when Corollary~\ref{corclosed} is used. 
This shows that the present argument cannot be directly applied to get a similar result for non weakly compact
 cardinals --- see the questions in Section~\ref{sectionquestions}.
\end{remark}

\begin{remark}\label{remcountableuncountable}
Comparing Theorem~\ref{theormain} (and its proof) with the analogous results concerning countable structures obtained in~\cite{frimot} and~\cite{cammarmot}, one can observe what follows:
\begin{enumerate}[(1)]
\item
Both in the countable and in the uncountable case, the structures used to show 
the (strong) invariant universality of the embeddability relation are trees. However, in 
the countable case (see~\cite{frimot, cammarmot}) were used 
\emph{combinatorial trees} (i.e.\ acyclic connected graph), while in the 
present paper we used \emph{(generalized) trees}. It would be interesting to 
understand whether \( \sqsubseteq^\omega_{\mathsf{TREE}} \) is invariantly 
universal\footnote{Notice that one cannot just consider DST-trees \( T \subseteq \pre{< \omega}{\omega} \) because by~\cite{nas} 
 the embeddability relation on 
such trees is a bqo, and hence cannot be even complete.} for 
analytic quasi-orders on \( \pre{\omega}{2} \), and whether  \( \sqsubseteq \) 
on combinatorial trees of size \( \kappa \) (for \( \kappa \) a weakly compact 
cardinal or, more generally, an uncountable cardinal satisfying \( \kappa^{< \kappa} = \kappa \)) is invariantly universal\footnote{The class of 
combinatorial trees of uncountable size (for relatively small \( \kappa \)'s) 
has also been considered in~\cite{andmot}.} for analytic quasi-orders on 
\( \pre{\kappa}{2} \).
\item
The technique used in~\cite{frimot,cammarmot} for the countable case cannot 
be used in the uncountable context because in general we do not have a good 
descriptive set theory on \( \pre{\kappa}{2} \) when \( \kappa > \omega \). 
For example, in the uncountable case Luzin's separation theorem does not hold 
(see e.g.\ \cite[Theorem 18]{frihytkul}), and similarly the fact that injective 
Borel images of Borel sets are Borel (which was crucially used in~\cite{frimot,cammarmot}) is no more true. 
Conversely, the technique used in this paper cannot be used for the countable 
case because, as far as we can see, to produce an \( \L_{\kappa^+ \kappa} \)-sentence like our \( \Uppsi \)  
(see Definition~\ref{defPsi}) one needs to be able to express the well-foundness of 
some parts of the ordering relation of the structure (or other equivalent 
properties), and this cannot be done in the infinitary logic \( \L_{\omega_1 
\omega} \).
\end{enumerate}
\end{remark}

\section{Open problems}\label{sectionquestions}

In this section we collect some questions related to the material presented in this paper. We divide them in two subsections, one related to completeness and the other one to invariant universality.

\subsection{Completeness}

The first natural question is of course to ask if it is possible to relax our condition on \( \kappa \) in Corollaries~\ref{corcompletetree} and~\ref{corcompletegraph}.

\begin{question}
Let \( \kappa \) be an uncountable cardinal satisfying \( \kappa^{< \kappa} = \kappa \). Is one of \( \sqsubseteq^\kappa_{\mathsf{TREE}} \), \( \sqsubseteq^\kappa_{\mathsf{GRAPH}}\) complete for analytic quasi-orders? 
\end{question}

A problem related to this question is to investigate the possibility of finding counterexamples.

\begin{question}
Is it consistent to have an uncountable cardinal \( \kappa \) and a countable relational language \( \L \) such that \( \sqsubseteq \restriction \Mod^\kappa_\L \) is \emph{not} complete for analytic 
quasi-orders? What about \( \sqsubseteq^\kappa_{\mathsf{GRAPH}} \)?
\end{question}

As discussed in Remark~\ref{rembaumgartner}, Corollary~\ref{corcompletetree} partially extends a result from Baumgartner's paper~\cite{bau}: however, Baumgartner's result uses linear orders, while in this paper we considered a much more complicated kind of structure, namely (generalized) trees. Therefore it is natural to ask the following:

\begin{question}
Given a weakly compact cardinal \( \kappa \), is \( \sqsubseteq^\kappa_{\mathsf{LO}} \) complete for 
analytic quasi-orders? What for arbitrary regular cardinals  \( \kappa \)?
\end{question}	

A possible approach to solve this problem is to first answer the following question.

\begin{question} \label{questionNSTAT}
Given a weakly compact cardinal \( \kappa \), is the analytic quasi-order \(  \subseteq^{\mathsf{NSTAT}}  \) (on the whole \( \pre{\kappa}{2} \)) complete?
\end{question}

\subsection{Invariant universality}

The first questions are again about the possibility of removing from Theorem~\ref{theormain} 
and Corollary~\ref{coruniversalgraph} the assumption that \( \kappa \) be weakly compact.

\begin{question}
For which uncountable cardinals \( \kappa \) satisfying \( \kappa^{< \kappa}  = \kappa \) the quasi-orders \( \sqsubseteq^\kappa_{\mathsf{TREE}} \) and \( \sqsubseteq^\kappa_{\mathsf{GRAPH}} \) are (strongly) invariantly universal?

Is it consistent to have an uncountable cardinal \( \kappa \) and a countable relational language \( \L \) such that \( \sqsubseteq \restriction \Mod^\kappa_\L \) is \emph{not} (strongly) invariantly universal? What about \( \sqsubseteq^\kappa_{\mathsf{GRAPH}} \)?
\end{question}

Another interesting open problem concerns the possibility of distinguishing the notions of completeness, invariant universality, 
and strong invariant universality with suitable embeddability relations.

\begin{question}
Is it consistent to have an infinite cardinal \( \kappa \), a countable relational language \( \L \), and two \( \L_{\kappa^+ \kappa} \)-sentences \( \upvarphi_0 , \upvarphi_1\) such that \( \sqsubseteq \restriction \Mod^\kappa_{\upvarphi_0} \) is complete but not invariantly universal, and \( \sqsubseteq \restriction \Mod^\kappa_{\upvarphi_1} \) is invariantly universal but not strongly invariantly universal?
\end{question}

\noindent
Notice that this last question remains unanswered also for \( \kappa  = \omega\) (see \cite[Question 6.3]{cammarmot}).

\end{document}